%% file: main.tex
\documentclass{article}
\usepackage[utf8]{inputenc}

\title{Continuum asymptotics for tree growth models\\ with uniform backward dynamics}
\author{David Geldbach\footnote{Department of Statistics, University of Oxford; {david.geldbach@stats.ox.ac.uk / david.geldbach@gmail.com}}}
\date{\today}

\input{configurations.tex}

\begin{document}

% title
\maketitle

\begin{abstract}
    We study (plane) tree-valued Markov chains $(T_n,n \geq 1)$ with uniform backward dynamics and show that they can be obtained by sampling from a real tree.
    As non--plane trees, every such Markov chain is represented by a weighted real tree. We equip this real tree with a planar order as well as some extra functions for the full representation theorem. 
    We also show that under an inhomogeneous rescaling after trimming leaves $(T_n, n\geq 1)$ converges to a random real tree in the Gromov--Prokhorov metric. This makes use of a special class of real trees, interval partition trees, which were introduced by Forman (2020). Moreover, this generalises and sheds some new light on work by Evans, Grübel and Wakolbinger (2017) on the binary special case.
\end{abstract}

%%%%%%%%%%%%%%%% INTRODUCTION V2 %%%%%%%%%%%%%%%%
\section{Introduction}

This paper studies a large family of tree--valued Markov chains with the property that the trees grow in time. The goal is to classify these Markov chains and to establish metric space scaling limits.
Let $\widetilde{\TT}$ be the space of rooted plane trees.
The tree growth chains that we consider take values in the space 
\begin{equation}\label{eq:def:planetrees}
    \TT_n = \left\{ T \in \widetilde{\TT}: T \text{ has $n$ leaves; for all $x \in T$}: \deg(x) \neq 2; \deg(\text{root})=1 \right\}, \quad n \geq 1.
\end{equation}
Leaves are non-root vertices of degree $1$. Trees with the condition $\deg(\text{root})=1$ are sometimes called \emph{planted} trees. Plane trees can be formally constructed using the Ulam--Harris encoding as subsets of $\bigcup_{k \geq 0} \NN^k_0$, the set of words with the alphabet $\NN_0$. We call $T\subset \bigcup_{k \geq 0} \NN^k_0$ a rooted plane tree if $\emptyset \in T$, if $(x_1,\ldots,x_k)\in T$ implies $(x_1,\ldots,x_{k-1}) \in T$ whenever $k\geq 1$, and if $(x_1,\ldots,x_k)\in T$ implies $(x_1,\ldots,x_k-1) \in T$ whenever $x_k \geq 1$. The empty word, $\emptyset$, is thought of as the root of $T$. Sometimes we consider labelled trees, denote by $\widetilde{\TT}^{\mathrm{labelled}}$ the set of plane trees where some (not necessarily all) vertices are labelled by $\{1,\ldots,n\}$ for some $n$.

\begin{definition}\label{def:uniform_backwards} 
We call $(T_n, n\geq 1)$ a \emph{tree growth chain with uniform backward dynamics} if $(T_n, n\geq 1)$ is a Markov chain such that for all $n\geq 1$ we have $T_n \in \TT_n$ and such that the backward dynamics are uniform in the following sense: given $n\geq 2$ and $T_n$,  the following procedure yields a tree with the same distribution as $T_{n-1}$.
\begin{enumerate}
    \item Choose a leaf of $T_n$ uniformly,
    \item remove this leaf and its associated edge,
    \item remove any vertex of degree $2$ and replace it and its two adjacent edges by a single edge.
\end{enumerate}
\end{definition}

\begin{figure}[hbt]
    \centering
    \includegraphics[scale = 1]{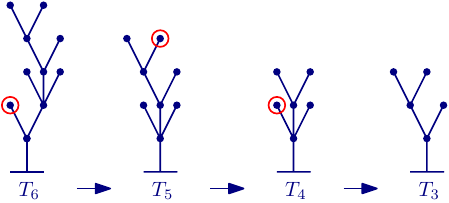}
    \caption{An example for the uniform backward dynamics. The red circle indicates which leaf has been uniformly chosen.}
    \label{fig:backwards_dynamics}
\end{figure}

Note that $\TT_1$ only has one element which is a tree consisting of two vertices, one thought of as the root and one thought of as a leaf -- the assumption $T_1\in \TT_1$ thus serves as a (deterministic) initial condition. Clearly the backward dynamics applied to $T_n \in \TT_n$ yield an element of $\TT_{n-1}$, see also Figure \ref{fig:backwards_dynamics} for an illustration. We do not make any further assumptions on the forward dynamics. 

The goal of this article is to classify all tree growth chains with uniform backward dynamics. In particular, we want to understand the different possible forward dynamics. We then use the classification of forward dynamics to show a scaling limit for $(T_n,n\geq 1)$ represented as a sequence of metric spaces. Characterising the forward dynamics is also the object of Evans, Grübel and Wakolbinger \cite{evans_doob-martin_2017} who study the special case of binary trees. Their work forms a basis for a lot of the ideas in this article, and we comment on the relationship to it in Remark \ref{remark:what_evans_grbel_wknbger_do}.

A key notion in the classification is that of real trees.

\begin{definition} [Real tree]
A real tree ($\RR$--tree) is a complete, separable metric space $(\bfT,d)$ with the property that for each $x,y \in \bfT$, there is a unique non--self--intersecting path from $x$ to $y$, denoted by $\llbracket x,y \rrbracket$, and furthermore this path is isometric to a closed real interval. We call $(\bfT,d,r)$ \emph{rooted} if we have a marked point $r\in \bfT$. We call $(\bfT,d,\mu)$ \emph{weighted} if we have a probability measure $\mu$ defined on the Borel $\sigma$--algebra of $\bfT$.
\end{definition}

Unless specified otherwise, all our real trees are weighted and rooted, and we often abbreviate $(\bfT, d ,r, \mu)$ to $\bfT$. The real trees that we consider in this paper are not necessarily (locally) compact. 

We consider two notions of subtrees for real trees. Given $x\in \bfT$, the \emph{fringe subtree} of $x$ is
\begin{equation}\label{def:fringe_subtree}
    F_\bfT(x)=\big\{y \in \bfT: x \in \llbracket r,y \rrbracket \big\},
\end{equation}
which we root at $x$. On the other hand, the \emph{subtrees of} $x$ are the connected components of $F_\bfT(x) \backslash \{x\}$. To each of them we add $x$ and root them at $x$. We define the \emph{degree} of $x \in \bfT$ as the number of connected components of $\bfT \backslash \{x \}$, which might be infinite. We call $x$ a leaf if $\deg(x)=1$ and a branchpoint if $\deg(x) \geq 3$ (we do not call the root $r$ a leaf even if $\deg(r)=1$). Given $x_1,\ldots, x_n \in \bfT$, the span of $x_1,\ldots, x_n$, denoted by $\spn(x_1,\ldots, x_n)$, is the minimal connected subset of $\bfT$ containing $x_1,\ldots, x_n$ as well as the root. In particular, $\spn(x)=\llbracket r,x \rrbracket$. We denote the ancestral partial order on $\bfT$ by $\preceq$: this means that $x \preceq y$ if and only if $x\in \llbracket r,y \rrbracket$.

We consider a special class of real trees. Denote by $\supp(\mu)$ the closed support of $\mu$.

\begin{definition}[IP--tree]\label{def:ip-tree}
A rooted, weighted real tree $(\bfT, d, r, \mu)$ is an \textit{interval--partition tree} if it possesses the following properties.
\begin{enumerate}
    \item \textit{Spanning.} Every leaf is in the support of $\mu$, i.e.\ $\bfT = \spn(\supp(\mu))$.
    
    \item \textit{Spacing.} For $x\in \bfT$, if $x$ is either a branchpoint or lies in the support of $\mu$, then
    \begin{equation}\label{eq:IP_spacing}
        d(r,x)+\mu(F_{\bfT}(x)) = 1.
    \end{equation}
\end{enumerate}
\end{definition}

IP--trees were introduced by Forman \cite{forman_exchangeable_2020} and arise as natural representatives for equivalence classes under mass--structural isomorphisms. Informally speaking, a map between two weighted, rooted real trees is a mass--structural isomorphism if it preserves the measure and the tree structure but not the metric. This leads to some freedom in the choice of metric which is then determined by the spacing property. The name \emph{interval partition} tree originates from a method of constructing IP--trees using the so--called \emph{bead--crushing} construction which we do not use here. 

\begin{figure}[tbh]
    \centering
    \includegraphics[width=0.9\linewidth]{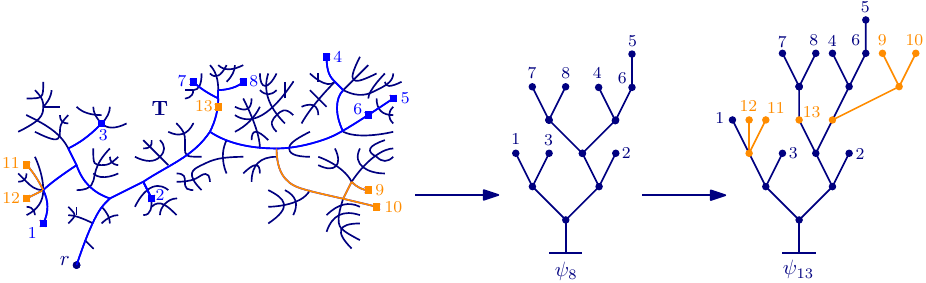}
    \caption{An example for $\psi$. Note that $\psi_8(x_1,\ldots,x_8)$ embeds into $\psi_{13}(x_1,\ldots,x_{13})$ respecting the labels.}
    \label{fig:psi_example}
\end{figure}

Call $a\in \bfT$ an atom if $\mu(\{a\})>0$. We can decompose $\mu = \mu_{\mathrm{atoms}} + \mu_s + \mu_\ell $ where $\mu_{\mathrm{atoms}}$ is supported on the atoms of $\mu$, $\mu_s$ is supported diffusely on $\bfT \backslash \{ \text{leaves} \}$ and $\mu_\ell$ is supported diffusely on the leaves of $\bfT$. We can choose the not necessarily closed supports of $\mu_{\mathrm{atoms}}, \mu_s, \mu_\ell$ to be disjoint. Further, we can choose $\supp(\mu_s)$ in such a way that for every $x\in  \supp(\mu_s)$ we have $\deg(x)=2$. This is possible because there are at most countably many branchpoints since $\bfT$ is separable.

\begin{definition}\label{def:planar}
We call $\psi=\{\psi_n,n\geq 2\}$, a family of measurable maps $\psi_n:\bfT^n\rightarrow \widetilde{\TT}^{\mathrm{labelled}}$, a planar order for $(\bfT,d,r)$ if it satisfies the following properties:
\begin{enumerate}
    \item As a rooted unlabelled, non--plane tree, $\psi_n(x_1,\ldots,x_n)$ has the same tree structure as $\spn(x_1, \ldots,x_n)$. Let $\tilde{r}_n$ be its root. If $\deg(\tilde{r}_n)>1$, add a new vertex $r_n$, connect it to $\tilde{r}_n$ and call it the root of $\psi_n(x_1,\ldots,x_n)$. If $\deg(\tilde{r}_n)=1$, set $r_n = \tilde{r}_n$.

    \item The vertex of $\psi_n(x_1,\ldots,x_n)$ corresponding to $x_i$ is labelled $i$.
    
    \item $\psi $ is consistent in the sense
    that for every $n,m \in \NN$ and every $x_1,\ldots,x_{n+m} \in \bfT$ we have that $\psi_n (x_1,\ldots,x_n)$ embeds into $\psi_{n+m}(x_1,\ldots,x_{n+m})$ respecting the planar order and labels.
\end{enumerate}
\end{definition}

Observe that $\psi_n(x_1,\ldots,x_n)$ is a partially labelled, planted, rooted plane tree with labels $\{1,\ldots,n\}$. Every leaf is labelled, some vertices might carry multiple labels (if $x_i=x_j$ say), and a vertex in $\psi_n(x_1,\ldots,x_n)$ has degree $2$ only if it is labelled. See Figure \ref{fig:psi_example} for an illustration. 

We can now state our main theorem. Note that the set of probability distributions of tree growth chains with uniform backward dynamics is a convex set. A distribution is called \emph{extremal} if it cannot be written as a non--trivial convex combination of two other distributions. Call a tree growth chain extremal if its distribution is extremal in this sense. 

\begin{theorem}[Classification]\label{thm:main_thm_new_version}
    Given an extremal tree growth chain $(T_n, n\geq 1)$ with uniform backward dynamics, its distribution is determined by 
    \begin{enumerate}
        \item an IP--tree $(\bfT, d, r, \mu)$,
        \item a planar order $\psi$ for $\bfT$,
        \item a function $\lambda:\bfT \to [0,1]$ called branch weight function,
        \item for every atom $a$ of $\mu$, a function $\beta_a:\{\text{subtrees of $a$}\} \to [0,1]$ called branchpoint weight function.
    \end{enumerate}
    Specifically, we can construct the distribution $\rho_\bfT$ of $(T_n, n\geq 1)$, see Construction \ref{construction:sampling_the_bridge} below.
    Here $\bfT$ is an abbreviation for $(\bfT,d,r,\mu,\psi,\lambda,\{\beta_a\}_a)$.
    These objects are unique in the following sense
    \begin{enumerate}
        \item $(\bfT, d, r, \mu)$ is unique up to isometries that preserve $\mu$ and $r$,
        \item $\psi$ is unique given a representative of the isometry class of $(\bfT, d, r, \mu)$,
        \item $\lambda$ is unique as element of $L^1(\mu_s)$, 
        \item for every atom $a$, $\beta_a$ is unique if we impose that $\beta_a(\bfS_1) \leq \beta_a(\bfS_2)$ for any two subtrees $\bfS_1$ and $\bfS_2$ of $a$ where $\bfS_1$ is to the left of $\bfS_2$, i.e.\ where given any $x_1 \in \bfS_1 \backslash\{a\}$ and $x_2 \in \bfS_2 \backslash \{a\}$, the leaf labelled $1$ is to the left of the leaf labelled $2$ in $\psi_2(x_1,x_2)$.
    \end{enumerate}

    \begin{remark}
        This classification simplifies to $(\bfT, d, r, \mu, \psi)$ if we can a priori determine that $\mu$ is supported diffusely on the leaves of $\bfT$. This is the case for many examples -- see in Examples \ref{example:Marchal} and \ref{example:patricia}.
    \end{remark}
\end{theorem}

Before we construct $\rho_\bfT$, we remark that
this theorem is very similar in spirit to a long list of theorems that seek to classify exchangeable random objects. The most classical one is de Finetti's theorem that states that the distribution of every sequence of exchangeable real random variables is a mixture of the distributions of sequences of $i.i.d.$ random variables. Another notable one is Kingman's paintbox theorem that describes every exchangeable partition of $\NN$ as a mixture of paintboxes. 
The article \cite{forman_representation_2018} by Forman, Haulk, Pitman also classifies a family of exchangeable objects, hierarchies in their case, by sampling from real trees and the work \cite{foutel-rodier_exchangeable_2021} of Foutel-Rodier, Lambert, Schertzer classifies various exchangeable objects via combs which are tree--like as well. Gerstenberg \cite{gerstenberg_exchangeable_2020} classifies exchangeable interval hypergraphs, trees are a special case here, by sampling from a random subset of $[0,1]^2$. See Kallenberg \cite{kallenberg_foundations_2021} for the classical theorems, and \cite{forman_representation_2018} for a non--exhaustive list of references to similar, modern theorems. 

\begin{construction} \label{construction:sampling_the_bridge}
Assume we are given $(\bfT,d,r,\mu,\psi,\lambda,\{\beta_a\}_a)$, we construct $\rho_\bfT$ for Theorem \ref{thm:main_thm_new_version} as the distribution of  $(T_n,n\geq1)$. To construct $(T_n,n\geq 1)$ with $T_n \in \TT_n$, we proceed as follows:
\begin{enumerate}
    \item Sample $(\xi_i,i\geq 1)\in \bfT$ $i.i.d.$ from $\mu$ and $(U_i,i\geq 1)$ independent uniform random variables on $[0,1]$.

    \item Set $S_n = \psi_n(\xi_1,\ldots, \xi_n)$.

    \item To construct $T_n$, we add leaves to $S_n$: for each vertex $x\in S_n$ which is labelled but not a leaf we add as many leaves to $x$ as there are labels at $x$. To each leaf $y\in S_n$ that has more than one label, we add as many leaves to $y$ as there are labels at $y$ (i.e.\ we add a new vertex and edge connected to $x$). See Figure \ref{fig:sampling_construction} for an illustration. We need to determine the planar order of the new leaves, for the leaf labelled $i$:

    \begin{enumerate}
        \item If $\xi_i \in \supp(\mu_s)$: then $\deg(\xi_i)=2$, which means that almost surely the leaf $i$ is attached to a vertex of degree $2$ in $S_n$. If $U_i \leq \lambda(\xi_i)$, we orient the leaf labelled $i$ to the left of the subtree of the vertex that was labelled $i$ in $S_n$ and otherwise to the right.

        \item If $\xi_i = a$ for an atom $a$: assume the leaf labelled $i$ is attached to the vertex $x\in S_n$. The subtrees of $x$ in $S_n$ naturally correspond to subtrees of $a$ in $\bfT$. We orient the leaf $i$ to the left of a subtree of $x$ in $S_n$ if $U_i \leq \beta_a(\bfS)$ (respectively to the right if $U_i > \beta_a(\bfS)$) for every subtree $\bfS$ of $a$. See Figure \ref{fig:where_are_the_leaves} for an illustration.

        \item If $\xi_i=\xi_j=a$ for $i\neq j$ then we attach two leaves to the same branchpoint. If $U_i < U_j$, we orient the leaf $i$ to the left of leaf $j$ and vice versa.
        
    \end{enumerate}

    \item Delete the labels to obtain $T_n \in\TT_n$. Doing this construction simultaneously for all $n$ yields $(T_n , n\geq 1)$.
\end{enumerate}
\end{construction}

\begin{figure}[ht]
    \centering
    \includegraphics[scale=1]{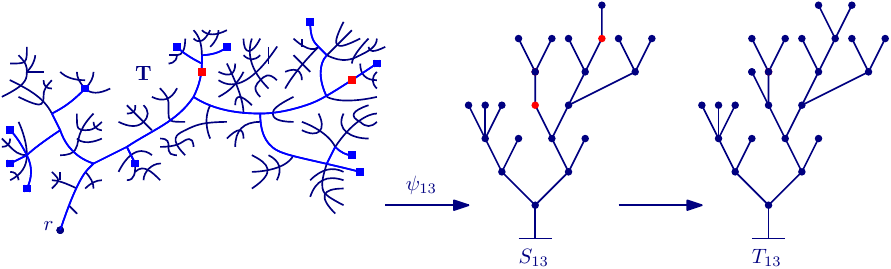}
    \caption{An example for sampling $T_{13}$. First, we sample $(\xi_i, i\leq 13)$ from $\bfT$; secondly, use $\psi_{13}$ to map $\spn(\xi_1, \ldots,\xi_{13})$ to a plane tree; and thirdly we add leaves to some interior vertices (in red) to obtain $T_{13}$. In the last step, the planar order is determined by $\lambda$ and $\{\beta_a\}_a$ as well as some additional randomness.}
    \label{fig:sampling_construction}
\end{figure}

\begin{figure}[ht]
    \centering
    \includegraphics[scale=1]{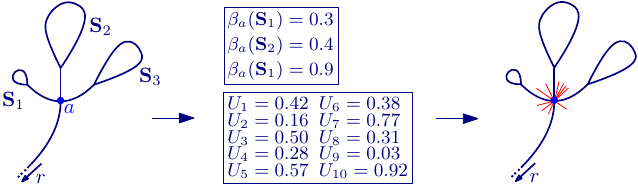}
    \caption{An illustration of the use of branchpoint weight function $\beta_a$ in Construction \ref{construction:sampling_the_bridge} where $\xi_1=\ldots=\xi_{10} =a$. Here $\beta_a(\bfS_i)$ is the probability of a leaf attached to $a$ to be to the left of $\bfS_i$ (by an abuse of notation $\bfS_i$ is a subtree of $a$ in $\bfT$ as well as a subtree of $S_n$ of the vertex corresponding to $a$).}
    \label{fig:where_are_the_leaves}
\end{figure}

\begin{remark}
    This results in the sequence $(T_n, n\geq 1)$ being a tree growth chain with uniform backward dynamics. Indeed, this is true because a backward step corresponds to removing $\xi_n$ from the construction. Once the labels are removed, this corresponds to uniformly choosing a point from $\xi_1, \ldots, \xi_n $ which in turn means choosing a leaf of $T_n$ uniformly.
\end{remark}

\begin{remark}
    Both in Theorem \ref{thm:main_thm_new_version} and in Construction \ref{construction:sampling_the_bridge} we could replace the planar order $\psi$ by a partial order $<$ on $\bfT$. Here we require that $<$ is a total order on the leaves, for $x,y\in \bfT$ which are not leaves we require that $x<y$ if and only if $\ell_x < \ell_y$ for all leaves $\ell_x,\ell_y$ with $x \prec \ell_x$ and $y \prec \ell_y$. We then recover $\psi_n$ by using $<$ to turn $\spn(x_1,\ldots,x_n)$ into a plane tree. 
\end{remark}

Theorem \ref{thm:main_thm_new_version} deals with extremal tree growth chains. Given a tree growth chain which is not extremal, we can decompose it into extremal tree growth chains. 

\begin{cor} \label{cor:decomposition}
    For every tree growth chain $(T_n, n\geq 1)$ with uniform backward dynamics there exists a unique probability measure $\nu$ on a space of probability measures $\mathcal{M}_1$ such that 
    \begin{equation}
        \PP \left( (T_n, n\geq 1)\in \cdot \right) = \int_{\mathcal{M}_1} \rho \left( (T_n, n\geq 1)\in \cdot \right) \nu (d\rho),
    \end{equation}
    where $\nu$--almost every $\rho$ is extremal, and there is $\bfT=(\bfT,d,r,\mu, \psi, \lambda, \{\beta_a\}_a)$ such that $\rho=\rho_\bfT$.
\end{cor}

We can now give some examples of tree growth chains. We start with Marchal's tree growth algorithm, introduced by Marchal \cite{marchal_note_2008}, which generalises Rémy's tree growth algorithm \cite{remy_procede_1985} which only grows binary trees (here $\alpha = 2$). 

\begin{example}[Marchal's tree growth algorithm] \label{example:Marchal}
 Let $\alpha \in (1,2]$, we construct $(T_n^M(\alpha), n\geq 1)$ recursively. By our general assumptions on tree growth chains, $T_1^M(\alpha)$ is given by a tree with two vertices and one edge. For $n\geq 1$, construct $T_{n+1}^M(\alpha)$ from $T_n^M(\alpha)$ as follows:
\begin{enumerate}
    \item Assign weight $\alpha-1$ to each edge of $T_n^M(\alpha)$ and weight $k-1-\alpha$ to each branchpoint of $T_n^M(\alpha)$ with degree $k\geq 3$. Choose an edge or a branchpoint according to these weights. 
    \item If an edge $e$ has been chosen, split it into two edges $e_1,e_2$ and attach a new leaf to the new vertex. For the new planar order, orient the new leaf to the left or right with probability $1/2$ independently from the previous steps.
     
    \item If a branchpoint $v$ has been chosen, attach a new leaf to this branchpoint. For the new planar order, orient the new leaf uniformly among the existing other children of the branchpoint. 
\end{enumerate}
This is a well studied model. The fact that the backward dynamics of Marchal's tree growth algorithm are uniform goes back to Haas, Miermont, Pitman and Winkel \cite{haas_continuum_2008}. There is a well known connection between Marchal's tree growth algorithm and $(\mathcal{T}_\alpha, 1<\alpha\leq 2)$, stable real trees. $(\mathcal{T}_\alpha, d,r,\mu_\alpha)$ is a random weighted, rooted real tree. Stable trees were introduced by Duquesne and Le Gall \cite{duquesne_random_2002} with $\mathcal{T}_2$ being the Brownian continuum random tree introduced by Aldous \cite{aldous_continuum_1993, aldous_continuum_1991}. 

Because $\mathcal{T}_\alpha$ is random, the distribution of Marchal's tree growth is not extremal, for this we need to condition on a realisation of $(\mathcal{T}_\alpha,d,r,\mu_\alpha)$. Then $(\bfT, d, r, \mu)$ is given by the unique mass--structural equivalence class of $(\mathcal{T}_\alpha, d,r,\mu_\alpha)$ which is an IP--tree, using \cite[Theorem 1]{forman_exchangeable_2020}. Because $\mu_\alpha$ is almost surely supported diffusely on the leaves, $\lambda$ and $\{\beta_a\}_a$ are almost surely trivial. Theorem \ref{thm:main_thm_new_version} then provides the existence of the planar order $\psi$. Alternatively, it can be obtained using encoding functions, see Duquesne \cite{duquesne_coding_2006}. More general background on encoding functions for real trees can be found in Evans \cite[Example 3.14]{evans_probability_2008}.
\end{example}

\begin{figure}[ht]
    \centering
    \begin{tabular}{c c}
        \includegraphics[scale = 0.22]{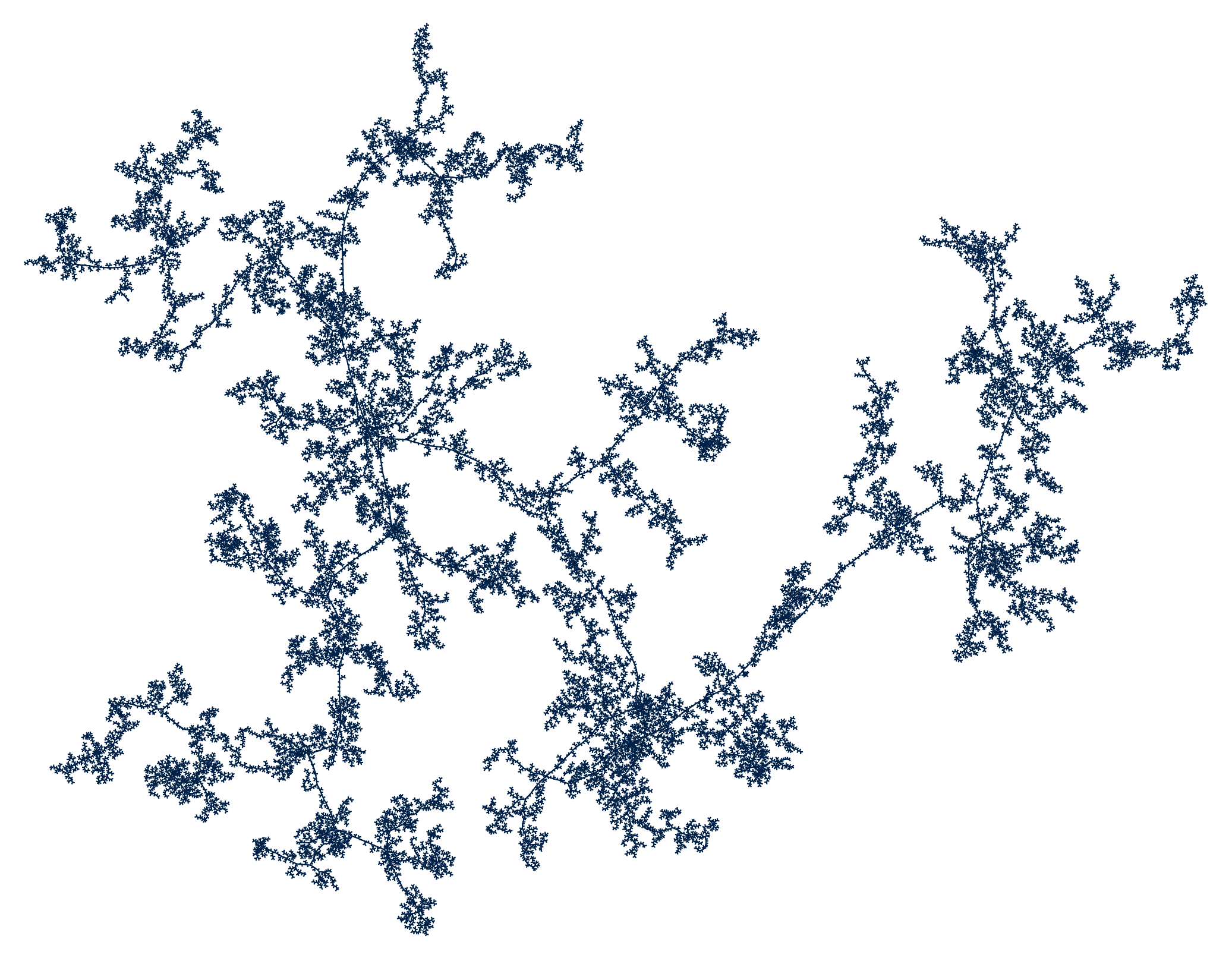} &  
        \includegraphics[scale = 0.22]{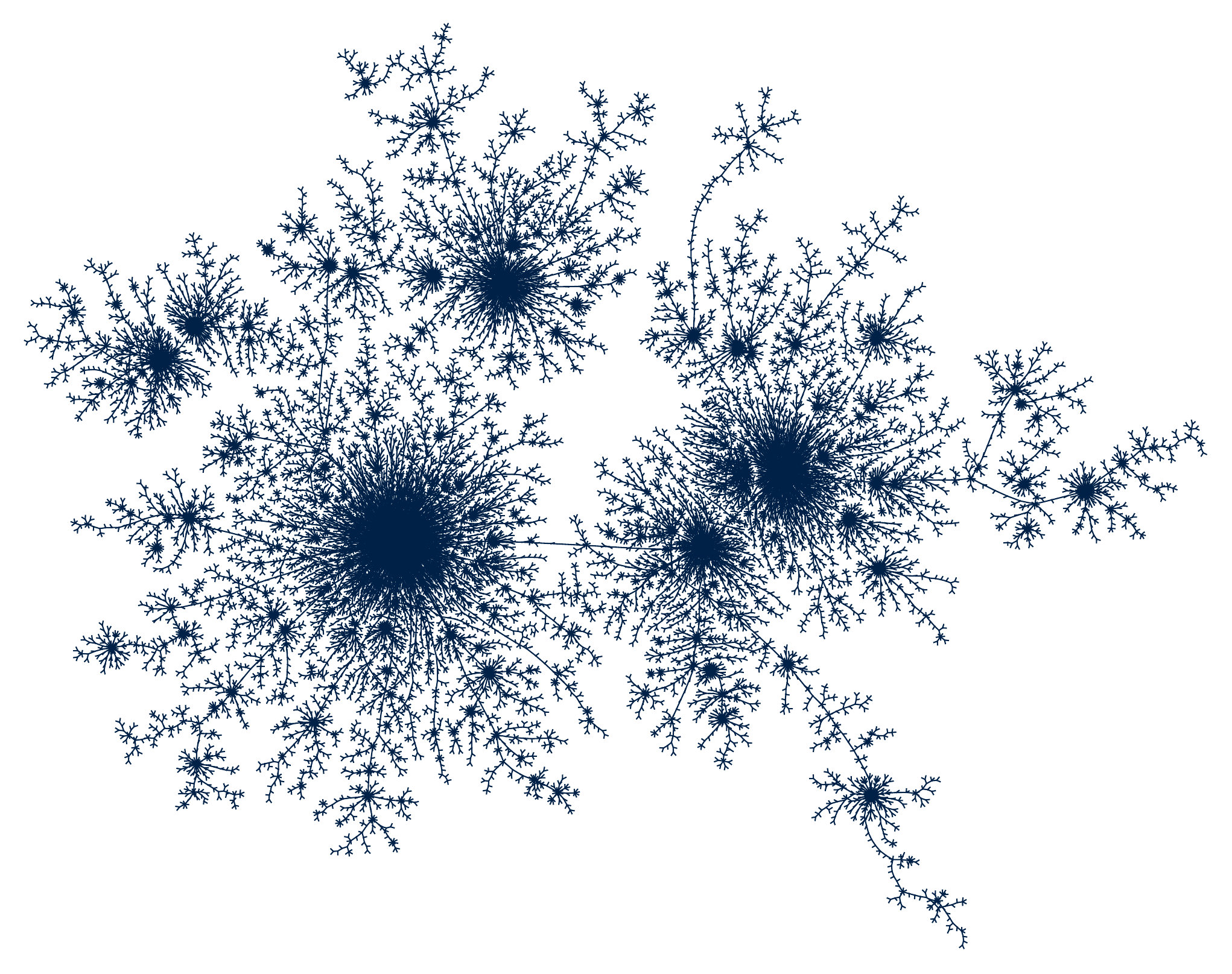}
    \end{tabular}
    \caption{A simulation of Marchal's tree growth after $25000$ growth steps for different values for $\alpha = 2$ (left) and $\alpha = 1.4$ (right) using Mathematica. These trees approximate the Brownian continuum random tree (left) and the $1.4$--stable tree (right).}
    \label{fig:marchal_example}
\end{figure}

\begin{remark}\label{remark:what_evans_grbel_wknbger_do} 
    The example of Marchal's tree growth algorithm is related to the work of Evans, Grübel and Wakolbinger \cite{evans_doob-martin_2017}.
  In their article, the authors study Rémy's tree growth \cite{remy_procede_1985} -- the case of $\alpha=2$ in Marchal's tree growth.
    Their main theorem \cite[Theorem 8.2]{evans_doob-martin_2017} gives a sampling representation for the binary tree growth chains with uniform backward dynamics, and thus corresponds to
    Proposition \ref{prop:evans}.
    The authors phrase their result as a classification of the Doob--Martin boundary which is equivalent to classifying extremal tree growth chains. Note that the Doob--Martin boundary is largely determined by the backward dynamics of a Markov chain. This, combined with the fact that for any $\alpha\in (1,2)$ we have
    for every $T\in \TT_n$
    \begin{equation*}
         \PP\left(T_n^M(\alpha) = T\right) >0,
    \end{equation*}
    implies that the Doob--Martin boundary is the same for $(T^M_n(\alpha),n\geq 1)$ for any parameter $\alpha\in (1,2)$ for Marchal's tree growth algorithm.     
    This article extends their work as our framework allows for multi--furcating trees instead of binary trees just like $\alpha$--stable trees extend the Brownian continuum random tree or like Marchal's tree growth extends Rémy's tree growth. Further, even in the case $\alpha=2$ we believe that our variation of their construction is more descriptive by identifying a unique representation.
\end{remark}

\begin{figure}[tbht]
    \centering
    % it was scale = 1.5 but that puts it at the end of the doc
    \includegraphics[scale=1.2]{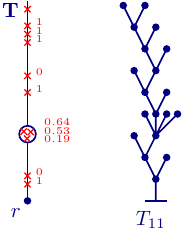}
    \caption{An example for $T_{11}$ if $\bfT = [0,1]$. The red crosses stand for $(\xi_i, i\leq 11)$. The red number next to them states the outcome of the coin-flip that determines if the leaf is attached left or right of the spine. The blue circle indicates an atom $a$. Here, we have $\beta_a([a,1])=1/2$ which means that if $U_i>1/2$ (in red), we orient the corresponding leaf to the right and to the left otherwise.}
    \label{fig:example:line_tree}
\end{figure}

\begin{example} \label{example:interval}
Assume Theorem \ref{thm:main_thm_new_version} yields $\bfT = [0,1], r=0$, the usual Euclidean distance $d=\vert \cdot \vert$ and with $\mu,\psi,\lambda, \{\beta_a\}_a$ arbitrary. We observe:
\begin{enumerate}
    \item There is only one valid planar order, $\psi_n(x_1,\ldots,x_n)$ is a line graph of length $k$ (thought of as plane tree) where $k$ is the number of distinct points in $(x_1,\ldots,x_n)$.

    \item For $(\bfT, d,r,\mu)$ to be an IP--tree, the spacing property of Definition \ref{def:ip-tree} requires that $\mu([0,x)) =x$ for all $x\in \supp(\mu)$. The spacing property translates to $1\in \supp(\mu)$.

    \item For any atom $a$, $a$ only has one subtree, namely $[a,1]$. Therefore the function $\beta_a$ can be identified with a single threshold $\beta(a)$.
\end{enumerate}

Let us construct $T_n$ in this example: after sampling $n$ points from $\bfT$ we receive a line graph of length $k\leq n$ where $k$ is the number of distinct points sampled, this line graph is a plane tree. To every vertex of the line graph, except for the two endpoints, we attach one or multiple leaves according to the number of points sampled on $\bfT$. For every leaf corresponding to a point $x\in \supp(\mu_s)$ we flip a coin with parameter $\lambda(x)$ to decide if we attach to the left or to the right of the spine. 
Similarly, for each leaf attached to an atom $a$ we draw a uniform random variable $U$ on $[0,1]$. If $U \leq \beta(a)$ we attach the leaf to the left of the spine and if $U>\beta(a)$ we attach it to the right of the spine.
This results in $(T_n, n\geq 1)$ being a sequence of growing spines with leaves hanging off on the sides, see Figure \ref{fig:example:line_tree} for an illustration.

\end{example}

\begin{example}\label{example:patricia}
    This is the main object of study of \cite{birkner_patricia_2020}:
    Let $\ell\in \NN, \ell\geq 2$ and let $S_n^\ell$ be the $\ell$--ary plane tree of height $n$, here every vertex has $\ell$ offspring. Turn $S_n^\ell$ into a real tree $\bfS_n^\ell$ by assigning intervals of length $2^{-k}$ to the edges at distance $k$ to the root, gluing them at the branchpoints. Let $\bfT^\ell$ be the completion of $\bigcup_{n\geq 1} \bfS_n^\ell$. Consider now any diffuse probability measure $\mu$ on the leaves of $\bfT^\ell$, \cite[Theorem 1.5]{forman_exchangeable_2020} states that there exists a choice of metric $d_\mu$ on $\bfT^\ell$ that renders $(\bfT^\ell, d_\mu, r, \mu)$ an IP--tree. (Note that $\mu$ can be thought of as a distribution on $[0,1]$ by considering $\ell$--adic expansions.) Because $(S_n^\ell, n\geq 1)$ are plane trees, this induces a natural choice of planar order for $(\bfS_n^\ell, n\geq 1)$ which induces maps $(\psi_n, n\geq 1)$ for $\bfT^\ell$.
    
    This corresponding Markov chain is also called the PATRICIA chain, see \cite{birkner_patricia_2020} for a study of this in the case of binary trees. PATRICIA stands for \emph{"practical algorithm to retrieve information coded in alphanumeric"}. Given $z_1,\ldots, z_n \in \{0,\ldots, \ell-1\}^{\infty}$, words of infinite length in the alphabet $\{0,\ldots, \ell-1\}$, we can construct words $y_i,i\leq n,$ of finite length such that $y_i$ is an initial segment of $z_i$ for all $i\leq n$, all $y_i$ are distinct and that $y_1,\ldots, y_n$ are the minimal length words with this property. These $y_1,\ldots, y_n$ form a tree with $n$ leaves, the so--called radix sort tree. 
    Consider now $\mu$ as measure on $\{0,\ldots, \ell-1\}^{\infty}$ and let $Z_1,\ldots, Z_n$ be $i.i.d.$ $\mu$--samples. Then $T_n$ is the radix sort tree corresponding to $Z_1,\ldots, Z_n$. See Figure \ref{fig:example:d_ary_tree} for an illustration.

    \begin{figure}[bht]
        \centering
        \includegraphics[scale=1.0]{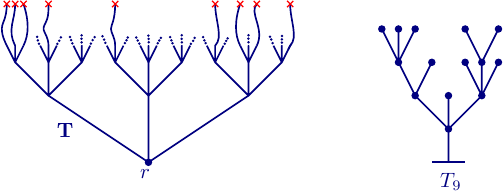}
        \caption{An example for $T_9$ if $\bfT^3$ is the $3$--ary tree. The red crosses stand for $(\xi_i, i\leq 9)$, sampled from the leaves of $\bfT^3$.}
        \label{fig:example:d_ary_tree}
    \end{figure}
\end{example}

Even though there are many examples of tree growth models that satisfy uniform backward dynamics, there are also many models that do no satisfy this assumption. Some examples include preferential attachment trees or random recursive trees. 

We now turn to metric space scaling limits of $(T_n,n\geq 1)$. This can be motivated by Marchal's tree growth algorithm, see Example \ref{example:Marchal}. Consider Marchal's tree growth algorithm $(T^M_n(\alpha), n \geq 1)$ and the $\alpha$--stable trees $(\mathcal{T}_\alpha,d,r,\mu_\alpha)$. We now think of $T^M_n(\alpha)$ as a metric space using the graph distance $d_n$ as metric. We also equip $T^M_n(\alpha)$ with a probability measure $\mu_n$ which is uniform on the vertices. Curien and Haas \cite[Theorem 5]{curien_stable_2013} show that 
\begin{equation}\label{eq:scaling_limit_marchal}
    \left(T^M_n(\alpha),n^{-1+1/\alpha}d_n,r_n,\mu_n\right) \xrightarrow[n \to \infty ]{\PP-a.s.} \left( \mathcal{T}_\alpha, d, r, \mu_\alpha \right),
\end{equation}
as metric spaces in the Gromov--Hausdorff--Prokhorov topology. Related statements are \cite[Corollary 24]{haas_continuum_2008}, \cite[Theorem 3.2]{marchal_note_2008} and \cite[Theorem 3.3.3]{duquesne_random_2002}. We do not use the Gromov--Hausdorff--Prokhorov topology and do not introduce it here. Instead we use the Gromov--Prokhorov topology.

\begin{definition}\label{def:GP_in_introduction}
    Let $(M_1,d_1,r_1,\mu_1)$ and $(M_2,d_2,r_2,\mu_2)$ be two rooted and weighted metric spaces.
    The Gromov--Prokhorov distance $\dGP(M_1, M_2)$ is the infimum of $\eps > 0$ such that exists a measurable subset $R \subseteq M_1 \times M_2$ and a coupling $\nu$ of $\mu_1$ and $\mu_2$, such that
    \begin{align*}
        \nu(R) \geq 1- \eps
        \quad \text{and} \quad
        \sup_{(x,y),(x',y') \in R} \vert d_1(x,x') - d_2(y,y') \vert \leq \eps.
    \end{align*}
\end{definition}

The map $\dGP$ is indeed a metric on isometry classes of weighted, complete, separable metric spaces. Further, we note that the induced topology is Polish, see \cite[Theorem 3.9]{janson_gromov-prohorov_2020}. The induced topology is sometimes also called the Gromov--weak topology and has an alternative formulation using sampling test--functions; more on this can be found in Athreya, Löhr, Winter \cite{athreya_gap_2016}.

We aim to show a statement similar to \eqref{eq:scaling_limit_marchal} for any tree growth chain with uniform backward dynamics. Because this is a very large class of Markov chains, we need to refine the rescaling. Instead of assigning a length $n^{\beta},\beta <0$ to every edge, we rescale $T_n$ inhomogeneously after trimming the leaves. 

\begin{definition}\label{def:trimming_rescaling}
    We define a rooted and weighted metric space $(T_n^{\mathrm{trim}}, d_n^{\mathrm{trim}}, r_n, \mu_n^{\mathrm{trim}})$ as a subtree of $T_n$ as follows:
    \begin{enumerate}
        \item  Remove every leaf and its corresponding edge, call the resulting set $T_n^{\mathrm{trim}}$.
        \item  For every leaf $x \in T_n$, distribute mass $1/n$ to the vertex in $T_n^{\mathrm{trim}}$ that is connected to $x$ in $T_n$. This defines a probability measure $\mu_n^{\mathrm{trim}}$ on $T_n^{\mathrm{trim}}$.
        \item Rescale edge lengths according to an \emph{inhomogenous IP--rescaling}. This means for an edge $(x,y)$ of $T_n^{\mathrm{trim}}$, we set its length to 
\begin{equation}\label{eq:ip_rescaling}
    d_n^{\mathrm{trim}}(x,y) = \left\vert \mu_n^{\mathrm{trim}} \left( F_{T_n^{\mathrm{trim}}}(x) \right) - \mu_n^{\mathrm{trim}} \left( F_{T_n^{\mathrm{trim}}}(y) \right) \right\vert.
\end{equation}
    We extend this to a metric $d_n^{\mathrm{trim}}$ on $T_n^{\mathrm{trim}}$ by adding up edge lengths along the unique path between two vertices. Let $r_n$ be the root of $T_n^{\mathrm{trim}}$.
    \end{enumerate}
\end{definition}

%Before we do this, we trim the tree: remove every leaf and its corresponding edge. This results in a new tree $T_n^{\mathrm{trim}}$ which is a subtree of $T_n$. 

%For every leaf $x \in T_n$, we now distribute mass $1/n$ to the vertex in $T_n^{\mathrm{trim}}$ that is connected to $x$ in $T_n$. This defines a probability measure $\mu_n^{\mathrm{trim}}$ on $T_n^{\mathrm{trim}}$. 

%Now, we rescale $T_n^{\mathrm{trim}}$. We do this by defining edge lengths according to an \emph{inhomogenous IP--rescaling}. This means for an edge $(x,y)$ of $T_n^{\mathrm{trim}}$, we set its length to 
%\begin{equation}\label{eq:ip_rescaling}
%    d_n^{\mathrm{trim}}(x,y) = \left\vert \mu_n^{\mathrm{trim}} \left( F_{T_n^{\mathrm{trim}}}(x) \right) - \mu_n^{\mathrm{trim}} \left( F_{T_n^{\mathrm{trim}}}(y) \right) \right\vert.
%\end{equation}

Note that also $d_n^{\mathrm{trim}}(x,y) = \vert \mu_n ( F_{T_n}(x) ) - \mu_n ( F_{T_n}(y) ) \vert$ for an edge $(x,y) \in T_n^{\mathrm{trim}}$.
See Figure \ref{fig:rescaling_example} for an example. The inhomogeneous rescaling \eqref{eq:ip_rescaling} should be compared to the spacing property \eqref{eq:IP_spacing} of IP--trees, this rescaling renders $T_n^{\mathrm{trim}}$ a discrete analogue of IP--trees.
In this setting, $(T_n,n\geq 1)$ satisfies a scaling limit.

\begin{figure}[bth]
    \centering
    \includegraphics[scale = 1.0]{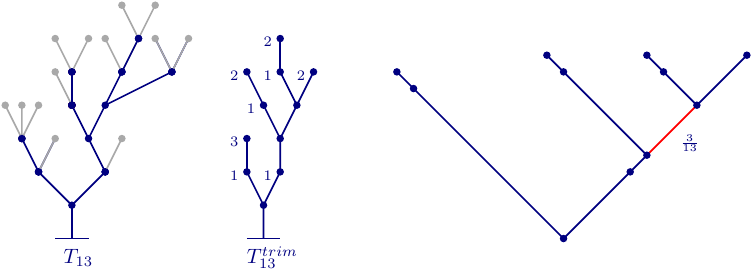}
    \caption{An example of the trimming and rescaling. On the left, there is $T_{13}$; in the middle is $T_{13}^{\mathrm{trim}}$ with a number $k$ indicating an atom of weight $k/13$ and on the right $T_{13}^{\mathrm{trim}}$ is drawn to scale after the rescaling. The marked edge has length $3/13$.}
    \label{fig:rescaling_example}
\end{figure}

\begin{theorem}\label{thm:scaling_limits_new}
    Let $(T_n, n\geq 1)$ be a tree--growth chain with uniform backward dynamics. Then there exists a random IP--tree $(\bfT,d,r,\mu)$ such that 
    \begin{equation}
        (T_n^{\mathrm{trim}}, d_n^{\mathrm{trim}}, r_n, \mu_n^{\mathrm{trim}})\xrightarrow{n\to \infty} (\bfT, d,r,\mu),
    \end{equation}
    almost--surely in the Gromov--Prokhorov topology. Further, the law of $(\bfT,d,r,\mu)$ is given by the pushforward of $\nu$ from Corollary \ref{cor:decomposition} under the projection onto the space of IP--trees.
\end{theorem}

This answers a question of Forman \cite[Question $2$]{forman_exchangeable_2020} if IP--trees arise as scaling limits of suitably rescaled discrete random trees. In fact, any IP--tree can be obtained in this way by using it in Construction \ref{construction:sampling_the_bridge}.

\begin{remark}
This theorem is optimal: both homogeneous rescaling as well as the Gromov-Hausdorff-Prokhorov topology are unsuitable under the general assumption that $(T_n,n\geq 1)$ is a tree growth chain with uniform backward dynamics. To see why a homogeneous rescaling does not work, join two typical realisations of an $\alpha$--stable tree $\mathcal{T}_\alpha$ and an $\alpha'$--stable tree $\mathcal{T}_{\alpha'}$ at the root with $\alpha > \alpha'$. If we were to rescale by $n^{\beta}$, then we would need both $\beta = -1+\frac{1}{\alpha}$ and $\beta = -1+\frac{1}{\alpha'}$ according \eqref{eq:scaling_limit_marchal} for the correct convergence. This is of course not possible. 
To see why the Gromov-Hausdorff-Prokhorov topology is too strong, construct a real tree $\bfT$ in the following way: for $k\geq 1$, let $a_k$ be an atom of weight $2^{-k}$. Connect $a_k$ to the root $r$ by an interval segment of length $1-2^{-k}$. One can see that for all $n\geq 1$ we have $d_{GHP}(T_n, \bfT) \geq 1/2$. This is because there is $k\in \NN$ such that for all $i\leq n$ we have $\xi_i \neq a_k$. This implies that $T_n$ will not converge after rescaling, in essence this is due to $\bfT$ not being a compact metric space. 
\end{remark}

\begin{remark}
    Instead of IP--trees, an alternative limit object would be algebraic trees as introduced by Löhr and Winter \cite{lohr_spaces_2021}. These are tree--like objects where the explicit metric space structure is replaced by a branch point map. 

    Nussbaumer and Winter study sequences of algebraic trees in the specific case of the $\alpha$--Ford model, see \cite{nussbaumer_algebraic_2020}. The fact that the $\alpha$--Ford model has uniform backward dynamics has been shown by Ford \cite[Prop. 42]{ford_probabilities_2005}. For this model, Theorem \ref{thm:scaling_limits_new} can be seen as an alternative to \cite[Section 3]{nussbaumer_algebraic_2020} using real trees. 
\end{remark}

The structure of this paper is as follows: in Section \ref{sec:background} we introduce an encoding for tree--valued Markov chains, namely dendritic systems. In Section \ref{sec:main_proofs} we prove Theorem \ref{thm:main_thm_new_version} and in Section \ref{sec:proof_of_evans} we prove an important auxiliary statement. Lastly in Section \ref{sec:scaling_limits} we prove Theorem \ref{thm:scaling_limits_new}.

%%%%%%%%%%%%%%%%% PRELIMINARIES %%%%%%%%%%%%%%%%%
\section{Dendritic systems}\label{sec:background}

In this section we introduce the notion of dendritic systems. These objects aim to generalise finite leaf--labelled plane trees to infinitely many labels with a strong focus on the leaves. The reason for considering dendritic systems is that they allow us to encode a tree--valued Markov chain as a more static object. This notion is similar to that of didendritic systems which has been introduced in \cite[Def. 5.8]{evans_doob-martin_2017} and slightly modified in \cite[Def. 6.2]{birkner_patricia_2020} which was introduced to generalise binary trees. Our notion has the advantage of accommodating multi--furcating trees as well.

\begin{definition}[Dendritic system]\label{def:planar_den_system}  Let $L\subset \NN$ be a finite or countably infinite set which we call leaf labels.  A planar dendritic system $\mathcal D = (L,\sim,\preceq,p)$ is the collection of the following objects: an equivalence relation $\sim$ on $L\times L$, we denote the space of equivalence classes as $T$; a partial order $\preceq$ on $T$ (which we call the ancestral partial order) and a function $p:T\times T\rightarrow \{0,1,-1\}$ (which we call the planarity function) satisfying the following properties for all $i,j,k,\ell \in L$:
\begin{itemize}
\item[(C1)] $(i,j) \sim (j,i)$, and $(i,j) \sim (k,k)$ if and only if $i=j=k$.
\item[(C2)] $(i,j) \preceq (i,i)$.
\item[(C3)] $(i,j)\preceq(k,\ell)$ and $(k,\ell)\preceq(i,j)$ if and only if $(i,j)\sim(k,\ell)$.
\item[(C4)] $a((i,j),(k,\ell))=\min_\preceq \{ (i,j),(k,\ell),(i,\ell),(i,k),(j,\ell),(j,k) \}$ exists in $T$.
\end{itemize}
Further, the planarity function $p$ satisfies for all $x,y,z\in T$:
\begin{itemize}
\item[(P1)] $p(x,y)=-p(y,x)$.
\item[(P2)] $p(x,y)=0$ if and only if $x\preceq y$ or $y \preceq x$.
\item[(P3)] If $p(x,y) = 1$ and $p(y,z)=1$ then $p(x,z)=1$.
\item[(P4)] If $p(x,y) = 1 $ and $y\preceq z$ then $p(x,z)=1$.
\end{itemize}
\end{definition}

The idea is that (C$1$)--(C$4$) encode a tree--like structure and that (P$1$)--(P$4$) determine a planar structure on $T$. In Lemma \ref{lemma:den_to_tree} we will show that this does indeed generalise plane, leaf--labelled trees.

Here $x \prec y$ corresponds to $x \preceq y$ and $x \neq y$. We will refer to $\{(i,i);i\in L \}$ as the leaves of $\mathcal D$. Moreover, consider two arbitrary vertices $(i,j)$ and $(k,\ell)$. Unless there is an ancestral relationship between $(i,j)$ and $(k,\ell)$, $(P4)$ allows us to determine $p((i,j),(k,\ell))$, namely $p((i,j),(k,\ell))=p(i,k)=p(j,k)=p(i,\ell)=p(j,\ell)$ where we abuse notation to write $i=(i,i)$. We can rephrase this as follows.

\begin{lemma}\label{lemma:den_sys_leaves_p}
For a dendritic system $(L,\sim,\preceq,p)$, $p$ is uniquely determined by $\preceq$ and $\{p(i,j);i,j\in L\}$ where we write $i=(i,i)$. 
\end{lemma}

In the case where $L$ is finite, without loss of generality $L=\{1,\ldots,n\}$ for some $n\in \NN$, dendritic systems correspond to rooted plane trees with leaf labels. Recall from \eqref{eq:def:planetrees} the set $\TT_n$ of plane, rooted, planted trees without vertices of degree $2$ and $n$ leaves. We additionally equip these trees with leaf labels,
\begin{equation*}
    \TT_n^{\mathrm{labelled}} = \left\{T\in \TT_n: \text{the leaves of $T$ are labelled by $\{1,\ldots,n\}$} \right\}.
\end{equation*}

\begin{lemma}\label{lemma:den_to_tree}
The set of all dendritic systems with $L=\{1,\ldots,n\}$ is in bijection with $ \TT_n^{\mathrm{labelled}}$.
\end{lemma}

% \begin{remark}\label{remark:not_necessary_anymore}
% not necessary anymore.
% \comment{
%     Similarly to the space of trees, we equip the space of dendritic systems with the $\sigma$--algebra that is generated by finite projections. In the case of trees, we project onto a finite ball around the root and in the case of dendritic systems we restrict the dendritic system to to $[n] \cap L$.}
% \end{remark}

\begin{proof}[Proof of Lemma \ref{lemma:den_to_tree}.]
On the one hand, let $T\in  \TT_n^{\mathrm{labelled}}$. We direct every edge towards the root. For two leaves labelled $i,j\in T$, we let $b(i,j)\in T$ be their most recent common ancestor. Define a dendritic system as follows: $(i,j)\sim (k,\ell)$ if $b(i,j)=b(k,\ell)$, $(i,j)\preceq (k,\ell)$ if there is a directed path from $b(k,\ell)$ to $b(i,j)$, and $p(i,j)=1$ for two leaves labelled $i,j$ if $i$ precedes $j$ in the lexicographic order of the Ulam--Harris encoding (e.g.\ $(1,1,2,6)$ precedes $(1,1,4,3,2)$). By Lemma \ref{lemma:den_sys_leaves_p}, this determines $p$ uniquely.
% \textcolor{violet}{M: is it clear that C1-C4, P1-P4 hold?}

On the other hand, let $\mathcal D = (L,\sim,\preceq,p)$ be a dendritic system with $L=\{1,\ldots,n\}$. We want to define a plane leaf--labelled tree $T$. The equivalence classes of $(L\times L,\sim)$ are the vertices and we add an edge between $x$ and $y$ if there is no $z$ such that $x \prec z \prec y$. Because $L$ is finite, this yields a tree. We direct an edge $(x,y)$ to $x$ if $x \prec y$ and to $y$ otherwise. The root $r$ is now the minimal element of this directed tree, it exists due to $(C3)$ and $(C4)$. Lastly, we need to impose a planar order on $T$, i.e.\ a valid Ulam--Harris encoding of the vertices. This is done iteratively from the root $r$, encoded by $\emptyset$. Then every vertex has finitely many children $x_1,\ldots,x_n$. Due to $(P3)$ there is a permutation $\sigma$ such that $p(x_{\sigma(i)},x_{\sigma(j)})=1$ if $i<j$. $x_i$ is then encoded by its parents encoding appended with $\sigma(i)$. Loosely speaking, $p$ determines a permutation at each branchpoint of $T$ which we use to obtain a planar order.

One can see that the two procedures described above are inverse to each other.
\end{proof}

This bijection extends to sequences of plane trees: if the trees are suitably consistent, the sequence corresponds to a dendritic system with labels given by $\NN$. 

\begin{lemma}\label{lem:den_system_to_sequences_of_trees}
    Let $(\widetilde{T}_n,n\geq 1)$ be a sequence of leaf--labelled plane trees such that $\widetilde{T}_n \in \TT_n^{\mathrm{labelled}}$. 
    Assume that this sequence is consistent in the sense that when removing the leaf labelled $n+1$ from $\widetilde{T}_{n+1}$ as well as any vertices of degree $2$ we obtain $\widetilde{T}_n$. Then there exists a unique dendritic system $\mathcal{D}=(\NN, \sim, \preceq,p)$ such that $\mathcal{D}$ restricted to $[n]$ is isomorphic to $\widetilde{T}_n$. This induces a bijection between dendritic systems and sequences of consistent plane trees.
\end{lemma}

\begin{proof}
    Let $m<n$ and let $\mathcal{D}_n =([n], \sim_n, \preceq_n, p_n)$ and $\mathcal{D}_m = ([m], \sim_m, \preceq_m, p_m)$ be the dendritic systems corresponding to $\widetilde{T}_{n}$ and $\widetilde{T}_{m}$ according to Lemma \ref{lemma:den_to_tree} respectively. Our assumptions on $(\widetilde{T}_n,n\geq 1)$ imply that $\mathcal{D}_n \big\vert_{[m]} = \mathcal{D}_m$. This means we can define $\mathcal{D}$ inspecting $\mathcal{D}_n:$ if for $i,j,k,\ell \leq n$ we have $(i,j) \sim_n (k,\ell)$ in $\mathcal{D}_n$, then also set $(i,j) \sim (k,\ell)$ in $\mathcal{D}$. The consistency between $\mathcal{D}_n$ and $\mathcal{D}_m$ implies that the choice of $n$ does not matter when defining $\sim$. We define $\preceq$ and $p$ similarly. By construction, we have $\mathcal{D}\big\vert_{[n]} =\mathcal{D}_n$. Further, this determines $\mathcal{D}$ uniquely, compare also to \cite[Lemma 5.13]{evans_doob-martin_2017}.

\end{proof}

\begin{remark}
    The bijection of Lemma \ref{lem:den_system_to_sequences_of_trees} induces a measurable structure on the space of dendritic systems. This is allows us to speak of random dendritic systems. 
\end{remark}

Given a finite permutation $\sigma$ on the leaf labels and a dendritic system $\mathcal D=(L,\sim,\preceq,p)$ we define the dendritic system $\mathcal{D}^\sigma=(L,\sim^\sigma,\preceq^\sigma,p^\sigma)$ by
\begin{enumerate}
    \item $(i,j)\sim^\sigma (k,\ell)$ if and only if  $(\sigma(i),\sigma(j))\sim (\sigma(k),\sigma(\ell))$,
    \item $(i,j)\preceq^\sigma (k,\ell)$ if and only if $(\sigma(i),\sigma(j))\preceq (\sigma(k),\sigma(\ell))$,
    \item $p^\sigma(i,j)=p(\sigma(i),\sigma(j))$.
\end{enumerate}

\begin{definition}
For a random dendritic system $\mathcal{D}$ we say that $\mathcal{D}$ is
\begin{enumerate}
    \item exchangeable, if $\mathcal{D}$ and $\mathcal D^\sigma$ have the same distribution for every finite permutation $\sigma$ on the leaf labels;
    \item ergodic, if we have that for any event $A$ we have that 
    \begin{equation*}
        \PP(\{ \mathcal{D}\in A \} \Delta \{ \mathcal{D}^\sigma\in A \}) = 0 
    \end{equation*}
    for every finite permutation $\sigma$ on the leaf labels implies that $\PP (\{ \mathcal{D}\in A \})\in \{0,1\}$
\end{enumerate}
\end{definition}

We can now state the main result of this section: any tree--valued Markov chain can be encoded in a dendritic system.

\begin{prop}\label{prop:tree_growth_to_dendritic+ergodic}
    Let $(T_n,n\geq 1)$ be a tree--growth chain with uniform backward dynamics. There exists an exchangeable random dendritic system $\mathcal{D}$ such that if we remove the leaf--labels of $(\widetilde{T}_n,n\geq 1)$, the sequence of leaf--labelled trees corresponding to $\mathcal{D}$, we obtain $(T_n,n\geq 1)$. Furthermore, the distribution of $(T_n,n\geq 1)$ is extremal if and only if $\mathcal{D}$ is ergodic.  
\end{prop}

\begin{proof}
    Recall the definition of uniform backward dynamics from Definition \ref{def:uniform_backwards}. We cannot apply Lemma \ref{lem:den_system_to_sequences_of_trees} right away as $(T_n,n\geq 1)$ does not have labelled leaves. We follow a similar procedure to \cite[Section 5]{evans_doob-martin_2017} to equip $(T_n,n\geq 1)$ with leaf labels, we recall this procedure here. Let $N \in \NN$ be large and let $S_N$ be a tree with the same distribution as $T_N$. Label the leaves of $S_N$ uniformly by $[N]$ to obtain $\widetilde{S}_N$. For $m<N$, let $\widetilde{S}_m$ be the tree obtained by removing the leaves with labels $m+1,\ldots, N$ from $\widetilde{S}_N$ as well as any vertices of degree $2$. Due to our assumption of uniform backward dynamics, $(\widetilde{S}_1, \ldots, \widetilde{S}_N)$ has the same distribution as $(T_1,\ldots, T_N)$ after removing the leaf labels. Furthermore, if we perform this procedure with $M>N$ and obtain $(\widetilde{S}_1', \ldots, \widetilde{S}_M')$, then $(\widetilde{S}_1, \ldots, \widetilde{S}_N)$ and $(\widetilde{S}_1', \ldots, \widetilde{S}_N')$ have the same distribution. Therefore, by Kolmogorov's extension theorem, there is a unique Markov chain $(\widetilde{S}_n,n\geq 1)$ of leaf--labelled plane trees such that, if the leaf labels are removed, $(\widetilde{S}_n,n\geq 1)$ has the same distribution as $(T_n,n\geq 1)$. Moreover, the backward dynamics of $(\widetilde{S}_n,n\geq 1)$ are given by removing the leaf with the biggest label. Because we chose the leaf leaves of $\widetilde{S}_N$ uniformly, these leaf labels are exchangeable.  

    Given a realisation of $(\widetilde{S}_n,n\geq 1)$, there exists a corresponding dendritic system $\mathcal{D}$ by Lemma \ref{lem:den_system_to_sequences_of_trees}. Consider this as random dendritic system, $\mathcal{D}$ is exchangeable because the leaf labels of $(\widetilde{S}_n,n\geq 1)$ are exchangeable. Given $\mathcal{D}$, we can recover $(\widetilde{S}_n,n\geq 1)$ by Lemma \ref{lem:den_system_to_sequences_of_trees} and thus $(T_n,n\geq 1)$ by removing the leaf labels. 
    Lastly, $\mathcal{D}$ is ergodic if and only if the distribution of $(T_n,n\geq 1)$ is extremal by \cite[Proposition 5.19]{evans_doob-martin_2017}. Their proposition deals with didendritic rather than dendritic systems but that does not affect the proof. 
\end{proof}

Due to the above propositions we will study exchangeable, ergodic dendritic systems instead of extremal tree--valued Markov chains with uniform backward dynamics.

%%%%%%%%%%%% CLASSIFICATION %%%%%%%%%%%%%
\section{Classification}\label{sec:main_proofs}

In the previous section we have shown that we can encode a tree growth chain $(T_n , n\geq 1)$ with uniform backwards dynamics in an exchangeable dendritic system $\mathcal{D}$.
The proof of Theorem \ref{thm:main_thm_new_version} consists of two steps, recall that we are looking for a rooted, weighted real tree $(\bfT,d,r,\mu)$, a planar order $\psi$ as in Definition \ref{def:planar}, a function $\lambda:\bfT\to[0,1]$ and functions $\beta_a:\{\text{subtrees of $a$}\} \to [0,1]$ for every atom $a$ of $\mu$.
Consider $(T_n , n\geq 1)$ so that its distribution is extremal, equivalently $\mathcal{D}$ is ergodic. In the first step, we show that there is \emph{some} $(\bfT,d,r,\mu,\psi,\lambda,\{\beta_a\}_a)$ that represents $(T_n , n\geq 1)$, see Proposition \ref{prop:there_is_dprt_for_dendritic}. In the second step, we show that this can be chosen uniquely, see Proposition \ref{prop:let_it_be_IP}. And in a third step we show Corollary \ref{cor:decomposition} by decomposing tree--valued Markov chains into extremal distributions. 

First, we translate Construction \ref{construction:sampling_the_bridge} into the language of dendritic systems.
This means we state how to sample a dendritic system $\mathcal{D}=(\NN, \sim, \preceq, p)$ from $(\bfT,d,r,\mu,\psi,\lambda,\{\beta_a\}_a)$. We split the construction into two parts, sampling $(\NN, \sim, \preceq)$ from the tree and determining the planarity function $p$ using $(\psi, \lambda, \{\beta_a\}_a)$ and extra randomness. 

\begin{construction}\label{construction:sample_den_from_tree:part_1}
Sample a sequence $\{ \xi_i\}_{i \in \NN}$ $i.i.d.$\ from $\mu$. We then define for $i,j,k,\ell \in \NN$:

\begin{enumerate}
    \item $(i,i)\sim (k,\ell)$ if and only if $i=k=\ell$.
    
    \item $(i,j)\sim(k,\ell)$ for $i\neq j,k\neq \ell$ if and only if $\llbracket r, \xi_i \rrbracket\cap \llbracket r, \xi_j \rrbracket=\llbracket r, \xi_\ell \rrbracket\cap \llbracket r, \xi_k \rrbracket$.
    
    \item A partial order $\preceq$ on $\NN^2 / \sim$ is inherited from the ancestral partial order $<$ on $\bfT$ and adding $(i,j)\prec(i,i)$ for $i \neq j$. This means for distinct $i,j,k,\ell \in \NN$ we set
    $$(k,\ell)\prec (i,j) \quad \text{   if and only if  }  \quad \llbracket r,\xi_k \rrbracket\cap \llbracket r,\xi_\ell \rrbracket \subsetneq \llbracket r,\xi_i \rrbracket \cap \llbracket r,\xi_j \rrbracket.$$
    
\end{enumerate}
\end{construction}

\begin{construction} \label{construction:sample_den_from_tree:part_2} In the setting of Construction \ref{construction:sample_den_from_tree:part_1}, we now sample $\{ U_i\}_{i \in \NN}$ $i.i.d.$\ uniform random variables from $[0,1]$ independently from $\{ \xi_i\}_{i \in \NN}$.
To determine the planarity function $p$, we distinguish four cases. Recall that we decomposed $\mu = \mu_{\mathrm{atoms}} + \mu_s + \mu_\ell$ into the mass on the atoms, the skeleton (diffusely) and the leaves (diffusely). 
    
    \begin{enumerate}
                    
        \item If neither $\xi_i \preceq \xi_j$ nor $\xi_j \preceq \xi_i$, then $p$ is determined by $\psi_2(\xi_i,\xi_j)$, see Figure \ref{fig:p_from_psi} for an illustration. More precisely, there is a unique plane tree with two leaves and one root, we need to check which of the two leaves is labelled $i$ and which is labelled $j$. We set 
        $$p(i,j)= \begin{cases}1 \quad &\text{if the left leaf is labelled $i$ and the right leaf is labelled $j$ in } \psi_2(\xi_i,\xi_j), \\ -1  &\text{otherwise}. \end{cases}$$
        
        \item If $\xi_i \prec \xi_j$ we distinguish two cases.
        
        \begin{enumerate}
            \item In the case that $\xi_i \in \supp \mu_s$, we set 
            $$p(i,j) =  \begin{cases}
            1 \quad & \text{if } U_i < \lambda( \xi_i), \\
            -1      & \text{if } U_i \geq  \lambda( \xi_i).
            \end{cases}$$
            
            \item In the case that $\xi_i \in \supp  \mu_{\mathrm{atoms}}$, let $\bfS_j$ be the subtree of $\xi_i$ such that $\xi_j \in \bfS_j$.
            We then set 
            $$p(i,j)= \begin{cases} 1 \quad & \text{if } U_i < \beta_{\xi_i}(\bfS_j), \\ -1 & \text{if } U_i \geq \beta_{\xi_i}(\bfS_j). \end{cases}$$

        \end{enumerate}
        
        \item If $\xi_i \succ \xi_j$, we do the same as above with reversed roles for $i$ and $j$.
        
        \item If $\xi_i = \xi_j$, we set 
        $$p(i,j)= \begin{cases}1 \quad &\text{if } U_i<U_j, \\ -1 &\text{if } U_i \geq U_j. \end{cases}$$
    \end{enumerate}
    
\end{construction}

Note that the value of $p((i,j),(k,\ell))$, where $(i,j),(k,\ell)$ are not leaves, is uniquely determined by the values of $p$ on the leaves by imposing the consistency relations $(P1)$--$(P4)$, see Lemma \ref{lemma:den_sys_leaves_p}.

\begin{figure}[ht]
    \centering
    \includegraphics[scale=1]{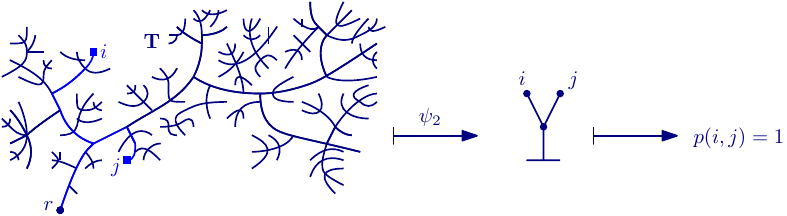}
    \caption{How $\psi$ is used to determine $p$ in the sampling construction. Given two (random) points $\xi_i$ and $\xi_j$ which do not satisfy $\xi_i \preceq \xi_j$ nor $\xi_j \preceq \xi_i$, $\psi_2$ maps these points to the unique tree with two leaves. The two options to label the leaves correspond to $p(i,j)=\pm 1$ respectively.}
    \label{fig:p_from_psi}
\end{figure}

We will show two key propositions.

\begin{prop}\label{prop:there_is_dprt_for_dendritic}
Let $\mathcal{D}$ be an ergodic, exchangeable dendritic system.
Then there exists deterministic $(\bfT, d, r, \mu, \psi, \lambda, \{\beta_a\}_a)$ such that the distribution of $\mathcal D$ equals the one sampled from Constructions \ref{construction:sample_den_from_tree:part_1} and \ref{construction:sample_den_from_tree:part_2}.
\end{prop}

\begin{prop}\label{prop:let_it_be_IP}
The objects $(\bfT, d, r, \mu, \psi, \lambda, \{\beta_a\}_a])$ in Proposition \ref{prop:there_is_dprt_for_dendritic} can be
chosen uniquely in the sense of Theorem \ref{thm:main_thm_new_version}. In particular, $(\bfT,d,r,\mu)$ can be uniquely chosen to be an IP--tree up to isometries.
\end{prop}

In Section \ref{sec:existence} we will prove Proposition \ref{prop:there_is_dprt_for_dendritic} building on Proposition \ref{prop:evans}, whose proof, in turn, will be appended in Section \ref{sec:proof_of_evans}. Proposition \ref{prop:let_it_be_IP} will be shown in Section \ref{sec:uniqueness}.
Theorem \ref{thm:main_thm_new_version} now follows from these propositions. 

\begin{proof}[Proof of Theorem \ref{thm:main_thm_new_version}.] 
Suppose that the distribution of $(T_n, n\geq 1)$ is extremal. Then by Proposition \ref{prop:tree_growth_to_dendritic+ergodic} it corresponds to an ergodic dendritic system $\mathcal{D}$. By Propositions \ref{prop:there_is_dprt_for_dendritic} and \ref{prop:let_it_be_IP} there exists $(\bfT, d, r, \mu, \psi, \lambda, \{\beta_a\})$ -- with the desired uniqueness -- such that $\mathcal D$ has the same distribution as the dendritic system obtained through Constructions \ref{construction:sample_den_from_tree:part_1} and \ref{construction:sample_den_from_tree:part_2}. By Proposition \ref{prop:tree_growth_to_dendritic+ergodic} we recover $(T_n, n\geq 1)$, the sampling construction for $\mathcal{D}$ becomes Construction \ref{construction:sampling_the_bridge}.
\end{proof}

\begin{proof}[Proof of Corollary \ref{cor:decomposition}.]
    Let $(T_n, n\geq 1)$ be any tree growth chain with uniform backward dynamics and let $\mathcal{D}$ be the dendritic system that encodes $(T_n, n\geq 1)$ by Proposition \ref{prop:tree_growth_to_dendritic+ergodic}. The distribution of any such $\mathcal{D}$ can be uniquely decomposed into ergodic distributions, see for example \cite[Theorem A1.4]{kallenberg_probabilistic_2005}.
    This means that there exists a unique probability measure $\nu$ on $\mathcal{M}_1$, the space of probability measures on the space of dendritic systems such that
    \begin{equation*}
        \PP \left(\mathcal{D}\in \cdot\right) = \int_{\mathcal{M}_1} \rho\left( \mathcal{D} \in \cdot\right)\nu(d\rho),
    \end{equation*}
    where $\nu$--almost every $\rho$ is ergodic. 
    Thus under the distribution $\rho$, the tree--growth chain $(T_n, n\geq 1)$ is extremal by  Proposition \ref{prop:tree_growth_to_dendritic+ergodic}. Lastly, by Theorem \ref{thm:main_thm_new_version}, there is $\bfT=(\bfT, d, r, \mu, \psi, \lambda, \{\beta_a\})$ such that $\rho=\rho_{\bfT}$.
\end{proof}

\subsection{Existence of a sampling representation}\label{sec:existence}

In this section we will prove Proposition \ref{prop:there_is_dprt_for_dendritic}.
A key step in this proof is the following Proposition \ref{prop:evans}. The proposition deals with the following construction. 

\begin{construction}
Assume that we are given a weighted, rooted real tree $(\bfT, d, r, \mu)$ and a function $F:(\bfT \times [0,1])^2\times [0,1]\rightarrow \{ \pm 1 \}$. Let $Leb$ be the Lebesgue measure on $[0,1]$. Assume that $F$ satisfies the following consistency relations for $\mu$--almost every $x,y,z$ and $Leb$--almost every $u,v,w,a,b,c$.
\begin{enumerate}\label{construction:sample_den_from_tree:part_3}
    \item [(F1)] $F(x,u,y,v,a) = -F(y,v,x,u,a)$,
    \item [(F2)] if $F(x,u,y,v,a)=F(y,v,z,w,b)$ then also $F(x,u,z,w,c)=F(x,u,y,v,a)$,
    \item [(F3)] if $\llbracket r,x \rrbracket \cap \llbracket r,y \rrbracket \notin \{\llbracket r,x \rrbracket, \llbracket r,y \rrbracket \}$ and $\llbracket r,y \rrbracket \subsetneq \llbracket r,z \rrbracket$ then $F(x,u,y,v,a)=F(x,u,z,w,b)$,
    \item [(F4)] if $\llbracket r,x \rrbracket \subsetneq \llbracket r,y \rrbracket \subsetneq \llbracket r,z \rrbracket$ then $F(x,u,y,v,a)=F(x,u,z,w,c)$.
\end{enumerate}
Then, in the context of Construction \ref{construction:sample_den_from_tree:part_1}, sample $i.i.d.$ uniform random variables $\{ U_i, U_{ij} \}_{i,j\in \NN, i<j}$ from $[0,1]$. 
For two leaves $i$ and $j$, 
abusing notation to write $i=(i,i)$, we set
$$p(i,j)=F(\xi_i,U_i,\xi_j, U_j, U_{i,j}).$$
The value of $p((i,j),(k,\ell))$, where $(i,j),(k,\ell)$ are not leaves, is uniquely determined by the values of $p$ on the leaves by imposing the consistency relations $(P1)$--$(P4)$, see Lemma \ref{lemma:den_sys_leaves_p}. Therefore this defines a planarity function $p$. 
\end{construction}

Combining Constructions \ref{construction:sample_den_from_tree:part_1} and \ref{construction:sample_den_from_tree:part_3} yields a dendritic system. In fact, the dendritic systems we are interested in can be represented this way.

\begin{prop}\label{prop:evans}
Every ergodic, exchangeable dendritic system $\mathcal D$ can be represented by a real tree $(\bfT, d, r)$, a probability measure $\mu$ on $\bfT$ and a measurable function $F:(\bfT \times [0,1])^2\times [0,1]\rightarrow \{ \pm 1 \}$ in such a way that we have $\bfT = \spn(\supp(\mu))$, $F$ satisfies the consistency relations stated in Construction \ref{construction:sample_den_from_tree:part_3} and $\mathcal D$ is then equal in distribution to the dendritic system obtained through Constructions \ref{construction:sample_den_from_tree:part_1} and \ref{construction:sample_den_from_tree:part_3}.
\end{prop}

\begin{remark}
    The consistency conditions on $F$ correspond somewhat naturally to the consistency conditions of $p$. We also note that in Theorem \cite[Theorem 8.2]{evans_doob-martin_2017} similar consistency relations are imposed.
    Evans, Grübel and Wakolbinger state that the additional randomisation given by $F$ is often not necessary \cite[p. 252 and 275]{evans_doob-martin_2017}. In Theorem \ref{thm:main_thm_new_version} we essentially formalise how much additional randomisation is necessary via $\lambda$ and $B$ compared to planar structure that can be intrinsically imposed on $\bfT$ by $\psi$.
\end{remark}

We will prove Proposition \ref{prop:evans} in Section \ref{sec:proof_of_evans}.
The proof of this proposition is analogous to the proof of \cite[Theorem 8.2]{evans_doob-martin_2017} but makes use of the full generality of a theorem of Gufler \cite{gufler_representation_2018}. 
Remark that it suffices to describe the dendritic system restricted to $[n]$ for every $n$ to determine the distribution of the dendritic system uniquely. 

In the remainder of Section \ref{sec:existence}
 we prove Proposition \ref{prop:there_is_dprt_for_dendritic} building on Proposition \ref{prop:evans}, which provides a tree $\bfT$ and a measure $\mu$. Hence we can consider specific trees and measures.
We first consider three special cases for $\bfT$: when the mass is distributed diffusely on the skeleton, when it is supported diffusely on the leaves and when it has atoms. Note that these cases arise naturally as we can decompose $\mu = \mu_{\mathrm{atoms}} + \mu_\ell + \mu_s$ into measures that place mass only on atoms, diffusely on leaves and diffusely on the branches (the tree without the leaves) respectively. The proof of Proposition \ref{prop:there_is_dprt_for_dendritic} then combines the ideas of these three cases. 

Before doing that, we will state an elementary lemma.

\begin{lemma}\label{lemma:app:ind}
Assume $X,Y,Z$ are independent random variables with laws $\lambda_1,\lambda_2, \lambda_3$ respectively and $f$ a measurable function. If $f(X,Y)=f(X,Z)$ $\PP$--$a.s.$ then $f(x,Y)$ is $\PP$--$a.s.$ constant for $\lambda_1$--$a.e.$ $x$. We then have $f(X,Y)=g(X)$ $\PP$--$a.s.$ for some measurable function $g$.  
%\textcolor{gray}{$\PP$ and the state spaces are left implicit, there is some notational abuse for the last equation.}
\end{lemma}

\begin{proof}
Note first that $f(x,Y)=f(x,Z)$ for $\lambda_1$--a.e.\ $x$, $\PP$--almost surely. Fix $x$ such that $f(x,Y)=f(x,Z)$ $\PP$--$a.s.$ and let $f_x(\cdot)=f(x,\cdot)$. Then $f_x(Y)=f_x(Z)$ almost surely. Because $Y\indep Z$, $Y \indep f_x(Z)$ and hence $f_x(Y) \indep f_x(Y)$ which implies that $f_x(Y)$ is $\PP$--$a.s.$ constant -- define $g(x)$ to be this constant, i.e.\ $g(x)=\int f(x,y)\lambda_2(dy)$ which is measurable. This completes the proof because $\lambda_1 \big(\big\{x: f(x,Y)=f(x,Z) \big\} \big)=1$.
\end{proof}

\begin{lemma}\label{lemma:case_s_line}
Assume that from Proposition \ref{prop:evans} we get $(\bfT, d, r, \mu, F)$ so that $\bfT=[0,1], r=0, \mu = Leb$. Then there exists $(\psi, \lambda, \{\beta_a \}_a)$ so that the dendritic systems obtained by Constructions \ref{construction:sample_den_from_tree:part_2} and \ref{construction:sample_den_from_tree:part_3} respectively have the same distribution.
Here, $\psi$ is the only possible planar order for $[0,1]$ and $\lambda$ is determined by $F$. Because $Leb$ has no atoms, $\{\beta_a\}_a$ is trivial. 
\end{lemma}

Note that in this case $\mu$ is supported diffusely on the spine. There is only one planar order for this tree in this case: given $n$ distinct points and the root in $\bfT$, $\psi_n$ maps them to a line graph with $n$ edges. 

Here the sampling representation is very concise: sample $n$ uniform $iid$ points $\xi_1,\ldots,\xi_n$ on $[0,1]$. 
For each $\xi_i$, sample a Bernoulli random variable with parameter $\lambda(\xi_i)$ independently and attach a leaf labelled $i$ to the left of $\xi_i$ -- if the Bernoulli random variable equals $1$ -- or to the right of $\xi_i$ otherwise. $T_n$ is then the plane combinatorial tree spanned by $r$ and the added leaves. This means that $T_n$ is a binary tree consisting of a spine with leaves hanging off the spine left and right.

\begin{proof}[Proof of Lemma \ref{lemma:case_s_line}.] 
Due to the special assumptions on $\bfT$ and $\mu$, we are almost surely always in the case where $\{\xi_i \prec \xi_j\}$ or $\{\xi_i \succ \xi_j\}$ for all $i$ and $j$. This leaves us to show that on the event $\{\xi_i \prec \xi_j \}$ we have almost surely
$$p(i,j) =  \begin{cases}
            1 \quad & \text{if } U_i < \lambda( \xi_i), \\
            -1      & \text{if } U_i \geq  \lambda( \xi_i)
            \end{cases}$$
for a suitable branch weight function $\lambda$.

Without loss of generality we always condition on the event $\{ \xi_i \prec \xi_j \}$ in the following. By Proposition \ref{prop:evans} we have:
$$p(i,j)=F(\xi_i,U_i,\xi_j, U_j, U_{ij}).$$
We sample a second copy of $(\xi_j, U_j, U_{ij})$ independently of $(\xi_i,U_i,\xi_j, U_j, U_{ij})$ and denote it by $(\xi_j^*, U_j^*, U_{ij}^*)$ and we restrict ourselves to the event $\{ \xi_i \prec \xi_j^* \}$. Due to the consistency properties of $F$, more precisely $(F4)$, we almost surely have
$$F(\xi_i, U_i, \xi_j, U_j, U_{ij})=F(\xi_i, U_i, \xi_j^*, U_j^*, U_{ij}^*).$$
Consider the family of regular conditional distribution $\PP(\cdot \vert \xi_i = x, \xi_j \succ \xi_i, \xi_j^* \succ \xi_i)$ under which for all $x \in (0,1)$, both $\xi_j$ and $\xi_j^*$ are $Uniform((x,1))$ distributed and $U_i,\xi_j,U_j,U_{ij},\xi_j^*,U_j^*,U_{ij}^*$ are independent of each other. This means we can apply Lemma \ref{lemma:app:ind} which tells us that there exists some measurable function $V': [0,1]^2 \to \{ \pm 1 \}$ such that $F(\xi_i, U_i, \xi_j, U_j, U_{ij})=V'(x, U_i)$, $\PP(\cdot \vert \xi_i = x, \xi_j \succ \xi_i, \xi_j^* \succ \xi_i)$--almost surely for $Leb$--almost every $x$. Integrating over $x$ and reversing the roles for $i$ and $j$ already gives us the following description of $p$:
$$p(i,j)=\begin{cases}V'(\xi_i, U_i) \quad & \text{if} \quad  \xi_i \prec \xi_j, \\ 
V'(\xi_j, U_j) \quad & \text{if} \quad \xi_j \prec \xi_i.  \end{cases}$$
Let us now define the branch weight function $\lambda$ by $\lambda(x)= \PP(V'(x,U)=1)$ for $x \in [0,1]$. We use this to define:
$$\Tilde{p}(i,j)=\begin{cases}V(\xi_i, U_i) \quad & \text{if} \quad  \xi_i \prec \xi_j, \\ V(\xi_j, U_j) \quad & \text{if} \quad \xi_j \prec \xi_i . \end{cases}$$
where $V(x,u)=\II_{\lambda(x)<u} - \II_{\lambda(x)>u}$. If we manage to show 
\begin{equation}\label{eq:lta}
    \{p(i,j): 1\leq i \neq j \leq n \} \overset{d}{=}\{\Tilde{p}(i,j): 1\leq i \neq j \leq n \}
\end{equation}
for every $n \in \NN$ then we have completed the proof the lemma. To this end, fix $2\leq n \in \NN$. Denote $\PP(\cdot \vert \xi_1 \prec \ldots \prec \xi_n, \xi_1,\ldots,\xi_n)$ by $\PP^\xi$ and let $a_1,\ldots,a_{n-1}\in \{ \pm 1 \}^{n-1}$. To show \eqref{eq:lta}, it suffices to show that we $\PP^\xi$--a.s.\ have
\begin{equation}\label{eq:case_s_line_a}
    \PP^\xi\big(\forall i < n: \ p(i,i+1)= a_i  \big)=\PP^\xi\big(\forall i < n: \ \Tilde{p}(i,i+1)= a_i  \big).
\end{equation}
Indeed, assume we are given $b=\{ b_{ij} \}_{i,j \leq n, i \neq j} \in \{ \pm 1 \}^{n(n-1)}$ with $b_{ij} = - b_{ji}$ and $b_{ij} = b_{in}$ for $i<n$. If we are given $b$ that does not satisfy these assumptions, then 
\begin{equation*}
    \PP^\xi\big(\forall i,j \leq n, i \neq j: \ p(i,j)= b_{ij}  \big) = 0,
\end{equation*}
due to $b$ violating the consistency properties required in $(P1)$ or $(P4)$. The same holds for $\Tilde{p}$. Now if $b$ satisfies the assumptions stated above, we then have
\begin{align*}
    \PP^\xi\big(\forall i,j \leq n, i \neq j: \ p(i,j)= b_{ij}  \big) \overset{(P1)}{=} \PP^\xi\big(\forall i,j \leq n, i < j: \ p(i,j)= b_{ij}  \big) \overset{(P4)}{=} \PP^\xi\big(\forall i < n: \ p(i,i+1)= b_{in}  \big),
\end{align*}
where both equalities hold $\PP$--almost surely. We can apply $(P4)$ here because for $i+1<j$ we have $(j,(i+1))\prec (j,j)$ due to $\xi_i < \xi_{i+1} < \xi_{j}$, hence $p(i,i+1)= p(i,(j,i+1))=p(i,j)$. The same is true for $\Tilde{p}$, which means that it suffices to only check \eqref{eq:case_s_line_a}.
Consider now \eqref{eq:case_s_line_a}, due to our definition of $\lambda$ and the independence of $U_1,\ldots,U_n$, we $\PP$--almost surely have
\begin{align*}
    \PP^\xi\big(\forall i < n: \ p(i,n)= a_i  ) &= \PP^\xi\big(\forall i<n: V'(\xi_i,U_i)=a_i\big) \\
    &= \prod_{i=1}^{n-1}\PP^\xi\big( V'(\xi_i,U_i)=a_i \big) \\
    &= \prod_{i=1}^{n-1} \big( \lambda(\xi_i) \II_{a_i=1}+ (1-\lambda(\xi_i))\II_{a_i=-1} \big) \\
    &= \prod_{i=1}^{n-1} \PP^\xi \big(V(\xi_i,U_i)=a_i \big) \\
    &= \PP^\xi\big(\forall i < n: \ \Tilde{p}(i,n)= a_i   \big).
\end{align*}
This implies \eqref{eq:lta} which completes the proof of the lemma. 
\end{proof}

As the second special case, we consider the case where $\bfT$ is arbitrary but $\mu$ is supported only on the leaves of $\bfT$. In this case $\lambda$ and $\{\beta_a\}_a$ are trivial, but we do need to define $\psi$. 

\begin{lemma}\label{lemma:case_s_leaves_only}
Assume that from Proposition \ref{prop:evans} we get $(\bfT, d, r, \mu, F)$ so that $\mu$ is supported diffusely on the leaves of $\bfT$. Then there exists $(\psi, \lambda, \{\beta_a\}_a)$ so that the dendritic systems obtained by Constructions \ref{construction:sample_den_from_tree:part_2} and \ref{construction:sample_den_from_tree:part_3} respectively have the same distribution. $\lambda$ and $\{\beta_a\}_a$ are trivial. 
\end{lemma}

This shows that in this case the additional randomisation provided by $F$ is not required and the planar order can be intrinsically attributed to $\bfT$ via $\psi$. This strongly depends on the assumption that $\mu$ is supported diffusely on the leaves. 

\begin{proof} 
Our main concern is to define the planar order $\psi$, recall the definition from Definition \ref{def:planar}. More concretely, we need to define $\psi_n(x_1,\ldots,x_n)$ for any $x=(x_1,\ldots,x_n) \in \bfT^n$ such that for all $i\neq j$ we have $x_i \npreceq x_j$ and $x_j \npreceq x_i$ as this determines $\psi$ uniquely. Note that we need to do this for all $(x_1,\ldots,x_n)$ with this property and not just on the support of $\mu$. Fix such $(x_1,\ldots,x_n) \in \bfT^n$. 

We let $\bfS_i$ be the subtree corresponding to $x_i$ in the following sense: let $\overline{x_i}$ be the branchpoint in $\spn(r,x_1,\ldots,x_n)$ closest to $x_i$. Let $y_i$ be the middle point of the segment $[\overline{x_i},x_i]_\bfT$. We set $\bfS_i=F_\bfT(y_i)$, the fringe subtree of $y_i$, see Figure \ref{fig:psi_for_t} for an illustration. The reason for using $F_\bfT(y_i)$ instead of $F_\bfT(x_i)$ is that we always have $\mu( F_\bfT(y_i))>0$ but not necessarily $\mu(F_\bfT(x_i))>0$. This is true because we have $\bfT = \spn(\supp(\mu))$ by Proposition \ref{prop:evans}. Note that if $x_i$ is a leaf, then $\mu(F_\bfT(x_i))=\mu(\{x_i\})=0$ because we assumed $\mu$ to be diffuse. 

Define now $\PP^x=\bigotimes_{i=1}^n\frac{1}{\mu(\bfS_i)}\mu\vert_{\bfS_i}$. Sampling $(\xi_1,\ldots, \xi_n)$ from $\PP^x$ is equivalent to sampling $(\xi_1,\ldots, \xi_n)$ from $\PP$ and conditioning on $\{\xi_i \in \bfS_i\}$ for every $i$. Using $(\xi_1,\ldots, \xi_n)$ sampled from $\PP^x$ in Construction \ref{construction:sample_den_from_tree:part_1} and \ref{construction:sample_den_from_tree:part_3} (instead of $i.i.d.$ $\mu$ samples) yields a random dendritic system $\mathcal{D}_n=([n], \sim_n,\preceq_n, p_n)$ with leaves labelled by $[n]$. We claim that $\mathcal{D}_n$ is in fact almost surely constant. If this is the case, then $\mathcal D_n$ will correspond to a non--random tree $T_n$ by Lemma \ref{lemma:den_to_tree} which is a combinatorial, leaf--labelled and plane tree. We then set
$$\psi_n(x_1,\ldots,x_n)=T_n.$$

Next we need to prove the claim than $\mathcal{D}_n$ is almost surely constant. Let $(\xi_i,U_i,U_{ij};i,j \leq n)$ be the random variables involved in the construction of $\mathcal{D}_n$, let $(\xi_i^*,U_i^*,U_{ij}^*;i,j \leq n)$ be an independent copy with the same distribution, extending the probability space. For every $i\neq j$ we have
$$F(\xi_i,U_i,\xi_j,U_j,U_{ij})=F(\xi_i^*,U_i^*,\xi_j^*,U_j^*,U_{ij}^*)$$
due to the consistency properties $(F2)$ and $(F3)$ of Proposition \ref{prop:evans} and the fact that $\xi_i,\xi_i^*\in \bfS_i$ whereas $\xi_j,\xi_j^*\in \bfS_j$. Informally, sampling once from $\bfS_i$ and $\bfS_j$ already determines which subtree is to the left of the other subtree, hence the second sample must agree with the left--right prescription of the first sample. More formally, $(\xi_i,U_i,U_{ij};i,j \leq n)$ and $(\xi_i^*,U_i^*,U_{ij}^*;i,j \leq n)$ are independent and hence $F$ restricted to $\bfS_i \times [0,1]\times \bfS_j \times [0,1] \times [0,1]$ is $\PP^x$--almost surely constant for any $i\neq j$ by Lemma \ref{lemma:app:ind}. This proves the claim that $\mathcal{D}_n$ is constant and thus yields the map $\psi_n$. By construction we also have the property that $\psi_n(x_1,\ldots,x_n)$ as non--plane combinatorial tree is the combinatorial tree corresponding to $\spn(r,x_1,\ldots,x_n)$. 

\begin{figure}[t]
    \centering
    \includegraphics[scale=1.1]{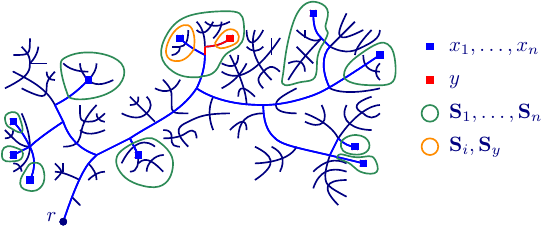}
    \caption{The subsets of $\bfT$ involved in the proof of Lemma \ref{lemma:case_s_leaves_only}.}
    \label{fig:psi_for_t}
\end{figure}

What remains to be shown is that $\psi_{n}(x_1,\ldots,x_{n})$ embeds into $\psi_{n+1}(x_1,\ldots,x_{n},y)$ for any $y\in \bfT$ such that for all $i$ we have $y \npreceq x_i$ and $x_i \npreceq y$. Let $(\bfS_1,\ldots,\bfS_n)$ denote the trees used in the construction of $\psi_n(x_1,\ldots,x_n)$ and $(\bfS_1',\ldots,\bfS_n',\bfS_y)$ the trees used in the construction of $\psi_{n+1}(x_1,\ldots,x_n,y)$.

We need to observe that in general we have $\bfS_i\neq\bfS_i'$. This is because either $y\in \bfS_i$ for some $i$ or because including $y$ introduces new branchpoints in $\spn(r,x_1,\ldots,x_n,y)$ which change $\overline{x_i}$ and hence $\bfS_i$. Nevertheless, we always have
$$\bfS_i' \subseteq \bfS_i \quad \forall i \leq n.$$

Recall that in the first part of the proof we sampled $\xi_i$ from $\mu_i$, i.e.\ from $\mu$ conditioned on $\xi_i\in \bfS_i$ to determine $\psi_n(x_1,\ldots,x_n)$. It is easy to see that we can a posteriori replace $\bfS_i$ with $\bfS_i'$ in the construction and still obtain the same $\psi_n(x_1,\ldots,x_n)$. This is because we concluded that $F(\xi_i,U_i,\xi_j,U_j,U_{ij})$ is almost surely constant for every $i,j\leq n$ so we can condition $\xi_i,i\leq n$ to be in the smaller sets $\bfS'_i$ for every $i\leq n$ and deduce the required constancy on these sets.

There is a distinct advantage to using $(\bfS_i', i\leq n)$ instead of $(\bfS_i, i\leq n)$ as these are the subsets of $\bfT$ used in the construction of $\psi_{n+1}(x_1,\ldots,x_n,y)$. This then yields the embedding of $\psi_n(x_1,\ldots,x_n)$ into $\psi_{n+1}(x_1,\ldots,x_n,y)$. There is a unique edge or branchpoint in $\psi_n(x_1,\ldots,x_n)$ to which we attach the leaf corresponding to $y$, determined by where $\overline{y}$ is located in $\spn(r,x_1,\ldots,x_n)$ because Definition \ref{def:planar} requires that $\psi_{n+1}(x_1,\ldots,x_n,y)$ without the planar order is the combinatorial tree corresponding to $\spn(r,x_1,\ldots,x_n,y)$. There also is a determined way for the planar order of the new leaf corresponding to $y$ which is compatible with the planar order of $\psi_n(x_1,\ldots,x_n)$ because we used the same sets $\bfS_i,i\leq n$, respectively identically distributed random variables, to construct the planar orders.

To conclude the proof of this lemma, note that the choice of $\lambda$ is trivial. Indeed, recall that we view $\lambda$ as an element of $L^1(\mu_s)$. Here $\mu_s = 0$, so $L^1(\mu_s)$ contains only a single element. Similarly, we need not define $B$ as $\mu$ does not have any atoms, hence $B$ is trivial. 

We have to check that the distribution of the dendritic system obtained via Constructions \ref{construction:sample_den_from_tree:part_1} and \ref{construction:sample_den_from_tree:part_2} using $\psi$ is the same as the distribution of the dendritic system in Proposition \ref{prop:evans}. By construction of $\psi$, this is the case on the event where $(\xi_i, i\geq 1)$ satisfies $\xi_i \npreceq \xi_j$ for all $i\neq j$. Because we assumed that $\mu$ is diffusely supported on the leaves of $\bfT$, this happens with probability $1$. This concludes the proof.
\end{proof}

It is important to note that we used the fact that $\mu$ is supported diffusely on the leaves only in the conclusion of the proof but not in the construction of $\psi$. Hence we can also repeat this construction in the general setting of Proposition \ref{prop:evans}. 

\begin{cor}\label{cor:there_is_psi}
In the setting of Proposition \ref{prop:evans}, $F$ induces a deterministic planar order $\psi$ for $\bfT$. In the context of Construction \ref{construction:sample_den_from_tree:part_1}, consider the event that for some finite set $I\subset \NN$ we have for $(\xi_i; i\in I)$ that $\xi_i \npreceq \xi_j$ for $i\neq j$. Then on this event the planarity function $p$ obtained in Construction \ref{construction:sample_den_from_tree:part_2} (which uses this $\psi$) restricted to $I$ has the same distribution as the planarity function of Construction \ref{construction:sample_den_from_tree:part_3} restricted to $I$. 
\end{cor}

As the third special case, we consider the case where $\bfT$ is a single point. Note that in that case $\lambda$ and $B$ are necessarily trivial and there is again only one choice for the planar order $\psi$. 

\begin{lemma}%\label{lemma:case_s_dot}
Assume that from Proposition \ref{prop:evans} we get $(\bfT, d, r, \mu, F)$ so that $\bfT=\{0\}$, $r=0$ and $\mu=\delta_0$. Then there is $(\psi, \lambda, \{\beta_a\}_a)$ so that the dendritic systems obtained by Constructions \ref{construction:sample_den_from_tree:part_2} and \ref{construction:sample_den_from_tree:part_3} have the same distribution. $\psi, \lambda$ and $\{\beta_a\}_a$ need not be specified due to the special structure of $\bfT$.
\end{lemma}

\begin{proof}
The only tree with $n$ leaves that can arise in this case is the tree where all $n$ leaves are attached directly to the root. This tree also has a unique planar order. We are left with distributing the leaf labels, but due to exchangeability they form a uniform permutation on $\{1,\ldots,n\}$. This means we have to assign them in a consistent way which can be realised by $p(i,j)=\II_{U_i<U_j}-\II_{U_j>U_i}$.
\end{proof}

Finally, we will prove Proposition \ref{prop:there_is_dprt_for_dendritic} in the general case by combining the ideas of the preceding lemmas.

\begin{proof}[Proof of Proposition \ref{prop:there_is_dprt_for_dendritic}.] 
Let $(\bfT,d,r,\mu,F)$ as in Proposition \ref{prop:evans}. The proof consists of two steps: first, we construct the planar order $\psi$, the branch weight function $\lambda$ and the branchpoint weight functions $\beta_a$. Then we check that the distribution of the dendritic system obtained through Constructions \ref{construction:sample_den_from_tree:part_1} and \ref{construction:sample_den_from_tree:part_2} is the same as the distribution of the dendritic system obtained through Constructions \ref{construction:sample_den_from_tree:part_1} and \ref{construction:sample_den_from_tree:part_3}. \\

% =================== PLANAR STRUCTURE ===================== 
\textbf{Step $1$: constructing $(\psi,\lambda,\{\beta_a\}_a)$.} First, define $\psi$ by Corollary \ref{cor:there_is_psi}.

% =================== lambda  ============================== 
Next, we construct the branch weight function $\lambda:\bfT \to [0,1]$. For $x \notin \supp(\mu_s)$, set $\lambda(x)=0$. Now fix $x \in \supp(\mu_s)$, i.e.\ $x$ is located on the diffuse mass on the branches. By our convention, we can assume $\deg(x)=2$. Define $\mu^x = \frac{1}{\mu(F_\bfT(x))}\mu\vert_{F_\bfT(x)}$, i.e.\ $\mu$ restricted to $F_\bfT(x)$ and normalised. Let $U_1,U_2,U_3$ be independent, uniform $[0,1]$ random variables and let $\xi$ be an independent $\mu^x$--distributed random variable. $\xi$ can also be seen as a $\mu$--distributed random variable conditioned on $\xi \in F_\bfT(x)$. We set
\begin{equation}\label{eq:q_when_deg_2}
    \lambda(x) = \PP\big( F(x,U_1,\xi, U_2, U_3) = 1 \big).
\end{equation}
Note that this defines a measurable function $\lambda$ because $F$ is measurable. On an informal level, $\lambda(x)$ is the probability that a leaf attached to $x$ will be to the left of the subtree $F_\bfT(x)$ of $x$. \\

% ======================== B  ============================== 
Lastly, fixing an atom $a$ of $\mu$, we define the branchpoint weight function $\beta_a: \{\text{subtrees of }a\} \to [0,1]$. Let $\bfS$ be a subtree of $a$, consider $\mu^\bfS = \frac{1}{\mu(\bfS)}\mu\vert_{\bfS}$, i.e.\ $\mu$ restricted to $\bfS$ and normalised. Let $U_1,U_2,U_3$ be independent, uniform $[0,1]$ random variables and let $\xi$ be an independent $\mu^\bfS$--distributed random variable. $\xi$ can also be seen as $\mu$--distributed random variable conditioned on $\xi \in \bfS$. We set
\begin{equation}\label{eq:q_when_deg_bigger}
    \beta_a(\bfS) = \PP\big( F(a,U_1,\xi, U_2, U_3) = 1 \big).
\end{equation}

Informally, this is the probability that a leaf attached to $a$ is left of the subtree $\bfS_i$. \\

% ==================== checking the distribution v2 =========================== 
\textbf{Step $2$: equivalence in distribution.} 
We need to check that the dendritic system $\mathcal{D}^*= (\NN, \sim^*, \prec^*, p^*)$ obtained by Constructions \ref{construction:sample_den_from_tree:part_1} and \ref{construction:sample_den_from_tree:part_2} using $(\psi,\lambda,\{\beta_a\}_a)$ obtained in steps $1$ has the same distribution as the dendritic system $\mathcal{D} = (\NN, \sim, \prec, p)$ obtained by Constructions \ref{construction:sample_den_from_tree:part_1} and \ref{construction:sample_den_from_tree:part_3}. We consider them under a partial coupling which is obtained by using the same sequence $(\xi_i, i\leq 1)$ of $i.i.d.$ $\mu$--random variables for Construction \ref{construction:sample_den_from_tree:part_1}. This means that $\sim$ is equal to $\sim^*$ and $\preceq$ is equal to $\preceq^*$, $\PP$--almost surely. Condition on $(\xi_i, i\leq n)$, let $\PP^\xi$ be a regular conditional probability of $\PP$ given $(\xi_i, i\geq n)$. It now suffices to check that the restrictions of $\mathcal D$ and $\mathcal D^*$ to the leaves labelled by $[n]$ have the same distribution for all $n\in \NN$. Due to the coupling, it suffices to show that $\{p(i,j): i\neq j \in [n] \}$ and $\{p^*(i,j): i\neq j \in [n] \}$ have the same distribution under $\PP^\xi$, $\PP$--almost surely.

We partition $[n]$: Choose a set $I_1 \subset [n]$ such that for $(\xi_i, i \in I_1)$ we have that $\xi_i \npreceq \xi_j$ for $i\neq j$, and such that $\spn(r, \xi_i, i \in I_1) = \spn(r, \xi_i, i \in [n])$. Next, let 
\begin{equation*}
    I_2 = \left\{i \in [n]: \xi_i \in \supp(\mu_s) \quad \text{and} \quad i \notin I_1 \right\}.
\end{equation*}
Lastly, for every atom $a$ of $\mu$, we let 
\begin{equation*}
    I_3^a = \left\{i \in [n]: \xi_i = a \quad \text{and} \quad i \notin I_1 \right\}.
\end{equation*}

\begin{figure}[h]
    \centering
    \includegraphics[scale=1.1]{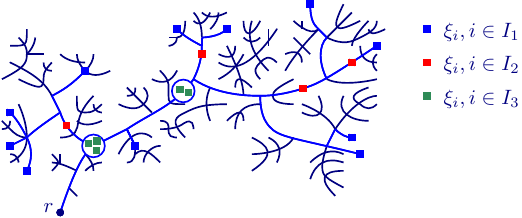}
    \caption{The three sets $I_1,I_2$ and $I_3$. The big, blue circles signify atoms of $\bfT$.}
    \label{fig:p_from_the_leaves}
\end{figure}

By construction, $I_1,I_2$ and $(I_3^a, a \ \text{atom})$ are disjoint and $I_1 \cup I_2 \cup \bigcup_a I_3^a = [n]$. See Figure \ref{fig:p_from_the_leaves} for an illustration of these sets. We show
\begin{equation}\label{eq:goal_step4}
    \{p(i,j): i, j \in [n]; i\neq j\} \overset{d}{=} \{p^*(i,j): i, j \in [n]; i\neq j \} \quad \PP-a.s.
\end{equation}
in three steps. First with $[n]$ replaced by $I_1$, then by $I_1 \cup I_2$ and lastly for $[n]$ itself.
We do these in three steps -- \emph{a,b,c} -- for $I_1$, $I_1 \cup I_2$ and $[n]$ respectively.

\emph{Step $2a$.}
By construction, for $\{ \xi_i, i \in I_1 \}$ we have that $\xi_i \npreceq \xi_j$ for $i\neq j$ and hence 
\begin{equation}\label{eq:step41}
    \{p(i,j): i, j\in I_1; i \neq j \} \overset{d}{=} \{p^*(i,j): i, j\in I_1; i \neq j \} \quad \PP-a.s.
\end{equation}
by Corollary \ref{cor:there_is_psi}. \\

%%%%%%%%% I2 %%%%%%%%%%%
\emph{Step $2b$.}
Consider now $I_2$. For $i \in I_2$ we let $s(i) = \min \{i': i'\in I_1, \xi_i < \xi_{i'} \}$ -- by construction of $I_2$ the set $\{i': i'\in I_1, \xi_i < \xi_{i'} \}$ is never empty for every $i \in I_2$ and thus $s(i)$ is well--defined. 
The idea behind considering $\xi_{s(i)}$ is similar to the proof of Lemma \ref{lemma:case_s_line}, it suffices to consider only one other leaf to determine the orientation of leaf $i$. Here we use $s(i)$, in the lemma we used the smallest leaf above.
Let $i,j \in I_2$ with $i \neq j$. We then have $3$ cases: either $\xi_i < \xi_j$, $\xi_j < \xi_i$ or $\xi_i \nleq \xi_j,\xi_j \nleq \xi_i$. If $\xi_i < \xi_j$, then $(i,j) = (i, s(i)) \prec (j,s(i))$. By $(P4)$ we then almost surely have
\begin{equation}\label{eq:step4:b}
    p(i,j) = p(i,(j,s(i)) = p(i,s(i)).
\end{equation}
Similarly, if $\xi_j < \xi_i$ we almost surely have $p(i,j) = p(s(j),j) = - p(j,s(j))$ where the second equality follows from $(P1)$. If $\xi_i \nleq \xi_j$ and $\xi_j \nleq \xi_i$ we have $(i,j)=(s(i),s(j)) \prec (i,s(i))$ and $(i,j) \prec (j,s(j))$. Hence we almost surely have
\begin{equation}\label{eq:step4:c}
    p(i,j) = p(i,s(j)) = p(s(i),s(j)),
\end{equation}
again by $(P4)$.
The same reasoning works for $p^*$ as well, so that we have analogues of \eqref{eq:step4:b} and \eqref{eq:step4:c} for $p^*$ as well. This implies that to show 
\begin{equation}\label{eq:goal_step4_2}
    \{p(i,j): i, j \in I_1 \cup I_2; i\neq j \} \overset{d}{=} \{p^*(i,j): i, j \in I_1 \cup I_2; i\neq j \} \quad \PP-a.s.,
\end{equation}
it suffices to show that 
\begin{equation}\label{eq:suffices_for_goal_step4_2}
    \{p(i,j):  i,j \in I_1: i \neq j \ \text{or} \ i\in I_2, j = s(i)  \} \overset{d}{=} \{p^*(i,j):  i,j \in I_1: i \neq j \ \text{or} \ i\in I_2, j = s(i)  \} \quad \PP-a.s..
\end{equation}
Consider now $i\in I_2$ and $j=s(i) \in I_1$. Let $\pi(\xi_j)$ be either the branchpoint in $\spn(r,\xi_\ell; \ell \in I_1)$ closest to $\xi_j$ or the closest $\xi_k, k\in [n]$ -- whichever is closer. Denote by $\Tilde{\xi_j}$ the midpoint of the interval $[\pi(\xi_j), \xi_j]$. Let $\bfS_j = F_\bfT(\Tilde{\xi_j})$, by construction $\xi_j \in \bfS_j$ and $\xi_i \notin \bfS_j$. Now let $\zeta_j, \zeta_j'$ be sampled from $\frac{1}{\mu(\bfS_j)}\mu \vert_{\bfS_j}$ -- we assume that they are all independent under $\PP^\xi$ and independent of all uniform variables. Here we note again that $\mu(\bfS_j)>0$, but not necessarily $\mu(F_\bfT(\xi_j))>0$. Let $V_j, V_j', V_{ij}, V_{ij}'$ be additional uniform random variables. We then $\PP^\xi$--almost surely have
\begin{equation*}
    p(i,j) = F(\xi_i, U_i, \xi_j, U_j, U_{ij}) = F(\xi_i, U_i, \zeta_j, V_j, V_{ij}) = F(\xi_i, U_i, \zeta_j', V_j', V_{ij}').
\end{equation*}
The first equality is how $p(i,j)$ is constructed, the other two inequalities follow from $\xi_i, \zeta_i, \zeta_i' \in \bfS_i$ and the consistency properties $(F3)$ and $(F4)$. By Lemma \ref{lemma:app:ind}, this means that there is a function $G_{i}$ such that 
\begin{equation}\label{eq:p(i,s(i))}
    p(i,s(i)) = G_{i}(\xi_i, U_i).
\end{equation}
Note that this is how we have defined $\lambda(\xi_i)$ in \eqref{eq:q_when_deg_2}, $\lambda(\xi_i) = \PP^\xi( G_{i}(\xi_i, U_i) = 1)$. Now let $\gamma = (\gamma_{ij})_{i,j \in I_1 \cup I_2} \in \{ \pm 1 \}^{I_1 \cup I_2}$ be in such a way that 
\begin{equation*}
    \PP^\xi \big( \forall i,j\in I_1 \cup I_2: i\neq j: p(i,j) = \gamma_{ij} \big) >0.
\end{equation*}
Informally, this means we only consider $c$ which does not break the consistency relations of $p$ in an obvious way, for example by not satisfying $\gamma_{ij} = - \gamma_{ji}$. Any such $\gamma$ would have probability $0$, both for the above expression and the same expression with $p$ replaced by $p^*$. Using the observations we have made so far, i.e.\ we only need to show \eqref{eq:suffices_for_goal_step4_2} for \eqref{eq:goal_step4_2} and \eqref{eq:p(i,s(i))}, we get 
\begin{align*}
    \PP^\xi\big( \forall i,j  &\in I_1 \cup I_2; i\neq j: p(i,j) = \gamma_{ij} \big)\\  
    &= \PP^\xi \big( \forall i \in I_2: p(i,s(i)) = \gamma_{i,s(i)}; \forall i ,j \in I_1; i \neq j: p(i,j) = \gamma_{ij}\big) \\
    &= \PP^\xi \big(\forall i \in I_2: p(i,s(i)) = \gamma_{i,s(i)} \big\vert \forall i ,j \in I_1; i \neq j: p(i,j) = \gamma_{ij} \big)  \PP^\xi \big( \forall i ,j \in I_1; i \neq j : p(i,j) = \gamma_{ij} \big)\\
    &=  \PP^\xi \big(\forall i \in I_2:G_{i}(\xi_i,U_i) = \gamma_{i,s(i)} \big\vert \forall i ,j \in I_1; i \neq j: p(i,j) = \gamma_{ij} \big)  \PP^\xi \big( \forall i ,j \in I_1; i \neq j : p(i,j) = \gamma_{ij} \big) \\
    &= \bigg(\prod_{i \in I_2} \PP^\xi \big( G_{i}(\xi_i,U_i) = \gamma_{i,s(i)} \big) \bigg)\PP^\xi \big( \forall i ,j \in I_1; i \neq j: p(i,j) = \gamma_{ij} \big) \\
    &= (*).
\end{align*}
In the last step, we have used the independence of the $U_i,i\in I_2$. As we have noted above, the distribution of $G_{i,s(i)} (\xi_i, U_i)$ is the same as the distribution of $p^*(i,j)$. Combining this with \eqref{eq:step41}, we have
\begin{equation*}
    (*) = \prod_{i \in I_2}\PP^\xi \big( p^*(i,s(i)) = \gamma_{i,s(i)} \big) \PP^\xi \big( \forall i ,j \in I_1; i \neq j: p^*(i,j) = \gamma_{ij} \big) = (**).
\end{equation*}
We reduce this expression with the same reasoning as above, this time for $p^*$ instead of $p$,
\begin{align*}
    (**)
    &= \PP^\xi \big(\forall i \in I_2: p^*(i,s(i)) = \gamma_{i,s(i)} \big\vert \forall i ,j \in I_1; i \neq j: p^*(i,j) = \gamma_{ij} \big)  \PP^\xi \big( \forall i ,j \in I_1; i \neq j : p^*(i,j) = \gamma_{ij} \big)\\
    &= \PP^\xi \big( \forall i \in I_2: p^*(i,s(i)) = \gamma_{ij} ; \forall i ,j \in I_1; i \neq j: p^*(i,j) = \gamma_{ij}) \\
    & = \PP^\xi \big( \forall i,j \in I_1 \cup I_2; i\neq j: p^*(i,j) = \gamma_{ij} \big).
\end{align*}
Because $\gamma$ was arbitrary, we have shown \eqref{eq:goal_step4_2}. \\

\emph{Step $2c$.}
Lastly, we want to show \eqref{eq:goal_step4}, given \eqref{eq:goal_step4_2}.
Fix an atom $a$ such that $I_3^a \neq \emptyset$. Note there can be the case where there is $i\in I_1$ with $\xi_i = a$. To deal with this case and to include this index, we define $\Tilde{I}_3^a = \{i: \xi_i = a \}$. Further, there are $d(a) \geq 1$ and $i_1^a,\ldots,i_{d(a)}^a \in I_1$ such that $a \leq \xi_{i_j^a}$ for $1 \leq j \leq d(a)$ and such that $\spn(r,\xi_i; i \in [n])$ and $\spn(r,\xi_{i_1^a}, \ldots, \xi_{i^a_{d(a)}})$ are the same in a small neighbourhood of $a$.
This means that we choose as many of the leaves in $\spn(r,\xi_i; i \in I_1)$ that sit above $a$ as needed to realise the degree of $a$ in $\spn(r,\xi_i; i \in [n])$. The case of $d(a)=1$ can also happen if there is $i \in I_1$ with $\xi_i =a$. We can choose them in such a way that $p(i^a_k,i^a_{k+1}) = 1$ for $k\leq d(a)-1$, i.e.\ they are indexed from left to right in an increasing manner. 
Let now be $i \in \Tilde{I}_3^a $ and $j \in [n] $, i.e.\ a leaf which is attached to $a$ and another leaf. We consider $p(i,j)$. There are three cases, $\xi_j < a$, or $a < \xi_j$, or $a \nless \xi_j$ and $\xi_j \nless a$. We show that in all three cases we have 
\begin{equation}\label{eq:step4:3:a}
    p(i,j) = p(i^a_k,j),
\end{equation}
$\PP^\xi$--almost surely for an appropriate choice of $1 \leq k \leq d(a)$. If $\xi_j < a$ or if $a \nless \xi_j, \xi_j \nless a$, we choose $i_k^a = i_1^a$ -- \eqref{eq:step4:3:a} then holds by the consistency property $(P4)$ of $p$. If $a< \xi_j$, then there exists some $k$ such that $\xi_{j}$ and $\xi_{i^a_k}$ are in the same subtree of $a$. \eqref{eq:step4:3:a} again holds by $(P4)$.
This means that for an appropriate $\gamma = (\gamma_{ij})_{i\neq j \in [n]} \in \{\pm 1 \}^{n(n-1)}$ that does not violate the consistency conditions we have
\begin{align}
    \PP^\xi \big( \forall i,j \in [n], i\neq j: p(i,j) &= \gamma_{ij} \big)  \nonumber  
    =\PP^\xi \big( \forall i,j \in I_1 \cup I_2, i\neq j: p(i,j) = \gamma_{ij}; \\& \forall a \forall i \in \Tilde{I}_3^a \forall k \leq d(a): p(i,i^a_k) = \gamma_{i,i_k^a}; \forall a \forall i ,j\in \Tilde{I}_3^a , i \neq j: p(i,j) = \gamma_{ij} \big). \label{eq:step4:3:loooong}
\end{align}
The same statement holds if we replace $p$ by $p^*$. Note that in the case where there is $i \in I_1$ with $\xi_i = a$ there is some redundancy in the above statement. There we have $d(a)=1$ and $i_1^a \in \Tilde{I}_3^a$. 

Consider the case where $a < \xi_{i_k^a}$ for all $k\leq d(a)$ (this is true for all $k$ unless $d(a)=1$ and $\xi_{i_1^a}=a$). Then every $k$ corresponds to a subtree $\bfS_k$ of $a$. We necessarily have $\mu(\bfS_k)>0$ for every $1\leq k \leq d(a)$. Let $\zeta_k, \zeta_k'$ be sampled from $\frac{1}{\mu(S_k)} \mu \vert_{S_k}$ and let $V_k,V_k',V_{i k}, V_{i k}'$ be uniform random variables, all of them are assumed to be independent under $\PP^\xi$ and independent of any other uniform random variables. Let $i \in I_3^a$, we then $\PP^\xi$--almost surely have
\begin{equation*}
    p(i,i_k^a) = F(a, U_i, \xi_{i_k^a}, U_{i_k^a}, U_{i, i_k^a}) = F(a,U_i, \zeta_k, V_k, V_{i k}) = F(a,U_i, \zeta_k', V_k', V_{i k}').
\end{equation*}
The first equality is Construction \ref{construction:sample_den_from_tree:part_3}, the latter two equalities follow from the consistency properties $(F3)$ and $(F4)$. By Lemma \ref{lemma:app:ind} we get that there is a function $G^a_k$ such that $\PP^\xi$--almost surely
\begin{equation*}%\label{eq:step4:3:b}
    p(i,i_k^a) = G^a_k( U_i).
\end{equation*} 
Using this, we continue the considerations of \eqref{eq:step4:3:loooong}. We restrict ourselves to the case where $I_3^a = \Tilde{I}_3^a$.
\begin{align}
    \PP^\xi \big( \forall i,j \in [n],& i\neq j: p(i,j) = \gamma_{ij} \big)  \nonumber  \\
    &=\PP^\xi \big( \forall i,j \in I_1 \cup I_2, i\neq j: p(i,j) = \gamma_{ij}; \forall a \forall i \in I_3^a  \forall k \leq d(a): p(i,i^a_k) = \gamma_{i,i_k^a}; \nonumber \\
    & \quad \quad  \forall a \forall i ,j\in I_3^a , i \neq j: p(i,j) = \gamma_{ij} \big) \nonumber\\
    &= \PP^\xi \big( \forall a \forall i \in I_3^a  \forall k \leq d(a): p(i,i^a_k) = \gamma_{i,i_k^a};\forall a \forall i ,j\in I_3^a , i \neq j: p(i,j) = \gamma_{ij} \big\vert \nonumber\\
    &\quad \quad   \forall i,j \in I_1 \cup I_2, i\neq j: p(i,j) = \gamma_{ij} \big) \PP^\xi\big(  \forall i,j \in I_1 \cup I_2, i\neq j: p(i,j) = \gamma_{ij} \big) \nonumber\\
    &= \prod_a \PP^\xi \big( \forall i \in I_3^a \forall k \leq d(a):G_k^a(U_i) = \gamma_{i,i_k^a}; \forall i ,j\in I_3^a , i \neq j: F(a, U_i, a, U_j, U_{ij} )= \gamma_{ij}\big) \nonumber\\
    &\quad\quad \cdot \PP^\xi\big(  \forall i,j \in I_1 \cup I_2, i\neq j: p(i,j) = \gamma_{ij} \big) \label{eq:step4:3:d}
\end{align}
For the last equality, we have used that due to \eqref{eq:goal_step4}, the uniform variables used at different atoms are independent because for $a\neq a'$ the sets $I_3^a$ and $I_3^{a'}$ are disjoint. Consider now 
\begin{equation*}
    \PP^\xi \big( \forall i \in I_3^a \forall k \leq d(a):G_k^a(U_i) = \gamma_{i,i_k^a}; \forall i ,j\in I_3^a , i \neq j: F(a, U_i, a, U_j, U_{ij} )= \gamma_{ij}\big),
\end{equation*}
for a fixed atom $a$.

Note that we have defined the thresholds of $\beta_a$ exactly so that $\PP(G_k^a(U_i) = 1) = \beta_a(\bfS_k)$, compare to \eqref{eq:q_when_deg_bigger}. 
% \textcolor{violet}{M: b_k depends on a}
For every $i \in I_3^a$ there is a unique $k(i)$ such that $a_{k i} = 1$ for $k\leq k(i)$ and $a_{k i}= -1$ for $k>k(i)$. The possible values for $k(i)$ reach from $0$ to $d(a)$ where $0$ and $d(a)$ correspond to the extremal cases where the leaf $i$ is to the left or to the right of all subtrees. Formally define $\beta_a(\bfS_{d(a)+1})=1$. This implies that 
\begin{align}
    &\PP^\xi \big( \forall i \in I_3^a \forall k \leq d(a):G_k^a(U_i) = \gamma_{i,i_k^a}; \forall i ,j\in   I_3^a, i \neq j: F(a, U_i, a, U_j, U_{ij} )= \gamma_{ij}\big)\nonumber \\
    &= \bigg(\prod_{i \in I_3^a} \beta_a(\bfS_{k(i)+1})-\beta_a(\bfS_{k(i)})\bigg) \nonumber \\
    & \hspace{3cm} \times \PP^\xi \big(\forall i ,j\in I_3^a, i \neq j: F(a, U_i, a, U_j, U_{ij} )= \gamma_{ij} \big\vert \forall i \in I_3^a \forall k \leq d(a):G_k^a(U_i) = \gamma_{i,i_k^a}\big) \nonumber\\
    &= \bigg(\prod_{i \in I_3^a}\beta_a(\bfS_{k(i)+1})-\beta_a(\bfS_{k(i)})\bigg) \bigg(\prod_{k=0}^{d(a)} \frac{1}{\vert \{ i\in I_3^a: k(i)=k \}\vert !} \bigg) \nonumber %\label{eq:step4:3:c}
\end{align}
The last equality holds because the leaves with indices in the set $\{ i\in I_3^a: k(i)=k \}$ form a uniform random permutation by exchangeability. We have chosen Construction \ref{construction:sample_den_from_tree:part_2} in such a way that
\begin{align}
     \bigg(\prod_{i \in I_3^a}\beta_a(\bfS_{k(i)+1})-\beta_a(\bfS_{k(i)})\bigg) & \bigg(\prod_{k=0}^{d(a)} \frac{1}{\vert \{ i\in I_3^a: k(i)=k \}\vert !} \bigg) \nonumber \\
     &=  \PP^\xi \big( \forall i \in I_3^a \forall k \leq d(a):p^*(i,i_k^a)=\gamma_{i,i_k^a}; \forall i ,j\in I_3^a, i \neq j:p^*(i,j)= \gamma_{ij}\big). \label{eq:step4:last_eq_I_hope}
\end{align}

Before \eqref{eq:step4:3:d} we assumed $I_3^a = \Tilde{I}_3^a$. Consider now the case where $I_3^a \neq \Tilde{I}_3^a$ which happens when there is $i\in I_1$ with $\xi_i = a$. In this case the computations \eqref{eq:step4:3:d} -- \eqref{eq:step4:last_eq_I_hope} become easier. The reason for this is that we do not need to consider the terms of the form
\begin{equation*}
    \left\{ \forall i \in I_3^a  \forall k \leq d(a): p(i,i^a_k) = \gamma_{i,i_k^a} \right\}.
\end{equation*}
This is because $d(a)=1$ and $\xi_{i^a_1} \in \Tilde{I}_3^a$. Besides that we consider $\Tilde{I}_3^a$ instead of $I_3^a$. The computations then proceed as above. \\

Recall that we have already shown \eqref{eq:goal_step4_2}, this means that
\begin{equation}\label{eq:step2:instep3}
    \PP^\xi \big( \forall i\neq j  \in I_1 \cup I_2: p(i,j) = \gamma_{ij} \big) = \PP^\xi \big( \forall i\neq j \in I_1 \cup I_2: p^*(i,j) = \gamma_{ij} \big). 
\end{equation}
Going back to \eqref{eq:step4:3:d}, with \eqref{eq:step2:instep3} we have that
\begin{align*}
    \PP^\xi \big( \forall i,j \in [n], & i\neq j: p(i,j) = \gamma_{ij} \big) \nonumber \\ &= \PP^\xi \big( \forall i,j \in I_1 \cup I_2, i\neq j: p^*(i,j) = \gamma_{ij}; \forall a \forall i \in I_3^a \forall k \leq d(a): p^*(i,i^a_k) = \gamma_{i,i_k^a}; \nonumber\\
    & \quad \quad  \forall a \forall i ,j \in I_3^a, i \neq j: p^*(i,j) = \gamma_{ij} \big) \nonumber \\
    & =\PP^\xi \big( \forall i,j \in [n], i\neq j: p^*(i,j) = \gamma_{ij} \big). 
\end{align*}
This completes showing that $p$ and $p^*$ have the same distribution under $\PP^\xi$ which completes the proof. 
\end{proof}

\subsection{Uniqueness of a canonical representation}\label{sec:uniqueness}

So far we have proven Proposition \ref{prop:there_is_dprt_for_dendritic} which states that there is \emph{some} choice of $(\bfT,d,r,\mu,\psi,\lambda,\{\beta_a\}_a)$ that corresponds to our dendritic system.
In this section we want to find a more canonical representation for this in the form of IP--trees, see Definition \ref{def:ip-tree}. This will lead to us proving Proposition \ref{prop:let_it_be_IP}.
The notion of IP--trees has been introduced by Forman \cite{forman_exchangeable_2020}. 

\begin{definition}[Special points] For a weighted, rooted real tree $(\bfT, d, r, \mu)$ the special points are
\begin{enumerate}
    \item the locations of atoms of $\mu$,
    \item the branch points of $\bfT$, and
    \item the isolated leaves of $\spn(\supp(\mu))$, by which we mean leaves of $\spn(\supp(\mu))$ that are not limit points of the branch points of $\spn(\supp(\mu))$. 
\end{enumerate}

\end{definition}

\begin{definition}[mass--structural isomorphism] \label{def:msi}
Let $\mathscr{S}_i$ be the sets of special points of weighted, rooted real trees $(\bfT_i, d_i, r_i, \mu_i)$ for $i=1,2$. A measurable map $\phi:\bfT_1 \rightarrow\bfT_2$ is a \textit{mass--structural isomorphism} if it has the following properties.

\begin{enumerate}
    \item \textit{Mass preserving.} For every $x\in \mathscr{S}_1$, $\mu_1([r_1,x]_{\bfT_1})=\mu_2([r_2,\phi(x)]_{\bfT_2})$, $\mu_1(\{x\})=\mu_2(\{\phi(x)\})$, and $\mu_1(F_{\bfT_1}(x))=\mu_2(F_{\bfT_2}(\phi(x)).$
    
    \item \textit{Structure preserving.} For $x,y\in \mathscr{S}_1$ we have $x\in[r_1, y]_{\bfT_1}$ if and only if $\phi(x)\in [r_2,\phi(y)]_{\bfT_2}.$
\end{enumerate}
\end{definition}

We call two rooted, weighted real trees mass--structurally equivalent if there exists a mass--structural isomorphism between the two. This is an equivalence relation. We then have the following two theorems of Forman \cite{forman_exchangeable_2020}, the second one concern itself with hierarchies. A hierarchy on $\NN$ $(\mathcal{H}_n,n\geq 1)$ \cite[Definition 1.6]{forman_exchangeable_2020} is an object such that for every $n\geq 1$, $\mathcal{H}_n$ is a collection of subsets of $[n]$ satisfying certain consistency assumptions -- we do not recall these here. To every IP--tree $(\bfT,d,r,\mu)$ we associate a hierarchy, $(\xi_i, i\geq 1)$ are $i.i.d.$ $\mu$--random variables,  
\begin{equation}\label{eq:induced_hierachies}
    \mathcal{H}_n = \left\{\left\{i \in [n]: \xi_i \in F_\bfT(x)\right\}: x\in \bfT\right\} \cup \left\{\left\{i \right\}: i \in [n]\right\} \quad \quad \text{for } n\geq 1. 
\end{equation}
Observe that this is very similar to the first two steps of Construction \ref{construction:sampling_the_bridge}. For a given $n$, $\mathcal{H}_n$ as above can be represented as a discrete tree, therefore we can think of a hierarchy  $(\mathcal{H}_n,n\geq 1)$ as a sequence of growing trees. 

\begin{theorem}\cite[Theorem 1.5]{forman_exchangeable_2020} \label{thm:noah_1}
Each mass--structural equivalence--class of rooted, weighted real trees contains exactly one isomorphism class of IP--trees.
\end{theorem}

\begin{theorem}\cite[Theorem 1.7]{forman_exchangeable_2020} \label{thm:noah_2}
Two IP--trees are mass--structurally equivalent if and only if the induced hierarchies in \eqref{eq:induced_hierachies} have the same law.
\end{theorem}

Before we can apply this to our setting, we make sure that we can also pass the planar order $\psi$, the branch weight function $\lambda$ and the branchpoint weight function $B$ through a mass--structural isomorphism.

\begin{lemma}\label{lemma:msi_induces_maps}
A mass--structural isomorphism $\phi$ induces a new planar order $\phi(\psi)$, a new branch weight function $\phi(\lambda)$ and a new branchpoint weight functions $\phi(\beta_a)$.
\end{lemma}

\begin{proof}
Assume we have $\phi:(\bfT, d, r, \mu) \rightarrow (\bfT', d', r', \mu')$ and that $\psi$, $\lambda$ and $B$ are a planar order, branch weight function and branchpoint weight function for $(\bfT, d, r, \mu)$. 

For a sequence $x_1',\ldots,x_n'\in \bfT'$ with $x_i \npreceq x_j$ for $i\neq j$, we define $$\phi(\psi)_n(x_1',\ldots,x_n')=\psi_n(\phi^{-1}(x_1'),\ldots,\phi^{-1}(x_n')).$$ Because $\phi$ is structure preserving in the sense of Definition \ref{def:msi} we obtain a sequence $\phi^{-1}(x_1'),\ldots,\phi^{-1}(x_n')$ with the property that $\phi^{-1}(x_i') \npreceq \phi^{-1}(x_j')$ for $i\neq j$. The same property and the fact that $\psi$ is a planar order also implies that we can embed $\phi(\psi)_m(x_1',\ldots,x_m')$ into $\phi(\psi)_n(x_1',\ldots,x_n')$ for $m<n$ respecting the planar structure.

For any $x' \in \bfT'$, define $\phi(\lambda)(x')=\lambda(\phi^{-1}(x'))$ which is again a branch weight function. Similarly, if $a' \in \bfT'$ is an atom of $\mu'$, then $a= \phi^{-1}(a')$ is an atom of $\mu$ because $\phi$ is mass--preserving. We can then define $\beta_{a'}=\beta_a$. Because $\phi$ is structure--preserving, the subtrees of $a'$ are in one-to-one correspondence with the subtrees of $a$.
\end{proof}

With Lemma \ref{lemma:msi_induces_maps} in hand, we can prove Proposition \ref{prop:let_it_be_IP}.

\begin{proof}[Proof of Proposition \ref{prop:let_it_be_IP}]
Assume we have $(\bfT_i,d_i,r_i,\mu_i,\psi^{(i)},\lambda_i,\{\beta_a^i\}_a), i\in \{1,2\}$ such that the induced Markov chains $(T_n^{(i)},n\geq 1), i \in \{1,2\}$ have the same distribution. We show the uniqueness of the associated tuple $(\bfT, d, r, \mu, \psi, \lambda, \{\beta_a\}_a)$ in multiple steps. 

\emph{Uniqueness of $(\bfT, d, r, \mu)$:}
Observe that applying a mass--structural isomorphism using the induced maps of Lemma \ref{lemma:msi_induces_maps} does not change the distribution of the sampled dendritic system: More precisely, assume we are given a mass--structural isomorphism $\phi:(\bfT_1, d_1, r_1, \mu_1)\rightarrow (\bfT_2, d_2, r_2, \mu_2)$ and a planar order $\psi$, branch weight function $\lambda$ and branchpoint weight function $B$ for $\bfT_1$. Sample $\{\xi_i\}_{i\in \NN}$ independently from $\mu_1$ in $\bfT_1$, then $\{\phi(\xi_i)\}_{i\in \NN}$ is an $i.i.d.$--$\mu_2$ sequence. Using these random variables and the same sequence of independent uniform random variables $\{U_i\}_{i \in \NN}$ we can construct two dendritic systems $\mathcal{D}_1=(\NN, \sim_1, \prec_1, p_1)$ and $\mathcal{D}_2=(\NN, \sim_2, \prec_2, p_2)$ via Construction \ref{construction:sampling_the_bridge}. Because $\phi$ is structure preserving, $\sim_1$ and $\sim_2$, respectively $\prec_1$ and $\prec_2$, are almost surely the same. Further, because we defined $\phi(\psi)$, $\phi(\lambda)$ and $\phi(B)$ by pullback, $p_1$ and $p_2$ are almost surely the same. In particular, the distribution of $\mathcal{D}_1$ and $\mathcal{D}_2$ is identical.  

On the other hand, observe that if in Construction \ref{construction:sampling_the_bridge} we do not add planarity to $T_n$ and keep the leaf labels, we retrieve the hierarchy given by \eqref{eq:induced_hierachies}. Now, by Theorem \ref{thm:noah_2} and because the induced Markov chains $(T_n^{(i)},n\geq 1), i \in \{1,2\}$ have the same distribution, the trees  $(\bfT_1, d_1, r_1, \mu_1)$ and $ (\bfT_2, d_2, r_2, \mu_2)$ are mass--structurally isomorphic. 
Having also shown that the distribution of a dendritic system is invariant under mass--structural isomorphism of the decorated planar real tree, this and Theorem \ref{thm:noah_1} then yield the desired uniqueness of $(\bfT, d, r, \mu)$.

\emph{Uniqueness of $\psi$:} 
Assume now that $(\bfT, d, r, \mu)$ is fixed and that we are given two planar orders $\psi^{(1)}$ and $\psi^{(2)}$ with the distribution of the Markov chain being the same. In particular, we assume that the distributions of $(\psi_n^{(1)}(\xi_1, \ldots, \xi_n), n\geq 2)$ and $(\psi_n^{(2)}(\zeta_1, \ldots, \zeta_n), n\geq 2)$ are the same where $(\xi_i,i\geq 1)$ and $(\zeta_i,i\geq 1)$ are $i.i.d.$ $\mu$--random variables. We will show that there is an isometry $\varphi: \bfT \to \bfT$ such that $\varphi(\psi^{(1)}) = \psi^{(2)}$.

By \cite[Theorem 3.4 (i)]{kallenberg_foundations_2021} there exists a kernel $K_1$ such that for appropriate events $A,B$ we have
\begin{equation*}
    \PP\left( (\xi_i, i\geq 1) \in A,(\psi_n^{(1)}(\xi_1, \ldots, \xi_n), n\geq 2) \in B \right) = \int_B K_1(S,A)\PP\left( (\psi_n^{(1)}(\xi_1, \ldots, \xi_n), n\geq 2) \in dS \right).
\end{equation*}
The same is true for $\psi^{(2)}$ with another kernel $K_2$. This means that
we can work on a probability space such that $\psi_n^{(1)}(\xi_1, \ldots, \xi_n) = \psi_n^{(2)}(\zeta_1, \ldots, \zeta_n)$ for all $n\geq 2$ while keeping the joint distribution of $(\xi_i, i \geq 1)$ and $(\psi^{(1)}_n(\xi_1, \ldots, \xi_n),n\geq 2)$ the same. Abbreviate $S_n = \psi_n^{(1)}(\xi_1, \ldots, \xi_n)$. 
We use this to define a map $\varphi:\bfT \to \bfT$. First, for all $i\geq 1$ we define $\varphi(\xi_i) = \zeta_i$. For any $i\geq 1$ and $n\geq i$, $\xi_i$ and $\varphi(\xi_i)$ correspond to the same vertex $x^n_i$ in $S_n$. Next, let  $\xi_i \wedge \xi_j$ be the most recent common ancestor of $\xi_i$ and $\xi_j$ and similarly let $x_i^n \wedge x_j^n$ be the most recent common ancestor of $x_i^n$ and $x_j^n$. Define $\varphi(\xi_i \wedge \xi_j) = \varphi(\xi_i) \wedge \varphi(\xi_j)$, both $\xi_i \wedge \xi_j$ and $\varphi(\xi_i \wedge \xi_j)$ correspond to $x_i^n \wedge x_j^n$ in $S_n$. Observe that for $i,j,k,\ell \in \NN$ if $\xi_i\wedge \xi_j = \xi_k \wedge \xi_\ell$ then $\varphi(\xi_i) \wedge \varphi(\xi_j)=\varphi(\xi_k) \wedge \varphi(\xi_\ell)$, hence $\varphi(\xi_i \wedge \xi_j)$ is well defined. This defines $\varphi$ on $\{\xi_i, i\geq 1\}$ as well as all branchpoints of $\bfT$. Let $\mu_n = \sum_{i=1}^n \delta_{x_i^n}$ on $S_n$ and we observe that $\lim_{n \to \infty} \frac{1}{n} \mu_n ( F_{S_n}(x_i^n)) = \mu(F_\bfT(\xi_i))$, almost--surely by the strong law of large numbers applied to $\{\xi_j,j>i\}$. 
Similarly, $\lim_{n \to \infty} \frac{1}{n} \mu_n ( F_{S_n}(x_i^n\wedge x_j^n)) = \mu(F_\bfT(\xi_i\wedge \xi_j))$ for all $i$ and $j$. This allows us to show that $\varphi$ restricted to $\{\xi_i, i\geq 1\}$ is an isometry, we use the IP--spacing \eqref{eq:IP_spacing},
\begin{align*}
    d(\varphi(\xi_{i}), \varphi(\xi_{j})) &=  \left\vert \mu(F_{\bfT_2}(\varphi(\xi_{i} \wedge \xi_{j}))) - \mu(F_{\bfT}(\varphi(\xi_{i})) ) \right\vert +  \left\vert \mu(F_{\bfT}(\varphi(\xi_{i} \wedge \xi_{j}))) - \mu(F_{\bfT}(\varphi(\xi_{j}) )) \right\vert \\
    & = \lim_{n \to \infty} 
    \frac{1}{n} \left\vert \mu_n(F_{S_n}(x_i^{n}\wedge x_j^n)) - \mu_n(F_{S_n}(x_i^n) ) \right\vert + \lim_{n \to \infty} 
    \frac{1}{n} \left\vert \mu_n(F_{S_n}(x_i^n\wedge x_j^n)) - \mu_n(F_{S_n}(x_j^n) ) \right\vert\\
    &= \left\vert \mu(F_{\bfT}(\xi_{i} \wedge \xi_{j})) - \mu(F_{\bfT}(\xi_{i}) ) \right\vert +  \left\vert \mu(F_{\bfT_1}(\xi_{i} \wedge \xi_{j})) - \mu(F_{\bfT}(\xi_{j}) ) \right\vert \\
    &= d(\xi_{i}, \xi_{j}).
\end{align*}

The same is true for branchpoints.
In particular, this means that $\varphi$ maps Cauchy--sequences to Cauchy--sequences. Hence, assume that for $y\in \supp  \mu$ there is a sequence $y_k,k\geq1$ with $\lim_{k \to \infty} y_k = y$ and for every $k$ we have either $y_k \in \{ \xi_i, i\geq 1\}$ or $y_k$ is a branchpoint in $\bfT$. We then define $\varphi(y) = \lim_{k \to \infty} \varphi(y_k)$. Due to the aforementioned properties of $\varphi$ and because $\bfT$ is a complete metric space, this limit exists and is well--defined, i.e.\ does not depend on the choice of sequence $(y_k)_k$.

The map $\varphi$ can be extended to an isometry. Indeed, because $\bfT$ is an IP--tree, it suffices to show that $\varphi$ restricted to special points ($\supp  \mu$ and branchpoints) is a mass--structural isomorphism. Theorem \ref{thm:noah_1} then tells us that there is an isometry, and by checking the proof in \cite{forman_exchangeable_2020} we can see that this isometry is an extension of the underlying mass--structural isomorphism between special points. Let us now check that $\varphi$ is a mass--structural isomorphism. Clearly, $\varphi$ is structure preserving because $\psi^{(1)}_n(x_1, \ldots, x_n)$ corresponds to $\spn(x_1, \ldots, x_n)$ as combinatorial trees. Further, $\varphi$ is mass preserving: consider $z \in \bfT$, both $z$ and $\varphi(z)$ correspond to the same point in $S_n$, call it $z_n$. We then have
\begin{equation*}
    \mu(F_{\bfT}(z)) = \lim_{n \to \infty} \frac{1}{n} \mu_n(F_{S_n}(z_n)) = \mu(F_{\bfT}(\varphi(z)), \quad \quad \text{almost--surely},
\end{equation*}
where we applied the strong law of large numbers twice. The same approach works for atoms and segments. Hence the $\varphi$ is a mass--structural isomorphism and thus can be extended to an isometry on the whole tree $\bfT$.

Next, we show that $\varphi(\psi^{(1)}) = \psi^{(2)}$. For this, let $n\geq 2$ and let $x_1, \ldots, x_n  \in \{\zeta_i, i\geq 1\} \cup \{\text{branchpoints}\}$ and therefore also $\varphi^{-1}(x_1), \ldots, \varphi^{-1}(x_n) \in \{\xi_i,i\geq 1\}\cup \{\text{branchpoints}\}$. 
Observe that for $N$ large enough $\varphi(\psi^{(1)})_n(x_1, \ldots, x_n)$ and $\psi^{(2)}_n(x_1, \ldots, x_n)$ are subtrees of $S_N$. Moreover due to the coupling they are the same, i.e.\ $\varphi(\psi^{(1)})_n(x_1, \ldots, x_n) = \psi^{(2)}_n(x_1, \ldots, x_n)$. This can be extended to $x_1,\ldots, x_n \in \supp  \mu$ by density of $\{\zeta_i, i\geq 1\}$. Because $\spn(\supp(\mu)) = \bfT$ this is also true for all $x_1,\ldots,x_n \in \bfT$. Indeed, it suffices to specify $\varphi(\psi^{(1)})_n(x_1, \ldots, x_n)$ for $x_1, \ldots, x_n$ with $x_i \npreceq x_j$ for $i\neq j$. If for some $i\in [n]$, $x_i \notin \supp(\mu)$, then we can choose any leaf $x_i'$ with $x_i < x_i'$ to obtain $\varphi(\psi^{(1)})_n(x_1, \ldots ,x_i, \ldots, x_n) = \varphi(\psi^{(1)})_n(x_1, \ldots ,x_i', \ldots, x_n)$. Because all leaves are in the support of $\mu$, this determines $\varphi(\psi^{(1)})_n(x_1, \ldots, x_n)$ -- the same argument works for $\psi_n^{(2)}$. Hence we have shown that $\varphi(\psi^{(1)})_n = \psi^{(2)}_n$, doing this for all $n$ shows that  $\varphi(\psi^{(1)}) = \psi^{(2)}$. This completes showing the uniqueness of $\psi$.

\emph{Uniqueness of $\lambda$:} Assume now that $(\bfT, d, r, \mu, \psi)$ is fixed and we are given two different branch weight functions $\lambda^{(1)}, \lambda^{(2)} \in L^1(\mu_s)$. Let $T_n^{(1)}$ and $T_n^{(2)}$ be the trees obtained from using $\lambda^{(1)}$ and $\lambda^{(2)}$ respectively while sampling from $(\bfT, d, r, \mu)$. There exists a segment $\llbracket x,y \rrbracket \subset \bfT$ such that $\int_{\llbracket x,y \rrbracket}\lambda^{(1)} d\mu_s \neq \int_{\llbracket x,y \rrbracket}\lambda^{(2)} d\mu_s$. The segment $\llbracket x,y \rrbracket \subset \bfT$ corresponds to segments $\llbracket x_n^{(1)}, y_n^{(1)} \rrbracket$ and $\llbracket x_n^{(2)}, y_n^{(2)}\rrbracket$ in $T_n^{(1)}$ and $T_n^{(2)}$ respectively. Let $L_n^{(1)}$ be the proportion of leaves directly attached to the left of $\llbracket x_n^{(1)}, y_n^{(1)}\rrbracket$ -- here we only count vertices of degree $2$ in $\bfT$ to avoid counting atoms. Define $L_n^{(2)}$ similarly. By the strong law of large numbers, we almost surely have as $n\to \infty$
\begin{equation*}
    L_n^{(1)} \longrightarrow \int_{\llbracket x,y \rrbracket}\lambda^{(1)} d\mu_s \quad \quad  \text{and} \quad \quad L_n^{(2)} \longrightarrow \int_{\llbracket x,y \rrbracket}\lambda^{(2)} d\mu_s.
\end{equation*}
By assumption, these two integrals are different and thus the distributions of $\big(T_n^{(1)} , n\geq 1\big)$ and $\big(T_n^{(2)} , n\geq 1\big)$ are different. This shows the uniqueness of $\lambda$.

\emph{Uniqueness of $B$:} Assume now that $(\bfT, d, r, \mu, \psi)$ is fixed and we are given two different branchpoint weight functions $\{\beta_a^1\}_a, \{\beta_a^2\}_a$. Then there exists an atom $a$ and a subtree $\bfS$ of $a$ such that $\beta_a^1(\bfS) \neq \beta_a^2(\bfS)$.
 Let $T_n^{(1)}$ and $T_n^{(2)}$ be the trees obtained from using $\{\beta^{(1)}_a\}_a$ and $\{\beta_a^{(2)}\}$ respectively. For $n$ sufficiently large, the atom $a$ corresponds to $a_n^{(1)} \in T_n^{(1)}$ and to $a_n^{(2)} \in T_n^{(2)}$ respectively, similarly $\bfS$ corresponds to subtrees $S_n^{(1)},S_n^{(2)}$ of $a_n^{(1)}$ and $a_n^{(2)}$ respectively. Let $K_n^{(1)}$ be the proportion of leaves directly attached to $a_n^{(1)}$ on the left of $S_n^{(1)}$, as compared to the right of $S_n^{(1)}$. Define $K_n^{(2)}$ similarly. By the strong law of large numbers, we almost surely have as $n\to \infty$
\begin{equation*}
    K_n^{(1)} \longrightarrow \beta_a^{1}(\bfS) \quad \quad  \text{and} \quad \quad K_n^{(2)} \longrightarrow \beta_a^{2}(\bfS).
\end{equation*}
By assumption, these two values are different and thus the distributions of $\big(T_n^{(1)} , n\geq 1\big)$ and $\big(T_n^{(2)} , n\geq 1\big)$ are different. This shows the uniqueness of $\{\beta_a\}_a$.
\end{proof}

%%%%%%% SCALING LIMITS %%%%%%%%%%%
\section{Scaling limits}\label{sec:scaling_limits}

In the following, let $(T_n, n\geq 1)$ be an extremal tree--valued Markov chain with uniform backward dynamics obtained from $(\bfT, d, r, \mu, \psi, \lambda, \{\beta_a\}_a)$ through Construction \ref{construction:sampling_the_bridge} where $(\bfT, d, r, \mu)$ is an IP-tree. The goal of this section is to show that $T_n$ -- trimmed and appropriately rescaled -- converges to $\bfT$ almost surely in the Gromov--Prokhorov topology. Recall the rescaling and trimming from Definition \ref{def:trimming_rescaling}, and the Gromov-Prokhorov metric from Definition \ref{def:GP_in_introduction}.

\begin{remark}
    One might ask why it is necessary to trim $T_n$ before rescaling it. Consider the decorated planar real tree that is made up from a single atom $a$ of weight $1$, here $d,\psi,\lambda,\{\beta_a\}_a$ are all trivial. For any $n\geq 2$ the tree $T_n$ is a star tree with $n$ leaves directly connected to the root. In the IP--rescaling \eqref{eq:ip_rescaling}, all these edges have length $1-1/n$. From Definition \ref{def:GP_in_introduction} we can see that $\dGP(\bfT,T_n) = 1$ for all $n\geq 2$, hence we have no convergence. This problem is solved by trimming.
\end{remark}

An important idea in the proof will be that $T_n^{\mathrm{trim}}$ can be viewed as a subtree of $\bfT$. Recall that $T_n$ is obtained by sampling $(\xi_1, \ldots, \xi_n)$ from $\bfT$ and that the tree structure of $T_n$ is obtained from $\spn(r,\xi_1, \ldots, \xi_n)$ plus additional leaves added through $\lambda$ and $\{\beta_a\}_a$. The trimming removes all additional leaves but also leaves which were not added through $\lambda$ and $\{\beta_a\}_a$. Let us define a function $\eta^n: \bfT^n \to \bfT^n$ which corresponds to trimming on the level of real trees. First, consider the set of all most recent common ancestors of $\xi_1, \ldots, \xi_n$:
\begin{equation*}
    M_n = \left\{ \xi_i \wedge \xi_j, 1 \leq i \neq j \leq n \right\}.
\end{equation*}
where $\xi_i \wedge \xi_j$ is the most recent common ancestor of $\xi_i$ and $\xi_j$. We then set
\begin{equation}\label{eq:def_eta}
    \eta_i^n(\xi_1, \ldots, \xi_n) = \text{argmin}_{y \in M_n, y \in \llbracket r,\xi_i \rrbracket} d_\bfT(y,\xi_i); \quad  \forall  i \leq n,
\end{equation}
which is the closest element of $M_n$ to $\xi_i$ that is an ancestor of $\xi_i$.
We write $\eta^n(\xi_1, \ldots, \xi_n)$ for the $n$--tuple with coordinates $\eta_i^n(\xi_1, \ldots, \xi_n)$
and we will abuse notation by writing $\eta^n(\xi_i) = \eta_i^n(\xi_1, \ldots, \xi_n)$. Equip $\spn(r, \eta^n(\xi_1, \dots, \xi_n) )$ with a probability measure $\mu_n^\eta$ by placing weight $1/n$ on $\eta^n(\xi_i)$ for every $i\leq n$ and with a metric $d^\eta$ according to an IP--rescaling as in \eqref{eq:ip_rescaling}. By construction, we then have the following lemma, see also Figure \ref{fig:all_the_trimmings} for an illustration.

\begin{lemma}\label{lemma:isometry_trimming}
    \hspace{-3pt}
    As rooted, weighted metric spaces, $(T_n^{\mathrm{trim}}, d_n^{\mathrm{trim}}, r_n, \mu_n^{\mathrm{trim}})$ and $(\spn(r, \eta^n(\xi_1, \dots, \xi_n)), d_n^\eta, r, \mu_n^\eta)$ are isomorphic.
\end{lemma}

We will implicitly use this representation of $(T_n^{\mathrm{trim}}, d_n^{\mathrm{trim}}, r_n, \mu_n^{\mathrm{trim}})$. We can now state the main theorem of this section. 

\begin{figure}[th]
    \centering
    \label{fig:all_the_trimmings}
    \includegraphics[scale = 1.0]{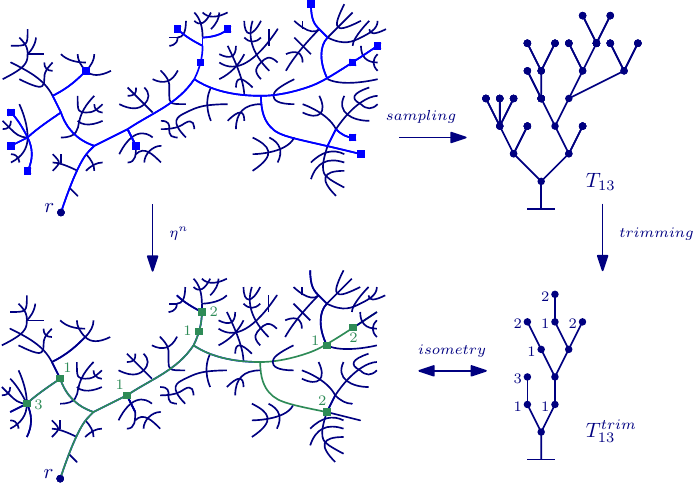}
    \caption{The different operations involved in trimming, this diagram commutes. The double--headed arrow is the isometry of Lemma \ref{lemma:isometry_trimming}. A number $k$ next to a vertex signifies an atom of weight $k/13$ for $\mu_{13}^{\eta}$ and $\mu_{13}^{\mathrm{trim}}$ respectively. Compare this to Figure \ref{fig:sampling_construction} where we obtain $T_{13}$ from $\bfT$ by sampling.}
\end{figure}

\begin{theorem}\label{thm:scaling_limit}
Let $(T_n, n\geq 1)$ be an extremal tree--valued Markov chain obtained from the decorated planar real tree $(\bfT, d, r, \mu, \psi, \lambda, \{\beta_a\}_a)$ through Construction \ref{construction:sampling_the_bridge} where $(\bfT, d, r, \mu)$ is an IP-tree. We then have
$$(T_n^{\mathrm{trim}}, d_n^{\mathrm{trim}}, r_n, \mu_n^{\mathrm{trim}})\xrightarrow{n\to \infty} (\bfT, d,r,\mu),$$
almost surely in the Gromov--Prokhorov topology.
\end{theorem}

From this we prove Theorem \ref{thm:scaling_limits_new}.

\begin{proof}[Proof of Theorem \ref{thm:scaling_limits_new}.] This follows immediately from Theorem \ref{thm:scaling_limit}, the decomposition into extremal distributions in Corollary \ref{cor:decomposition} and the classification of extremal tree--valued Markov chains with uniform backward dynamics in Theorem \ref{thm:main_thm_new_version}.
\end{proof}

The proof of the above theorem proceeds by comparing $T_n^{\mathrm{trim}}$ and $\bfT$ with $\spn(r,\xi_1, \ldots, \xi_n)$. For this purpose, let $\bfS_n = \spn(r,\xi_1, \ldots, \xi_n)$ and choose $\mu_n = \frac{1}{n} \sum_{i=1}^n \delta_{\xi_i}$. We then construct a metric $d_n$ for $\bfS_n$ by again considering the inhomogenous IP--rescaling \eqref{eq:ip_rescaling} with respect to the root $r$. The proof of Theorem \ref{thm:scaling_limit} consists of showing that both $\dGP(T_n^{\mathrm{trim}}, \bfS_n)$ and $\dGP(\bfS_n, \bfT)$ are small for $n$ sufficiently large. 

Before we can do this, we will show some general statements about IP--trees. In Lemma \ref{lemma:ip_trees_small_stuff} we show some small auxiliary statements and in Lemma \ref{lemma:ip_trees_partition} we construct a partition of $\bfT$ that helps us to approximate subtrees, i.e.\ sets of the form $F_\bfT(x)$ for $x \in \bfT$.

\begin{lemma}\label{lemma:ip_trees_small_stuff}Let $(\bfT, d, r, \mu)$ be an IP--tree.

    \begin{enumerate}
        \item[(i)] For all $x\in \bfT$ we have
        \begin{equation*}
             d(r,x) \leq 1 - \mu(F_\bfT(x)).
        \end{equation*}

        \item[(ii)] For any $c<1$ the set 
        \begin{equation*}
            \left\{x\in \bfT: d(r,x) = c \right\},
        \end{equation*}
        is at most countably infinite. 

        \item[(iii)]
        
        Suppose $(\bfT', d', r', \mu')$ is another IP--tree and $\varphi:\text{dom} \ \varphi \to \supp ( \mu')$ is an injective map respecting the tree structure with $\text{dom} \ \varphi \subseteq \supp ( \mu)$. This means $\varphi(x)\wedge \varphi(y) = \varphi(x\wedge y)$ for all $x,y\in \bfT$, and if $y \in \llbracket r,x \rrbracket$ then $\varphi(y) \in \llbracket r',\varphi(x) \rrbracket $ for all $x,y\in \bfT$. We then have
        \begin{equation*}
            \sup_{x,y \in \text{dom} \ \varphi} \left\vert 
            d(x,y) - d'(\varphi(x), \varphi(y))\right\vert \leq 4 \sup_{x \in \bfT} \left\vert \mu\left(F_\bfT(x) \right) - \mu'\left(F_{\bfT'}(\varphi(x)) \right) \right\vert.
        \end{equation*}

    \end{enumerate}
\end{lemma}

\begin{proof}
Recall the definition of IP--tree from Definition \ref{def:ip-tree}.
    \begin{enumerate}
        \item [(i)]If $x$ is either a branch point, a leaf or lies in the support of $\mu$ we have $d(r,x) = 1-\mu(F_\bfT(x))$. If that is not the case, consider 
        \begin{equation}\label{eq:closest_special_child}
            x^* = \inf \{y\in \bfT: x<y \text{ and } y \text{ is a branchpoint, a leaf or in the support of } \mu \},
        \end{equation}
        where the infimum is taken with respect to the ancestral order of $\bfT$. This may be closest descendant of $x$ for which the IP--tree property holds. We then have $\mu(F_\bfT(x)) = \mu(F_\bfT(x^*))$ and hence
        \begin{equation*}
            d(r,x) \leq d(r,x^*) = 1 - \mu(F_\bfT(x^*)) = 1- \mu(F_\bfT(x)).
        \end{equation*}
        It may happen that the infimum in \eqref{eq:closest_special_child} is not a branchpoint, a leaf or in the support of $\mu$. In that case the argument is easily adapted by considering a sequence that converges to $x^*$.

        \item [(ii)]Let $x\in \bfT$ be so that $d(r,x) = c$. Due to the spanning property of IP--trees, i.e.\ $\spn (  \supp (\mu)) = \bfT$, we necessarily have that $\mu(F_\bfT(x)) > 0$. Indeed, if we had $\mu(F_\bfT(x))=0$ then $x$ would need to be a leaf of $\bfT$ contained in the support of the diffuse part of $\mu$. In this case we would have $d(r,x)=1$ which contradicts $d(r,x)=c<1$. Because this is true for all $x$ in $\left\{y\in \bfT: d(r,y) = c \right\}$, this set has to be at most countably infinite. 

        \item [(iii)]Note that for $x,y \in \supp ( \mu)$ we have
        \begin{equation*}
            d(x,y)=\begin{cases}\vert \mu(F_\bfT(x))-\mu(F_\bfT(y))\vert \quad &\text{if} \quad y\in \llbracket r,x \rrbracket \ \ \text{or} \ \ x \in \llbracket r,y \rrbracket ,\\ 2 \mu(F_\bfT(x \wedge y)) - \mu(F_\bfT(x))-\mu(F_\bfT(y)) \quad &\text{else} .\end{cases}
        \end{equation*}
        An analogous statement holds for $\varphi(x),\varphi(y)$ with respect to $d'$ and $\mu'$. We then have for $x,y \in dom \ \varphi$ with $y \in \llbracket r,x \rrbracket$ that
        \begin{align*}
            \left\vert d(x,y) - d'(\varphi(x), \varphi(y)) \right\vert &= \left\vert \left(\mu(F_\bfT(y))-\mu(F_\bfT(x)) \right) - \left(\mu'(F_{\bfT'}(\varphi(y)))-\mu'(F_{\bfT'}(\varphi(x))) \right) \right\vert \\
            &\leq \sup_{z \in \bfT} 2 \left\vert \mu(F_\bfT(z))-\mu'(F_\bfT(\varphi(z))) \right\vert,
        \end{align*}
        where we have used the triangle inequality
        Similarly, if $y \notin \llbracket r,x \rrbracket$ and $x \notin \llbracket r,y\rrbracket $ then \begin{equation*}
            \left\vert d(x,y) - d'(\varphi(x), \varphi(y)) \right\vert \leq \sup_{z \in \bfT} 4 \left\vert \mu(F_\bfT(z))-\mu'(F_\bfT(\varphi(z))) \right\vert. \qedhere
        \end{equation*}
        
    \end{enumerate}
\end{proof}

\begin{lemma}\label{lemma:ip_trees_partition}
    Let $(\bfT, d,r,\mu)$ be an IP--tree. Then for any $\eps>0$ there exist $m_1,m_2\in \NN$ and measurable sets $A_1, \ldots A_{m_1},B_1,\ldots,B_{m_2}, S \subset \bfT$ such that 
    \begin{enumerate}
        \item 
        the sets $A_1, \ldots A_{m_1},B_1,\ldots,B_{m_2}, S$ partition $\bfT$,

        \item 
        $\mu(S)\leq \eps$ and for all $i \in [m_1],j\in [m_2]$ 
        \begin{equation*}
            \text{diam}(A_i) \leq \eps, \ \ \mu(A_i)\leq \eps \quad \text{and} \quad \# B_j  =1,
        \end{equation*}

        \item
        the closure of $\bigcup_{i=1}^{m_1} A_i \cup \bigcup_{j=1}^{m_2} B_j$ is connected,
        
        \item 
        for every $x\in \bfT$ there are $I_x \subseteq [m_1],J_x\subseteq [m_2]$ and $k_x \in [m_1]$ such that
        \begin{equation*}
            F_\bfT(x) \Delta \left( \bigcup_{i \in I_x} A_i \cup \bigcup_{j \in J_x} B_j \right)\subset A_{k_x} \cup S,
        \end{equation*}
        where $\Delta$ denotes the symmetric difference of two sets.

    \end{enumerate}
\end{lemma}

\begin{proof}
    See Figure \ref{fig:partition_tree} for an illustration of this partition. To construct it, we first consider the atoms of $\mu$. Enumerate them by $\{a_j, 1 \leq j \leq J\}$ with $J\in \NN_0\cup \{\infty\}$ so that $\mu(a_j) \geq \mu(a_{j+1})$. Choose $m_2$ in such a way that
    \begin{equation}\label{eq:partition_small_2}
        \sum_{j=m_2+1}^J \mu(a_j) \leq \frac{\eps}{2}.
    \end{equation}
    Note that if $\mu$ has no atoms we have $m_2 = 0$. We then define
    \begin{equation*}
        B_j = \{ a_j\},
    \end{equation*}
    for all $j\in [m_2]$. Next, we construct $A_1,\ldots, A_{m_1}$. For this, let $L= \lfloor 2/\eps \rfloor$. For a given $0\leq c<1$ consider the set
    \begin{equation*}
        D(c) = \{ x\in \bfT: d(r,x) = c \}.
    \end{equation*}
    By Lemma \ref{lemma:ip_trees_small_stuff} $(ii)$, this is always at most a countable set. Without loss of generality assume  $L2/\eps < 1$. Define $D=\bigcup_{\ell = 0}^L D(\ell \eps /2)$ which is also a countable set. For $x\in D$ we set
    \begin{equation*}
        T(x) = \{ y \in F_\bfT(x): 0\leq d(x,y) < \eps/2 \}.
    \end{equation*}
    Note that $\{T(x): x\in D\}$ is a partition of $\bfT$ and $diam(T(x))\leq \eps$. Choose a finite subset $C\subset D$ such that $r\in C$, $\bigcup_{x \in C} T(x)$ is connected and such that
    \begin{equation}\label{eq:partition_small_1}
        \mu \left( \bigcup_{x\in D\backslash C} T(x) \right) \leq \frac{\eps}{2}.
    \end{equation}
    This is always possible because $\mu$ is a probability measure. Further, we choose $C$ so that $C \neq \emptyset$.
    For every $x\in C$ we set
    \begin{equation*}
        A(x) = T(x) \backslash \bigcup_{j=1}^{m_2}\{ a_j\},
    \end{equation*}
    this means we remove any atoms from $T(x)$ that are already included in $\{B_j\}_j$. 
    Observe that $\mu(A(x))\leq \eps$: Indeed, assume we had $\mu(A(x))>\eps$, then there would exist $y\in A(x)$ with $\mu(F_\bfT(y))<\mu(F_\bfT(x))-\eps/2$. 
    Because of the spacing property \eqref{eq:IP_spacing} of IP--trees this would imply $d(x,y)>\eps/2$ but because $y\in A(x)$ we have $d(x,y)<\eps/2$ which is a contradiction. 
    %Describing the choice of $y$ precisely is tedious as the spacing property \eqref{eq:IP_spacing} of IP--trees property does not necessarily apply to any $z\in F_\bfT(x)$ with $d(x,z)= \eps /2$ as it only applies to branchpoints and points in the support of $\mu$.

    Now, let $m_1 = \vert C \vert$ and enumerate $\{A(x), x \in C \}$ by $A_1, \ldots, A_{m_1}$. 
    Lastly, we define 
    \begin{equation*}
        S = \bigcup_{x\in D\backslash C}T(x) \cup \bigcup_{\substack{j=m_2 + 1 \\ \forall i\leq m_1: j \notin A_i}}^J \{ a_j \}.
    \end{equation*}
    We include all atoms that are not in $\bigcup_{i=1}^{m_1} A_i \cup \bigcup_{j=1}^{m_2} B_j$. 
    By combining \eqref{eq:partition_small_2} and \eqref{eq:partition_small_1} we can see that $\mu(S)\leq \eps.$ This shows properties $1.)$ and $2.)$ of the lemma. Property $3.)$ follows from the fact that if we take the closure, we have that $ T(x) \subseteq\overline{A(x)}$ and that $\bigcup_{x \in C}T(x)$ is connected. 

    Finally, to complete the proof of the lemma, for a given $x\in \bfT$ we specify $I_x \subseteq [m_1],J_x \subseteq [m_2]$ and $k_x \in [m_1]$ with the required properties. 
    If $x\in S$, then we set $I_x = J_x = \emptyset$ and we choose $k_x$ arbitrarily, say $k_x=1$.

    Assume that $x\notin S$. Then there is $z\in C$ such that $x\in T(z)$. Choose $k_x$ so that $A_{k_x} = A(z)$. For every $i\in [m_1]\backslash \{k_x \}$ we have either $A_i \subset F_\bfT(x)$ or $A_i \cap F_\bfT(x) = \emptyset$. Based on this, we set
    \begin{equation*}
        I_x = \left\{i\in [m_1]: A_i \cap F_\bfT(x) = A_i \right\}.
    \end{equation*}
    Similarly, we set $J_x = \{j\in [m_2]: B_j \cap F_\bfT(x) = B_j \}$. This is the set of atoms $a_j$ with $j\leq m_2$ that are contained in $F_\bfT(x)$. By construction, we have
    \begin{equation*}
        F_\bfT(x) \Delta \left( \bigcup_{i \in I_x} A_i \cup \bigcup_{j \in J_x} B_j \right)\subset A_{k_x} \cup S.
    \end{equation*}
\end{proof}

\begin{figure}[bht]
    \centering
    \includegraphics[scale = 1.0]{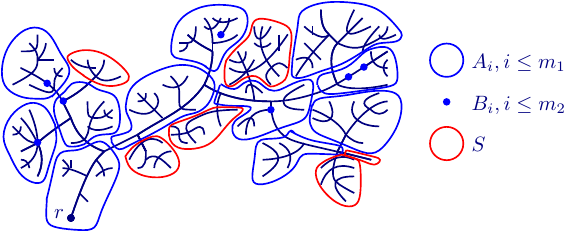}
    \caption{The partition of $\bfT$ in Lemma \ref{lemma:ip_trees_partition}. }
    \label{fig:partition_tree}
\end{figure}

\noindent Recall that $\mu_n = \frac{1}{n} \sum_{i=1}^n \delta_{\xi_i}$.
Now that we have shown some topological properties of IP--trees we show that for large $n$ $\mu_n$ approximates $\mu$ well in the following sense.

\begin{lemma}\label{lemma:control_on_Ft}
For any $\eps>0$, there is a random variable $N_1=N_1(\eps)$ such that for $n\geq N_1$ we have almost surely
$$ \sup_{x \in \bfT}\big\vert \mu(F_\bfT(x))-\mu_n(F_\bfT(x)) \big\vert \leq \eps.$$
\end{lemma}

\begin{proof}
    We make use of Lemma \ref{lemma:ip_trees_partition} with constant $\eps/5$. 
    For $x\in S$, we set $I_x=\emptyset, J_x = \emptyset$ and by abuse of notation $A_{k_x}=\emptyset$.
    For all $x\in \bfT$ we let
    \begin{equation}
        \Tilde{F}_\bfT(x)=\bigcup_{i \in I_x} A_i \cup \bigcup_{j \in J_x} B_j \cup S.
    \end{equation}
    By use of the triangle inequality we get
    \begin{align*}
        \big\vert \mu(F_\bfT(x))-\mu_n(F_\bfT(x)) \big\vert 
         &\leq \big\vert \mu(F_\bfT(x))-\mu(\Tilde{F}_{\bfT}(x)) \big\vert 
         + \big\vert \mu_n(F_\bfT(x))-\mu_n(\Tilde{F}_\bfT(x)) \big\vert  
        + \big\vert \mu(\Tilde{F}_{\bfT}(x))-\mu_n(\Tilde{F}_{\bfT}(x)) \big\vert .
    \end{align*}
    We then have by Lemma \ref{lemma:ip_trees_partition} 
    \begin{align*}
        \big\vert \mu(F_\bfT(x))-\mu(\Tilde{F}_{\bfT}(x)) \big\vert \leq \mu(A_{k_x}) + \mu(S) \leq \frac{2}{5}\eps,
    \end{align*}
    as well as 
    $$\big\vert \mu_n(F_\bfT(x))-\mu_n(\Tilde{F}_{\bfT}(x)) \big\vert \leq \mu_n(A_{k_x}) + \mu_n(S),$$
    for all $n \geq 1$.
    Using the definition of $\Tilde{F}_{\bfT}(x)$ we get 
    \begin{align*}
    \big\vert \mu(\Tilde{F}_{\bfT}(x))-\mu_n(\Tilde{F}_{\bfT}(x)) \big\vert 
    &\leq \vert \mu(S)-\mu_n(S) \vert + \sum_{i\in I_x} \vert \mu(A_i) - \mu_n(A_i)\vert + \sum_{j\in J_x} \vert \mu(B_j) - \mu_n(B_j)\vert\\
    &\leq \vert\mu(S)-\mu_n(S) \vert + \sum_{i=1}^{m_1} \vert \mu(A_i) - \mu_n(A_i)\vert + \sum_{j=1}^{m_2} \vert \mu(B_j) - \mu_n(B_j)\vert.
    \end{align*}
    And lastly, using the bound
    $$\mu_n(A_{k_x}) \leq \mu(A_{k_x}) + \vert \mu(A_{k_x})-\mu_n(A_{k_x})\vert \leq \frac{\eps}{5} + \sum_{i=1}^{m_1} \vert \mu(A_{i})-\mu_n(A_{i})\vert,$$
    and similarly
    $$\mu_n(S) \leq \frac{\eps}{5} + \vert \mu(S) - mu_n(S)\vert,$$
    we obtain 
    \begin{equation}\label{eq:estimate_FTx}
        \big\vert \mu(F_\bfT(x))-\mu_n(F_\bfT(x)) \big\vert \leq \frac{4}{5} \eps + 2\vert\mu(S)-\mu_n(S) \vert + 2\sum_{i=1}^{m_1} \vert \mu(A_i) - \mu_n(A_i)\vert + \sum_{j=1}^{m_2} \vert \mu(B_j) - \mu_n(B_j)\vert.
    \end{equation}
    Note that this estimate is uniform in $x\in \bfT$.
    By the strong law of large numbers the family of random variables $\{ \vert\mu(S)-\mu_n(S)\vert,(\vert \mu(A_i) - \mu_n(A_i)\vert)_{i=1,\ldots,m_1},(\vert \mu(B_j) - \mu_n(B_j)\vert)_{j=1,\ldots,m_2}\}$ converges jointly $\PP$--almost surely to $0$. Applying this to \eqref{eq:estimate_FTx} this yields the existence of a random variable $N_1$ such that for every $n\geq N_1$ we have
    \begin{equation*}
        \sup_{x \in \bfT} \big\vert \mu(F_\bfT(x))-\mu_n(F_\bfT(x)) \big\vert \leq \frac{4}{5} \eps + \frac{1}{5}\eps.
    \end{equation*}
\end{proof}

After having established control over $\mu_n$, we can show that $\dGP(\bfS_n, \bfT)$ and $\dGP(T_n^{\mathrm{trim}}, \bfS_n)$ are small. 

\begin{lemma}\label{prop:scaling_limit_part_3}
For any $\eps>0$, there is a random variable $N_2=N_2(\eps)$ such that for $n\geq N_2$ we have almost surely
$$\dGP \left((\bfS_n, d_n, r, \mu_n), (\bfT, d, r, \mu)\right) \leq  \eps.$$
\end{lemma}

\begin{proof} 
    Fix $\eps>0$. Recall the defining property of the metric of IP--trees. For $x \in supp(\mu)$ we have
    \begin{equation*}
        d(r,x) = 1-\mu(F_\bfT(x)),
    \end{equation*}
    and similarly $d_n(r,x)= 1-\mu_n(F_{\bfS_n}(x))$ for $x\in \supp (\mu_n)$.
    This means that to understand the metric, we only need to understand the measure. Note that for every $x\in \bfS_n$ we have
    \begin{equation}\label{eq:sip:c}
        \mu_n(F_{\bfS_n}(x)) = \mu_n(F_{\bfT}(x)),
    \end{equation}
    by extending $\mu_n$ to $ \bfT$. 
    Lemma \ref{lemma:control_on_Ft} allows us to control the expressions above. 

    To use Definition \ref{def:GP_in_introduction} to estimate $\dGP(\bfT,\bfS_n)$, we need to couple $\mu$ and $\mu_n$. To do this, we apply Lemma \ref{lemma:ip_trees_partition} with $\eps/12$ in place of $\eps$. Note that $\supp (\mu_n) \subset \supp ( \mu)$ $\PP$--almost surely. Conditional on $\xi_1, \ldots, \xi_n$, consider any coupling $\nu_n$ of $\mu$ and $\mu_n$ such that for every $i\in [m_1], j\in [m_2]$
    $$\nu_n(A_i \times A_i) = \min \{ \mu(A_i),\mu_n(A_i) \}, \ \ \nu_n(B_j \times B_j) = \min \{ \mu(B_j),\mu_n(B_j) \}.$$
    Note that such a coupling always exists. Consider the following subset of $\bfT \times \bfT$,
    \begin{equation*}
        R = \bigcup_{i=1}^{m_1} (A_i \times A_i) \cup \bigcup_{j=1}^{m_2} (B_j \times B_j),
    \end{equation*}
    here we again use that $\bfS_n$ is a subset of $\bfT$.
    By the strong law of large numbers the family of random variables $\{ \vert\mu(S)-\mu_n(S)\vert,(\vert \mu(A_i) - \mu_n(A_i)\vert)_{i=1,\ldots,m_1},(\vert \mu(B_j) - \mu_n(B_j)\vert)_{j=1,\ldots,m_2}\}$ converges jointly $\PP$--almost surely to $0$.   
    Hence, there exists a random variable $N_2^*(\eps)$ such that for every $n\geq N_2^*$ we have
    $$\nu_n \big(R) \geq 1-\eps.$$
    By Definition \ref{def:GP_in_introduction}, it suffices to show for $n$ sufficiently large that
    \begin{equation*}
        \sup_{(x,y),(x',y') \in R} \vert d(x,x') - d_n(y,y') \vert  \leq \eps.
    \end{equation*}

    First, for $x\in A_i$ and $y\in A_j$, we want to decompose $d(x,y)$. We implicitly restrict ourselves to $x,y \in \supp (\mu_n) \subset \supp ( \mu )$ so that we can later apply Lemma \ref{lemma:ip_trees_small_stuff} $(iii)$. For every $i\in [m_1]$, choose $r_i \in A_i$ arbitrarily. In fact, if we look into the proof of Lemma \ref{lemma:ip_trees_partition}, we see that we can choose $r_i$ to be the root of $A_i$ but we will not use this here. By the triangle inequality we have
    $$d(r_i,r_j)- d(x,r_i) - d(y,r_j)\leq d(x,y)\leq d(r_i,r_j)+d(x,r_i)+d(y,r_j),$$
    and by using that $diam(A_i)\leq \eps/12$ and $diam(A_j) \leq \eps/12$ we get
    $$d(r_i,r_j)-\frac{1}{6}\eps \leq d(x,y) \leq d(r_i,r_j) + \frac{1}{6}\eps.$$
    
    Next, we must show a similar statement for $d_n$. For this we need to estimate $diam_n  A_i$, the diameter of $A_i$ under the metric $d_n$. We apply Lemma \ref{lemma:control_on_Ft}, for $n \geq N_1(\eps/48)$ we have
    \begin{align*}
        diam_n(A_i) & \leq 2 \sup_{x\in A_i} d_n(r_i,x) \\
        &\leq 2 \sup_{x \in A_i} d(r_i,x) + 2 \sup_{x \in A_i} \vert d(r_i,x)-d_n(r_i,x) \vert \\
        &\leq 2 \ diam(A_i) + 8 \sup_{z \in \bfT} \vert \mu(F_\bfT(z)) -  \mu_n(F_\bfT(z)) \vert \\
        &\leq \frac{\eps}{6} + \frac{\eps}{6} 
    \end{align*}
    where we also applied Lemma \ref{lemma:ip_trees_small_stuff} $(iii)$; the map $\varphi$ here is the inclusion $\bfS_n \hookrightarrow \bfT$. This yields that for $x\in A_i$ and $y\in A_j$ we have
    \begin{equation*}
        d_n(r_i,r_j)-\frac{2}{3} \eps \leq d_n(x,y) \leq d_n(r_i,r_j) + \frac{2}{3}\eps.
    \end{equation*}
    The same reasoning works if $x\in A_i$ and $y\in B_j$ for some $i \in [m_1], j\in [m_2]$. In that case we have also
    \begin{equation*}
        \vert d(x,y) - d(r_i, y) \vert \leq \frac{\eps}{12},
    \end{equation*}
    and
    \begin{equation*}
        \vert d_n(x,y) - d_n(r_i, y) \vert \leq \frac{\eps}{3}.
    \end{equation*}
    As a consequence, for $(x,y),(x',y') \in R$ we have
    \begin{align*}
        \vert d(x, x') - &d_n(y,y') \vert \leq \frac{5}{6}\eps + \max \left\{ \left\vert d(r',r'') - d_n(r',r'') \right\vert; r',r'' \in \{r_i,i\in [m_1] \} \cup \bigcup_{j=1}^{m_2} B_j\right\}.
    \end{align*}
    And by Lemma \ref{lemma:ip_trees_small_stuff} $(iii)$ and Lemma \ref{lemma:control_on_Ft} as above,
    \begin{align*}
       \vert d(x, x') - d_n(y,y') \vert \leq \frac{5}{6}\eps + 4 \sup_{z\in \bfT} \vert \mu (F_{\bfT}(z)) - \mu_n (F_{\bfT}(z) )\vert \leq \frac{5}{6}\eps + \frac{1}{12}\eps < \eps,
    \end{align*}
    for $n \geq N_1(\eps/48)$.

    We now have
    \begin{equation*}
        \sup_{(x,y),(x',y') \in R} \vert d(x,x') - d_n(y,y') \vert \leq \eps
    \end{equation*}
    Recall that $\nu_n(R)\geq 1 - \eps$. With Definition \ref{def:GP_in_introduction} this yields $\dGP(\bfT,\bfS_n) \leq \eps$, $\PP$--almost surely for $n\geq N_2 := \max \{N_1(\eps/48),N_2^*\}$. 
\end{proof}

\begin{lemma}\label{prop:scaling_limit_part_2}
For any $\eps>0$, there is a random variable $N_3=N_3(\eps)$ such that for $n\geq N_3$ we have almost surely
$$ \dGP \left((T_n^{\mathrm{trim}}, d_n^{\mathrm{trim}}, r_n, \mu_n^{\mathrm{trim}}) ,
(\bfS_n, d_n, r, \mu_n)\right) \leq \eps.$$
\end{lemma}

\begin{proof}
    Fix $\eps > 0$ small. First, we need to understand $\eta^n$ better, recall the definition of $\eta^n$ from \eqref{eq:def_eta}. To this end, consider the sets 
    \begin{align*}
        \textbf{B} &= \left\{x\in \bfT: d(r,x) = 1-\eps/16 \right\}, \\
        \textbf{C} &= \left\{x \in \bfT: \mu(\{x\}) > \eps/16 \ \text{and} \ F_\bfT(x)=\{x\} \right\}.
    \end{align*}
    This means that $\textbf{B}$ is the level set at height $1-\eps/16$ and $\textbf C$ is the set of atoms with mass greater than $\eps/16$ that are also leaves. By Lemma \ref{lemma:ip_trees_small_stuff} $(ii)$ the set $\textbf{B}$ is at most countably infinite and $\textbf{C}$ is clearly finite. Enumerate $\textbf{ B} \cup \textbf{C}$ by $\{ z_i, i\geq 1\}$,
    \begin{equation*}
        \textbf{B} \cup \textbf{C} = \bigcup_{i=1}^\infty \{ z_i \}.
    \end{equation*}
    Observe that
    \begin{align*}
        \bfT = \bigcup_{i=1}^\infty \llbracket r,z_i) \cup F_\bfT(z_i) = \bigsqcup_{i=1}^\infty \bigg(\llbracket r,z_i) \big\backslash \bigg( \bigcup_{j=1}^{i-1} \llbracket r,z_j) \bigg) \bigg) \cup  F_\bfT(z_i),
    \end{align*}
    where $\sqcup$ is a disjoint union. This implies that we can choose $K$ large enough such that 
    \begin{equation*}
        \mu \bigg(\bfT \big\backslash \bigcup_{i=1}^K \llbracket r,z_i) \cup F_\bfT(z_i) \bigg) \leq \frac{\eps}{2}
    \end{equation*}
    Such a $K$ exists because for every $z\in \textbf{B}$ we have $\mu(F_\bfT(z)) \leq 1-d(r,z) \leq \eps/16$ by Lemma \ref{lemma:ip_trees_small_stuff} $(i)$ and the definition of $\textbf B$. We now define the sets 
    \begin{align*}
        \textbf{L} = \bigcup_{i=1}^K F_\bfT(z_i) \quad \text{and} \quad
        \textbf{D} = \bigcup_{i=1}^K \llbracket r, z_i).
    \end{align*}
    Note that $\textbf{L}$ and $\textbf{D}$ are disjoint and $\mu(\textbf{L}) + \mu(\textbf{D})\geq 1 - \eps/2$ due to our choice of $K$. We think of $\textbf{L}$ as the part of $\bfT$ that is close to the leaves and of $\textbf{D}$ as the skeleton of the tree that leads to $\textbf{L}$. 

    Observe that for every $1\leq i \leq n$ we have $\eta^n(\xi_i)=\xi_i$ if 
    \begin{equation*}
        \left\vert F_\bfT(\xi_i) \cap \left\{ \xi_j; 1 \leq j \leq n \right\}\right\vert \geq 2.
    \end{equation*}
    This leads us to consider the event
    \begin{equation*}
       \mathcal{A}_n = \left\{ \forall 1\leq j \leq K:  \left\vert F_\bfT(z_j) \cap \left\{ \xi_i; 1 \leq i \leq n \right\}\right\vert \geq 2 \right\}. 
    \end{equation*}
    On the event $\mathcal{A}_n$ we have for all $1\leq i \leq n$ with $\xi_i \in \textbf{D} \cup \textbf{L}$ that
    \begin{equation}\label{eq:what_eta_actually_does}
        \begin{cases}
            \eta^n(\xi_i) = \xi_i \quad & \text{if} \quad \xi_i \in \textbf{D}, \\
            \eta^n(\xi_i) \in F_\bfT(z_j) & \text{if} \quad \xi_i \in F_\bfT(z_j) \text{ for some } j\leq K.
        \end{cases}
    \end{equation}
    Next, we want to start estimating $\dGP \left((T_n^{\mathrm{trim}}, d_n^{\mathrm{trim}}, r_n, \mu_n^{\mathrm{trim}}) ,
    (\bfS_n, d_n, r, \mu_n)\right)$. We use Definition \ref{def:GP_in_introduction} to compute $\dGP$. Consider the following subset of $\bfT\times \bfT$
    \begin{equation*}
        R_n = \bigcup_{i \in [n]: \xi_i \in \textbf{D} \cup \textbf{L}} \{\xi_i\} \times \{ \eta^n( \xi_i) \}.
    \end{equation*}
    We choose the natural coupling $\nu_n$ of $\mu_n$ and $\mu_n^{\mathrm{trim}}$, that places mass $1/n$ on $(\xi_i , \eta^n (\xi_i ) )$,
    \begin{equation*}
        \nu_n = \sum_{i=1}^n \frac{1}{n}\delta_{(\xi_i , \eta^n (\xi_i ) )}.
    \end{equation*}
    Note that as $n \to \infty$, we have
    \begin{equation}\label{eq:correspondence_has_mass}
        \liminf_{n \to \infty} \nu_n(R_n) = \liminf_{n \to \infty} \mu_n\left( \textbf{D}\cup \textbf{L} \right) \geq 1- \frac{\eps}{2},
    \end{equation}
    $\PP$--almost surely. Further, note that $\mathcal{A}_{n-1} \subset \mathcal{A}_n$ and $\lim_{n \to \infty}\PP(\mathcal{A}_n) = 1$. Hence, there is a random variable $N_3^* = N_3^*(\eps)$ such that for all $n\geq N_3^*$ we have $\nu_n(R_n) \geq 1 - \eps$. 

    Let us now estimate 
    \begin{equation*}
        \sup_{(x,y),(x',y') \in R_n}  \left\vert d_n(x,x') - d_n^{\mathrm{trim}}(y,y') \right\vert,
    \end{equation*}
    on the event $\mathcal A_n$. Let $I_n = \{i \in [n]: \xi_i \in \textbf{D} \cup \textbf{L} \}$. Note that by the definition of $R_n$ we have
    \begin{equation}
        \sup_{(x,y),(x',y') \in R_n}  \left\vert d_n(x,x') - d_n^{\mathrm{trim}}(y,y') \right\vert = \sup_{i,j \in I_n} \left\vert d_n(\xi_i, \xi_j) - d_n^{\mathrm{trim}}(\eta^n(\xi_i), \eta^n(\xi_j))\right\vert.
    \end{equation}
    We now apply Lemma \ref{lemma:ip_trees_small_stuff} $(iii)$ where $\varphi$ is given by $\eta^n$ restricted to $\{\xi_i, i\in  I_n \}$. We combine this with \eqref{eq:what_eta_actually_does} to obtain
    \begin{align*}
        \sup_{(x,y),(x',y') \in R_n}  \left\vert d_n(x,x') - d_n^{\mathrm{trim}}(y,y') \right\vert &\leq 4 \sup_{i \in I_n} \left\vert \mu_n \left( F_\bfT(\xi_i) \right) - \mu_n\left( F_\bfT(\eta^n(\xi_i)) \right)\right\vert  \\
        &\leq 4 \sup_{\substack{1\leq j \leq K \\ z_j \in \textbf{B}}} \sup_{x,y \in F_\bfT(z_j)}\left\vert \mu_n \left( F_\bfT(x) \right) - \mu_n\left( F_\bfT(y) \right)\right\vert .
    \end{align*}
    By using the triangle inequality twice we have
    \begin{align*}
        4 \sup_{\substack{1\leq j \leq K \\ z_j \in \textbf{B}}} \sup_{x,y \in F_\bfT(z_j)}\left\vert \mu_n \left( F_\bfT(x) \right) - \mu_n\left( F_\bfT(y) \right)\right\vert &\leq 8 \sup_{\substack{1\leq j \leq K \\ z_j \in \textbf{B}}} \sup_{x \in F_\bfT(z_j)}\left\vert \mu_n \left( F_\bfT(x) \right) - \mu\left( F_\bfT(x) \right)\right\vert \\ 
        & \ \ \ \ + 4 \sup_{\substack{1\leq j \leq K \\ z_j \in \textbf{B}}} \sup_{x,y \in F_\bfT(z_j)}\left\vert \mu \left( F_\bfT(x) \right) - \mu\left( F_\bfT(y) \right)\right\vert \\ 
        &\leq 8 \sup_{x \in \bfT}\left\vert \mu_n \left( F_\bfT(x) \right) - \mu\left( F_\bfT(x) \right)\right\vert + 8 \sup_{\substack{1\leq j \leq K \\ z_j \in \textbf{B}}} \mu\left(F_\bfT(z_j) \right).
    \end{align*}
    By construction, we have $\mu\left(F_\bfT(z_j) \right) \leq \eps/16$ for every $z_j \in \textbf{B}$. The other term, $\sup_{x \in \bfT}\vert \mu_n \left( F_\bfT(x) \right) - \mu\left( F_\bfT(x) \right)\vert$, is controlled by Lemma \ref{lemma:control_on_Ft} -- we apply it with parameter $\eps/16$. This means that for $n\geq N_3 := \max\{ N_1(\eps/16), N_3^* \}$ we have
    \begin{equation}
        \sup_{(x,y),(x',y') \in R_n}  \left\vert d_n(x,x') - d_n^{\mathrm{trim}}(y,y') \right\vert \leq 8 \frac{\eps}{16} + 8 \frac{\eps}{16}.
    \end{equation}
    Recall that for $n\geq N_3^*$ we have $\nu_n(R_n) \geq 1- \eps$. This implies that for $n\geq N_3$
    \begin{equation*}
        \dGP \left((T_n^{\mathrm{trim}}, d_n^{\mathrm{trim}}, r_n, \mu_n^{\mathrm{trim}}) ,
(\bfS_n, d_n, r, \mu_n)\right) \leq \eps,
    \end{equation*}
     which is the statement of the lemma.
\end{proof}

Finally, we prove Theorem \ref{thm:scaling_limit}.

\begin{proof}[Proof of Theorem \ref{thm:scaling_limit}.] This now follows straight away from Lemmas \ref{prop:scaling_limit_part_3} and \ref{prop:scaling_limit_part_2}. We have by the triangle inequality
\begin{align*}
    \dGP \big((T_n^{\mathrm{trim}} & , d_n^{\mathrm{trim}}, r_n, \mu_n^{\mathrm{trim}}),
    (\bfT, d, r, \mu) \big)
     \\
    & \leq \dGP \big((T_n^{\mathrm{trim}}, d_n^{\mathrm{trim}}, r_n, \mu_n^{\mathrm{trim}}) ,
    (\bfS_n, d_n, r, \mu_n)\big) + \dGP \big(\bfS_n, d_n, r, \mu_n), (\bfT, d, r, \mu)\big) \\
    & \leq 2\eps. \qedhere
\end{align*}
\end{proof}

\section{Appendix: proof of Proposition \ref{prop:evans}}\label{sec:proof_of_evans}

Here we will sketch the proof of Proposition \ref{prop:evans}. Large parts of it are analogous to arguments seen in Evans, Grübel, Wakolbinger \cite[Sections 6 and 7]{evans_doob-martin_2017} with the difference being that the authors of \cite{evans_doob-martin_2017} consider only binary trees. Compare this also to \cite[Sections 6 and 8]{birkner_patricia_2020} which discusses a similar proof for the modified definition of dendritic systems in \cite{birkner_patricia_2020}. We present the proof to give the reader a more complete picture. 

The proof consists of three steps: First, we go from the dendritic system to an ultra--metric on $\NN$ which can be represented by coalescent tree. In doing this, we lose information about the planar structure. Secondly, we apply a theorem of Gufler \cite{gufler_representation_2018} to derive a sampling procedure for the ultra--metric. The ultra--metric can be represented by sampling points from a real tree. In a third step, we recover the planar structure by use of the Aldous--Hoover--Kallenberg theory of exchangeable arrays, see the book of Kallenberg \cite[Chapter 28]{kallenberg_foundations_2021} for a general reference. 

In the following sections we encode a given exchangeable, ergodic dendritic system $\mathcal{D}=(\NN, \sim, \preceq,p)$. This dendritic system was obtained from a tree--valued Markov chain $(T_n, n\geq 1)$ with uniform backward dynamics, see Proposition \ref{prop:tree_growth_to_dendritic+ergodic}. Each section corresponds to one of the step of the proof of Proposition \ref{prop:evans} outlined above. We will end up with the objects of Proposition \ref{prop:evans}: a rooted real tree $(\bfT,d,r)$, a probability measure $\mu$ and a function $F$ which encodes the planarity function $p$ of $\mathcal{D}$.

\subsection{From the dendritic system to a coalescent tree}

Assume we are given an exchangeable, ergodic dendritic system $\mathcal{D}=(\NN, \sim, \preceq,p)$, our goal is to connect its distribution to a real tree.
The first thing we need to do is to derive an ultra--metric. To this end, given $i,j \in \NN$ and any leaf $k\in \NN$ we set
$$I_k:= \II\{(i,j) \preceq k \},$$
recall that we abuse notation to write $k=(k,k)$.
By exchangeability of $\mathcal D$ the sequence $(I_k)_{k>\max\{i,j\}}$ is also exchangeable. Hence the limit
$$d(i,j)=\lim_{n \to \infty} \frac{1}{n} \sum_{k=1}^n I_k$$
exists almost surely by de Finetti's theorem. Note that $d$ is random. 

\begin{lemma}\label{lemma:ultrametric}
$d$ is almost surely an ultra--metric on $\NN$, that is for all $i,j,k\in \NN$ we have:
\begin{enumerate}
\item $d(i,j)\geq 0$, and $d(i,j)=0 \Leftrightarrow i=j$. 
\item $d(i,j)=d(j,i).$
\item $d(i,k)\leq \max\{ d(i,j),d(j,k)\}.$
\end{enumerate}
\end{lemma}

\begin{proof}
The proof of this lemma is analogous to the proof of \cite[Lemma 6.1]{evans_doob-martin_2017}, we repeat it for the reader's convenience. Notice that little changes in going from the binary trees of \cite{evans_doob-martin_2017} to multi--furcating in our setting.

The symmetry of $d$ and $d(i,i)=0$ follow readily from the definition of $I_k$.
We show that $i \neq j$ implies $d(i,j)>0$ almost surely. To this end, we first observe that the events $\{d(i,j)=0\}$ and $\{\forall k \notin \{i,j\}:I_k=0\}$ agree almost surely. Indeed, if we had $\PP(I_k=1)>0$ for some $k$, then this would be true for infinitely many $k$, by exchangeability, and then by de--Finetti's theorem we would then have $d(i,j) >0$. Hence, we almost surely have
$$\big\{d(i,j)=0\big\}=\big\{\forall k \notin \{i,j\}:I_k=0\big\}= \big\{ \nexists k \notin \{i,j\}: (i,j) \preceq k \big\}$$
where the second equality follows from the definition of $I_k$. On the level of trees this means that for all $n\geq \max\{i,j\}$, the leaves labelled $i$ and $j$ are attached to the same vertex in $T_n$ and no other leaves or subtrees are attached to the same vertex. The authors of \cite{evans_doob-martin_2017} call this a \textit{cherry}. 

We now want to estimate the probability of the event that $i$ and $j$ are part of the same cherry in $T_n$, equivalently in the dendritic system restricted to $[n]$. Because $T_n$ has $n$ leaves, the number of cherries is at most $\frac{n}{2}$. Recall that the leaves are labelled exchangeably. This means that we can relabel the leaves of $T_n$ uniformly without changing the distribution. The probability that the labels $i$ and $j$ are part of the same cherry is at most $\frac{n}{2}\frac{2}{n(n-1)}$. This allows us to conclude
$$\PP\big( d(i,j)=0 \big) \leq \limsup_{n \to \infty} \PP \big(i \text{ and } j \text{ form a cherry} \big)\leq \limsup_{n \to \infty}\frac{n}{2}\frac{2}{n(n-1)}=0$$
which is equivalent to $d(i,j)>0$ almost surely for $i\neq j$. 

Lastly, consider the subtree spanned in $T_n$ by the leaves labelled $i,j,k$. We see that we have either $(i,k)=(j,k)\preceq (i,j)$ or a permutation thereof for all $n\geq \max\{i,j,k\}$. 
In the limit as $n\to\infty$ this entails $d(i,k)=d(j,k)\geq d(i,j)$ or a permutation thereof. In any case, the ultra--metric inequality $d(i,k) \leq \max\{ d(i,j), d(j,k) \}$ is satisfied.  
\end{proof}

Given an ultra--metric $d$ on $\NN$ that is bounded above by $1$, we can associate a coalescent and a real tree to it in a canonical way, see for example Evans \cite[Example 3.41]{evans_probability_2008} 
%\textcolor{gray}{[D: Just as a commentary to our discussion: in the book by Evans it says, 'this tree is compact if and only if the associated coalescent process has finitely many blocks for every $t>0$'. -> in particular this implies that non-compact trees can arise (which will be the case for us)]} or \cite[Section 6]{evans_doob-martin_2017}.

We explain this procedure here. Define a family of equivalence relations $(\equiv_t, t\in [0,1])$ on $\NN$ by declaring $i\equiv_t j$ if and only if $d(i,j) \leq t$. Notice that elements of $\NN$ can be identified with equivalence classes of $\equiv_0$ and that all elements of $\NN$ are equivalent under $\equiv_1$. 
Now, we extend $d$ to a metric of pairs of the form $(A,t)$ where $A$ is an equivalence class of $\equiv_t$, call this set $\bfS^\circ$. Given $(A,t)$ and $(B,s)$ we set
$$H((A,t),(B,s)) = \inf \big\{ u \geq \max \{s,t\}: k\equiv_u \ell, \forall k\in A, \ell \in B \big\}$$
and
$$d((A,t),(B,s)) = H((A,t),(B,s)) - \frac{s+t}{2}.$$
One can check that $d((\{i\},0),(\{j\},0))=d(i,j)$, so $d$ is an extension of the previous metric. Further, one can check that $d$ is indeed a metric and that the metric completion of $(\bfS^\circ, d)$ is a real tree, call it $(\bfS, d)$. $\bfS$ can be endowed by an ancestral order $<$ by setting $(A,t)<(B,s)$ if $s<t$ and $B\subset A$. Root $\bfS$ at the minimal element of $<$, call the root $r$.

Recall that we obtained the dendritic system $\mathcal{D} = (\NN, \sim, \preceq, p)$ from a tree--valued Markov chain $(T_n ,n\geq 1)$. $\mathcal{D}$ restricted to $[n]$ corresponds to $T_n^+$ via Lemma \ref{lemma:den_to_tree} where the $+$ signifies that we have added leaf labels. Remove the planar order from $T_n$ but keep the leaf labels and call the resulting tree $T_n^{\text{unordered}}$.
Similarly to Lemma \ref{lemma:den_to_tree}, $T_n^{\text{unordered}}$ corresponds to $([n], \sim, \preceq)$. 

Further, we can consider $\bfS$ restricted to the set spanned by $\{(\{1\},0),\ldots,(\{n\},0), r\}$. We call the tree structure of this set -- as combinatorial, leaf--labelled tree -- $S_n$. Write $(i,j)\approx_n (k,\ell)$ if the most recent common ancestor of $i$ and $j$ is also the most recent common ancestor of $k$ and $\ell$ in $S_n$. 

\begin{lemma}\label{lemma:t_unordered}
As leaf--labelled, non--plane trees, we have $S_n = T_n^{\text{unordered}}$ almost surely. 
\end{lemma}

\begin{proof}
Recall by the bijection between dendritic systems (with the planarity function removed) and discrete trees, $T_n^{\text{unordered}}$ corresponds to $([n], \sim, \preceq)$ and that $S_n$ corresponds to $([n], \approx_n, <)$ where $<$ is induced by the ancestral order on $\bfS$. Hence we check that $([n], \sim, \prec)$ and $([n], \approx_n, <)$ have the same distribution. 

Fix distinct $i,j,k \in [n]$. Observe that $(i,k) \approx_n (j,k) < (i,j)$ if and only if $d(i,j)< d(i,k) = d(j,k)$, and $(i,j) \approx_n (i,k) \approx_n (j,k)$ if and only if $d(i,j)=d(i,k)=d(j,k)$. This holds because $S_n$ is derived from $\bfS$.

We will prove that $(i,k)\sim(j,k) \prec (i,j)$ if and only if $d(i,j)< d(i,k) = d(j,k)$ as well as $(i,j) \sim (i,k) \sim (j,k)$ if and only if $d(i,j)=d(i,k)=d(j,k)$.
Because trees are uniquely determined by the relationship between triples of vertices, this will imply that $S_n=T_n^{\text{unordered}}$. To show this, we use ideas of \cite[Lemma 6.2]{evans_doob-martin_2017} which are similar to ideas used in the above proof of Lemma \ref{lemma:ultrametric}.

Note that $(i,k) \sim (j,k)$ if and only if we do not have $(i,k) \prec (j,k)$ or $(j,k) \prec (i,k)$. Similarly, $d(i,k)=d(j,k)$ if and only if we do not have $d(i,k) < d(j,k)$ or $d(j,k)< d(i,k)$. By contraposition it thus suffices to show that $(j,k) \prec (i,j) $ if and only if $d(j,k)> d(i,j)$. 

On the one hand, $d(j,k) > d(i,j)$ implies $(j,k) \prec (i,j)$ in $\mathcal{D}$ and this does not change when we restrict $\mathcal{D}$ to $[n]$. On the other hand, $(j,k) \prec (i,j)$ clearly implies $d(j,k) \geq d(i,j)$. The crucial part is to show that $(j,k) \prec (i,j)$ implies $d(j,k) \neq d(i,j)$.

Similar to Lemma \ref{lemma:ultrametric}, by exchangeability and de--Finetti's theorem we almost surely have 
\begin{equation*}
    \big\{(j,k) \prec (i,j)\big\} \cap \big\{d(j,k)=d(i,j) \big\} = \big\{(j,k) \prec (i,j) \big\} \cap \big\{ \nexists \ell \in \NN \backslash\{k\}: (j,k) \prec \ell, (i,j) \nprec k \big\}.
\end{equation*}

We now want to estimate the probability of the latter event for the tree $T_m^{\text{unordered}}$ where $m > n$. Here, the event corresponds to the vertices $(j,k)$ and $(i,j)$ being connected by a single edge and further $(j,k)$ has only one other offspring, namely the leaf labelled $k$. Call this event $\mathcal{A}$.

We now proceed similarly to Lemma \ref{lemma:ultrametric}. Recall that the leaves are labelled exchangeably. This means that we can relabel the leaves of $T_n^{\text{unordered}}$ uniformly without changing the distribution. We do this, and condition on $(j,k) \prec (i,j)$. The conditional probability of $\mathcal A$ then is $0$ if $(j,k)$ has more than two children or if the child of $(j,k)$ which is not an ancestor of $(i,j)$ has children on its own. If this is not the case, i.e.\ when the only child of $(j,k)$ besides $(i,j)$ is a single leaf, then the probability that this leaf is labelled $k$ is $\frac{1}{m-2}$ which converges to $0$ as $m\to \infty$. This yields
$$\PP \big( \big\{(j,k) \prec (i,j)\big\} \cap \big\{d(j,k)=d(i,j) \big\} \big)=0.$$
Hence $(j,k) \prec (i,j)$ implies $d(j,k) \neq d(i,j)$.
%\textcolor{gray}{[D: the difference to binary trees here is that we have \emph{more} cases where the conditional probability is $0$. Namely there are two cases: $1)$ $(j,k)$ has degree $2$ and the child which is not $(i,j)$ has degree $\geq 2$. $2)$ $(j,k)$ has degree $\geq 3$. For binary trees, we can the first case.]}
\end{proof}

The considerations of this section lead us to the following statement, 
an analogue of \cite[Proposition 6.3]{evans_doob-martin_2017}. 
This proposition states that we can encode all information contained in $\mathcal{D}$ except for the planar order in the real tree $\bfS$. 

\begin{prop}\label{prop:there_is_S}
There is a rooted real tree $(\bfS, d, r)$ and an injective map $\iota:\mathcal{D} \to (\bfS, d, r)$ such that the tree structure of the span of $\iota([n])$ and $r$ is $T_n^{\text{unordered}}$ for all $n\in \NN$ as combinatorial trees.
Further, the ancestral order $\preceq$ of $\mathcal D$ coincides with the ancestral order partial order $<$ of $\bfS$. 
\end{prop}

\begin{proof}
    For each $k\in \NN$, we set $\iota(k):= (\{k\},0)\in \bfS$. By construction, $S_n$ is the tree structure of the span of $\iota([n])$ and $r$. By Lemma \ref{lemma:t_unordered}, $T_n^{\text{unordered}}$ and $S_n$ are almost surely the same tree.
\end{proof}

Note that $\bfS$ is random because $\mathcal D$ is random. In the next section we study the distribution of $\bfS$.

\subsection{From the coalescent tree to a sampling from a real tree}

We have encoded the non--planar part of the dendritic system in an ultra--metric $d$ which we represented as a real tree. This coalescent tree is random and we want to find a more algorithmic representation for it. This is done by applying Gufler \cite[Theorem 3.9]{gufler_representation_2018}, we state this theorem in our notation. This theorem states that we can represent $\bfS$ by attaching some extra branches to a deterministic real tree $\bfT$.

We formalise this construction: assume we are given a rooted real tree $(\bfT,d_\bfT,r)$ and a probability measure $m$ on $\bfT \times \RR_+$. Consider a sequence of $i.i.d.$ samples $(\xi_i,v_i)_{i\in \NN}$ where $(\xi_i,v_i)$ is distributed according to $m$ for every $i$. Define 
$$\delta(i,j)=\big(d_\bfT (\xi_i,\xi_j)+v_i+v_j \big)\II_{i \neq j},$$
which can be shown to be a pseudo--metric. The idea behind this construction is that we attach a leaf labelled $i$ with a branch of length $v_i$ to the point $\xi_i$ in $\bfT$. $\delta$ is then the path metric on $\bfT$ plus the lengths of the branches. 

Let $(\bfS, d_\bfS, r)$ be the real tree of Proposition \ref{prop:there_is_S}. Let $d$ be the induced ultra--metric on $\NN$ by restricting $\bfS$ to $\iota(\NN)$. This is the ultra--metric which we used to construct $\bfS$. 

\begin{theorem}\cite[Theorem 3.9]{gufler_representation_2018}\label{thm:gufler}
There exists a rooted real tree $(\bfT,d_\bfT,r)$ and a probability measure $m$ on $\bfT \times \RR_+$ such that $\delta = d$ where $d$ is the ultra--metric of the coalescent and $\delta$ is constructed as above. Moreover, under the assumption of ergodicity on $d$, $(\bfT,d_\bfT,r)$ and $m$ are deterministic.
\end{theorem}

Denote the marginal distribution of $m$ on $\bfT$ by $\mu$. In fact, we can choose $\bfT$ to be the span of $\{\pi(\iota(i))\}_{i\in \NN}$ where $\pi$ is the map that maps isolated leaves, i.e.\ leaves that are not accumulation points of other leaves, to the closest branchpoint. If $x$ is a leaf but not isolated, we set $\pi(x) = x$.

% Alternatively, this means that $\bfS$ equals $\bfT$ with additional leaves attached: we attach a leaf of length $v_i$ to $\xi_i$ for every $i$. In our setting, $v_i$ is simply a function of $\xi_i$ as we need to have $d_\bfT(r,\xi_i) + v_i = 1$ due to $\bfS$ being an ultra--metric tree. 

Let us comment on how this theorem is proved. Define $\bfT$ to be closure of the smallest subtree of $\bfS$ that contains $(\pi(y_i))_{i \geq 1}$. Next, let $m_n = \sum_{i=1}^n \delta_{(\pi(y_i), d_\bfS(y_i,\pi(y_i)))}$ the empirical measure on $\pi(y_i)$ with the associated leaf lengths. $m$ then is the weak limit of $m_n$ as $n\to \infty$. Gufler's proof makes use of exchangeability and de Finetti--style theorems which are used to show that the weak limit exists as well as that ergodicity implies that $\bfT$ is deterministic. 

Recall the properties $(C1)$--$(C4)$ of dendritic systems from Definition \ref{def:planar_den_system}.
At this point we have successfully encoded $(C1)$--$(C4)$ in a real tree $(\bfT, r)$ with associated probability measure $\mu$, the distribution $m$ does not matter if we only want to retrieve the dendritic system. Let us explain how to obtain $(\NN, \sim, \preceq)$ of our dendritic system from $(\bfT,r, \mu)$. Given $(\bfT, r, \mu)$, sample $(\xi_i)_{i \in \NN}$ $i.i.d.$ from $\mu$. We define a random equivalence relation $\sim^*$ and a random ancestral order $\preceq^*$ on $\NN \times \NN$. 

\begin{enumerate}
    \item $(i,i)\sim^* (k,l)$ if and only if $(i,i)=(k,l)$.
    \item $(i,j)\sim^*(k,l)$ for $i\neq j,k\neq l$ if and only if $[r, \xi_i]\cap [r, \xi_j]=[r, \xi_l]\cap [r, \xi_k]$.
    \item The partial order $\preceq^*$ is inherited from the partial order on $S$ and adding $(i,j)\preceq^*(i,i)$ for $i \neq j$. This means for distinct $i,j,k,\ell \in \NN$ we have
    $$(k,\ell)\prec^* (i,j) \quad \text{   if  }  \quad \llbracket r,\xi_k\rrbracket\cap [r,\xi_\ell] \subsetneq \llbracket r,\xi_i\rrbracket \cap \llbracket r,\xi_j\rrbracket.$$
\end{enumerate}

\begin{prop}\label{prop:first_part_of_prop_evans}
The above--defined random relations $(\NN,\sim^*, \preceq^*)$ have the same distribution as $(\NN,\sim, \preceq)$ of the dendritic system $\mathcal D = (\NN,\sim, \preceq,p)$.
\end{prop}

\begin{proof}
    This is a combination of Theorem \ref{thm:gufler} and Proposition \ref{prop:there_is_S}.
\end{proof}

This means we have almost proven Proposition \ref{prop:evans}, except for the representation of the planarity function $p$. We will do this in the next section. 

\subsection{Encoding the planar structure}\label{section:encoding_p}

We complete the proof of Proposition \ref{prop:evans} by encoding the planar structure, i.e.\ the planarity function $p$.

\begin{proof}[Proof of Proposition \ref{prop:evans}.]

Recall Proposition \ref{prop:there_is_S} and let $\pi$ be the map that maps a leaf of $\bfS$ to the closest branchpoint. Set $\xi_i=\pi(\iota(i))$. By Theorem \ref{thm:gufler} and Proposition \ref{prop:first_part_of_prop_evans}, $(\xi_i)_{i \in \NN}$ is an exchangeable sequence of $\mu$--distributed random variables on $\bfT$.

Consider the array 
$$\{(\xi_i,\xi_j,p(i,j)) \}_{i,j\in \NN, i \neq j}.$$ 
This array takes values in sequences of $\bfT^2 \times \{\pm 1\}$--valued random variables and is jointly exchangeable.

This implies that there is a measurable function $F:(\bfT \times [0,1])^2 \times [0,1] \to \{\pm 1 \}$ with 
\begin{equation}\label{eq:rep3}
    \{\xi_i, \xi_j, p(i,j)) \}_{i,j\in \NN, i \neq j}  \overset{d}{=} \{\xi_i, \xi_j, F(\xi_i, U_i, \xi_j, U_j, U_{ij}) \}_{i,j\in \NN, i \neq j}.
\end{equation}
where $(U_i)_{i \in \NN}, (U_{ij})_{i,j\in \NN, i\neq j}$ are independent uniform random variables on $[0,1]$ with $U_{ij}=U_{ji}$ which are independent of $(\xi_i)_{i\in \NN}$. This is a general result from the Aldous--Hoover--Kallenberg theory for exchangeable arrays that we state and deduce as Lemma \ref{lemma:application_of_Aldous_Hoover_Kallenberg} later. 

The function $F$ satisfies some consistency relations which we will state and prove here. Let $Leb$ be the Lebesgue measure on $[0,1]$. For $\mu$--almost every $x,y,z \in \bfT$ and $Leb$--almost every $u,v,w,a,b,c$ we the following consistency relations.

\begin{enumerate} 
    \item [(F1)] $F(x,u,y,v,a) = -F(y,v,x,u,a)$,
    \item [(F2)] if $F(x,u,y,v,a)=F(y,v,z,w,b)$ then also $F(x,u,z,w,c)=F(x,u,y,v,a)$,
    \item [(F3)] if $\llbracket r,x \rrbracket \cap \llbracket r,y \rrbracket \notin \{\llbracket r,x \rrbracket, \llbracket r,y \rrbracket \}$ and $\llbracket r,y \rrbracket \subsetneq \llbracket r,z \rrbracket$ then $F(x,u,y,v,a)=F(x,u,z,w,b)$,
    \item [(F4)] if $\llbracket r,x \rrbracket \subsetneq \llbracket r,y \rrbracket \subsetneq \llbracket r,z \rrbracket$ then $F(x,u,y,v,a)=F(x,u,z,w,c)$.
\end{enumerate}

Let us prove these claims. Recall the consistency relations (P1)-(P4) of $p$ as specified in Definition \ref{def:planar_den_system}. Let $\xi_i, \xi_j, \xi_k$ be $i.i.d.$ $\mu$--random variables and let $U_i,U_j,U_k,U_{i k},U_{ij}, U_{j k}$ be independent $i.i.d.$ uniform random variables on $[0,1]$.  By Skorokhod's representation theorem, we can work on a probability space where \eqref{eq:rep3} is an almost--sure equality. The statements in the new probability will translate back to the claimed distributional statements claimed above.

\begin{enumerate}
    \item [(F1)] Firstly, by $(P1)$, we have
    $$p(i,j)=-p(i,j) \quad a.s. \implies F(\xi_i, U_i, \xi_j, U_j, U_{ij})=-F(\xi_j, U_j, \xi_i, U_i, U_{ij})\quad a.s..$$

    \item [(F2)] Secondly, consider the event $A_{ijk}=\{F(\xi_i, U_i, \xi_j, U_j, U_{ij}) = F(\xi_j, U_j, \xi_k, U_k, U_{j k}) = 1\}$. By $(P3)$ we have 
    $$\begin{cases}p(i,j) = 1 \\ p(j,k) = 1 \end{cases} \xrightarrow{(P3)}  \quad p(i,k)=1 \quad a.s.\implies\quad  F(\xi_i, U_i, \xi_k, U_k, U_{ik}) = 1 \quad \text{on } A_{ijk}. $$
    By $(P1)$ and $(F1)$, the same works if we replace $1$ by $-1$. 
    
    \item [(F3)] Thirdly, consider the event $B_{ijk} = \{\llbracket r,\xi_i\rrbracket \cap \llbracket r,\xi_j\rrbracket \notin \{\llbracket r,\xi_i\rrbracket, \llbracket r,\xi_j\rrbracket \}$ and $\llbracket r,\xi_j\rrbracket \subsetneq \llbracket r,\xi_k\rrbracket \}$. This implies that in the dendritic system we have $(i,j) \prec (j,k)$. Then on the intersection of $B_{ijk}$ and $\{p(i,j)=1$, $p((i,i),(j,k))=1\}$, we have by $(P4)$ that $p(i,k)=1$ which in turn is equivalent to $F(\xi_i, U_i, \xi_k, U_k, U_{i k}) =1$ on these events. The same works if we replace $1$ by $-1$ by $(P1)$ and $(F1)$. 
    
    \item [(F4)] Fourthly, consider the event $\{\llbracket r,\xi_i\rrbracket \subsetneq \llbracket r,\xi_j\rrbracket \subsetneq \llbracket r,\xi_k\rrbracket\}$. On this event we have that $(i,j) \prec (j,k)$. In this case $(P4)$ states that $p(i,j)=1$ implies that $p((i,i),(j,k))=1$ which in turn implies $p(i,k)=1=F(\xi_i, U_i, \xi_k, U_k, U_{ik})$. The same works if we replace $1$ by $-1$ by $(P1)$ and $(F1)$. 
\end{enumerate}
\end{proof}

Lastly, let us comment on why there is no consistency relation for $F$ which is derived from $(P2)$. To apply $(P2)$, we need two vertices of our dendritic system $x,y\in \mathcal D$ which satisfy $x \prec y$. This will never be the case for leaves of $\mathcal{D}$ and $F$ is only used to determine the planar relation between leaves.

Finally, we prove a lemma that we skipped earlier. 
Assume we work on the probability space $(\Omega, \mathcal{A}, \PP)$.

\begin{lemma}\label{lemma:application_of_Aldous_Hoover_Kallenberg}
Assume we have a jointly exchangeable, ergodic array $\{ \xi_i, \xi_j, \zeta_{ij} \}_{i,j\in \NN; i\neq j}$ of random variables taking values in $S_1 \times S_1 \times S_2$ where $S_1$ and $S_2$ are some Borel spaces. We can enlarge the probability space so that there exists an array $\{U_i,U_{ij} \}_{i,j\in \NN, i< j}$ of $i.i.d.$ uniform $[0,1]$ random variables which is independent of $\{\xi_i\}_{i\in \NN}$. Set $U_{ji}=U_{ij}$ for $i<j$.
We then have
\begin{equation*}
    \{\xi_i, \xi_j, \zeta_{ij} \}_{i,j\in \NN; i\neq j} \overset{d}{=} \{ F(\xi_i, U_i, \xi_j, U_j, U_{ij}) \}_{i,j\in \NN; i\neq j},
\end{equation*}
for some measurable function $F:S_1 \times [0,1] \times S_1 \times [0,1] \times [0,1] \to S_2$. 
\end{lemma}

\begin{proof}
Without loss of generality we can assume that $S_1 = S_2 = [0,1]$. We use the Aldous--Hoover--Kallenberg theory of exchangeable arrays.

The representation theorem \cite[Theorem 7.22]{kallenberg_probabilistic_2005} for arrays of exchangeable random variables yields the existence of a measureable function $G':[0,1]^4 \to [0,1]^3$ such that
\begin{equation}\label{eq:rep}
    \{(\xi_i,\xi_j,\zeta_{ij}) \}_{i,j\in \NN, i \neq j} \overset{d}{=} \{G'(V,V_i,V_j,V_{ij} ) \}_{i,j\in \NN, i \neq j},
\end{equation}
where $V, (V_i)_{i \in \NN}, (V_{ij})_{i,j\in \NN, i< j}$ are independent uniform random variables on $[0,1]$ and we set $V_{ij}=V_{ji}$.

Recall our assumption of ergodicity of the exchangeable array $\{(\xi_i,\xi_j,\zeta_{ij}) \}$. \cite[Lemma 7.35]{kallenberg_probabilistic_2005} now yields that our representation \eqref{eq:rep} does not depend on $V$. More precisely, there is a measurable function $G:[0,1]^3 \to [0,1]^3$ such that
\begin{equation}\label{eq:rep2}
    \{(\xi_i,\xi_j,\zeta_{ij}) \}_{i,j\in \NN, i \neq j}  \overset{d}{=} \{G(V_i,V_j,V_{ij} ) \}_{i,j\in \NN, i \neq j}.
\end{equation}

We can work in a probability space where \eqref{eq:rep2} is a $\PP$--almost sure equality. We now condition on $\{ \xi_i \}_{i \in \NN} = \{ x_i \}_{i\in \NN}$ for some sequence in $[0,1]$. Choose a family of regular conditional distributions $\PP^x$ under which the $\{V_i,V_{ij} \}$ are still all independent of each other, $\{V_{ij}\}$ is still uniformly distributed but $\{V_i\}$ are not necessarily uniformly distributed. For $t\in [0,1]$, consider
$$\Phi_{x_i}(t) = \PP^x(V_i \leq t).$$
Observe that $\Phi_{x_i}(t)$ and $\Phi_{x_i,x_j}(t)$ depend measurably on $x_i$ and $x_j$ for any $i,j$. Enlarge the probability space again so that there is $\{U_i\}_{i\in \NN; i}$, an array of $i.i.d.$ uniform $[0,1]$ random variables. We then have the distributional equality under $\PP^x,$
\begin{equation*}
    \{\zeta_{ij} \}_{i,j\in \NN, i \neq j} \overset{d}{=} \{G_3(\Phi^{-1}_{x_i}(U_i), \Phi^{-1}_{x_j}(U_j), V_{ij} \}_{i,j\in \NN, i \neq j}.
\end{equation*}
Here $G_3$ is the third coordinate of $G$, i.e.\ $G(\cdot) = (G_1(\cdot), G_2(\cdot),G_3(\cdot))\in [0,1]^3$.
This means that there exists a measurable function $F:[0,1]^5 \to [0,1]$ so that 
\begin{equation*}
    \{\zeta_{ij} \}_{i,j\in \NN, i \neq j} \overset{d}{=} \{F(x_i,U_i,x_j,U_j,U_{ij})\}_{i,j\in \NN, i \neq j}.
\end{equation*}
Because we are using the same random variables $(U_i,U_{ij})_{i,j}$ regardless of the choice of $\{ x_i \}_{i\in \NN}$, we have that $\{U_i,U_{ij} \}_{i,j\in \NN, i\neq j}$ is independent of $\{ \xi_i, \xi_j\}_{i,j\in \NN, i\neq j}$. 
\end{proof}

\begin{center}
    \textbf{Acknowledgements}
\end{center}

The author would like to thank his PhD--supervisor Matthias Winkel for many useful discussions and for reading countless drafts. He would also like to acknowledge the support of EPSRC grant EP/W523781/1.

%%%%%%%%%%%%%%%%%%%% BIBLIOGRAPHY %%%%%%%%%%%%%%%%%
\printbibliography[heading=bibintoc]

\end{document}

%% file: configurations.tex
\usepackage[utf8]{inputenc}
\usepackage[T1]{fontenc}
\usepackage[english, ngerman, english]{babel}
\usepackage[pdftex]{graphicx}
\usepackage{latexsym}
\usepackage{amsmath,amssymb,amsthm}
\usepackage{stmaryrd} %for \llbracket
\usepackage{commath}
\usepackage{enumitem}
\usepackage{mathrsfs}
\usepackage[margin=1 in]{geometry}
\usepackage{dsfont}
\usepackage{xcolor}
\usepackage{hyperref}   
\usepackage{todonotes}  % insert [disable] to disable todo notes
\usepackage[doi=false,
            isbn=false,
            url=false,
            eprint=false, 
            maxbibnames=99,
            date=year]{biblatex}   % bib
\usepackage{csquotes}
\usepackage{amsfonts}
\usepackage{glossaries-extra}

\hypersetup{
    colorlinks=true,
    linkcolor=black,
    filecolor=black,      
    urlcolor=black,
    citecolor=black,
    pdftitle={Notes regarding Marchals Tree Growth},
    pdfpagemode=FullScreen,
    }

\renewbibmacro{in:}{}
\AtEveryBibitem{
  \clearlist{language}
}
\graphicspath{ {./} }

\setuptodonotes{color=yellow!50, linecolor=blue!30}

\newtheorem{theorem}{Theorem}
\newtheorem{lemma}[theorem]{Lemma}
\newtheorem{prop}[theorem]{Proposition}
\newtheorem{cor}[theorem]{Corollary}

\theoremstyle{definition}
\newtheorem{definition}[theorem]{Definition}
\newtheorem{remark}[theorem]{Remark}
\newtheorem{example}[theorem]{Example}
\newtheorem{construction}[theorem]{Construction}

\numberwithin{equation}{section} 
\numberwithin{theorem}{section}

\newcommand*{\eps}{\ensuremath{\varepsilon}}            % das bessere epsilon
\newcommand*{\RR}{\ensuremath{\mathbb{R}}}	        	% reelle Zahlen
\newcommand*{\NN}{\ensuremath{\mathbb{N}}}		        % natürliche Zahlen
\newcommand*{\PP}{\ensuremath{\mathbb{P}}}	        	% W-Maß
\newcommand*{\TT}{\ensuremath{\mathbb{T}}}	        	% space of trees
\newcommand*{\II}{\ensuremath{\mathds{1}}}	        	% Indikatorfunktion
\newcommand*{\indep}{\perp \!\!\! \perp}                % independent symbol
\newcommand*{\bfT}{\ensuremath{\textbf{T}}}	        	% bold face T
\newcommand*{\bfS}{\ensuremath{\textbf{S}}}	        	% bold face S
\DeclareMathOperator\supp{supp}	        % support
\DeclareMathOperator\spn{span}	        % span
\newcommand*{\dGP}{\ensuremath{d_{\mathrm{GP}}}}	        % GH distance